\documentclass[12pt]{article}

\usepackage{amssymb}
\usepackage[english]{babel}
\usepackage[utf8]{inputenc}
\usepackage{amsmath}
\usepackage{amsthm}
\usepackage{latexsym}
\usepackage{stmaryrd}
\usepackage{color}
\usepackage{mathtools}
\usepackage{amsfonts}
\usepackage[margin=2cm]{geometry}
\usepackage[mathscr]{euscript}
\usepackage{tikz}
\usepackage[colorlinks=true,allcolors=black]{hyperref}
\usepackage{subcaption}
\usepackage{caption}

\newtheorem{thm}{Theorem}
\newtheorem{proposition}{Proposition}

\newtheorem{rem}{Remark}
\newtheorem{lem}{Lemma}
\newtheorem{claim}{Claim}
\newtheorem{cor}{Corollary}

\newcommand{\R}{\mathbb{R}}
\newcommand{\He}{\mathbb{H}}
\newcommand{\overHe}{\overline{\He}}

\makeatletter
\newcommand*{\rom}[1]{\expandafter\@slowromancap\romannumeral #1@}
\makeatother

\DeclareMathOperator*{\esssup}{ess\,sup}
\DeclareMathOperator*{\essinf}{ess\,inf}

\author{Costante Bellettini\\University College London \and Myles Workman\\University College London}
\title{Embeddedness of Min-Max CMC Hypersurfaces in Manifolds with Positive Ricci Curvature}
\date{}

\begin{document}

\maketitle

\begin{abstract}
    We prove that on a compact Riemannian manifold of dimension $3$ or higher, with positive Ricci curvature, the Allen--Cahn min-max scheme in \cite{BW-inhomogeneous-allen-cahn}, with prescribing function taken to be a non-zero constant $\lambda$, produces an embedded hypersurface of constant mean curvature $\lambda$ ($\lambda$-CMC). More precisely, we prove that the interface arising from said min-max contains no even-multiplicity minimal hypersurface and no quasi-embedded points (both of these occurrences are in principle possible in the conclusions of \cite{BW-inhomogeneous-allen-cahn}). The immediate geometric corollary is the existence (in ambient manifolds as above) of embedded, closed $\lambda$-CMC hypersurfaces (with Morse index $1$) for any prescribed non-zero constant $\lambda$, with the expected singular set when the ambient dimension is $8$ or higher.
\end{abstract}

\begin{thm}\label{thm: Main Theorem}
    For any $\lambda \in \mathbb{R} \setminus \{0\}$, and compact Riemannian manifold $(N,g)$, with positive Ricci curvature and $\text{dim} \, N = n + 1\geq 3$, there exists a smooth, embedded, two-sided hypersurface $M$, with constant mean curvature $\lambda$ ($\lambda$-CMC), and
    \begin{enumerate}
        \item $M$ is closed when $2 \leq n \leq 6$, 
        \item $\overline{M} \setminus M$ consists of finitely many points when $n = 7$, 
        \item $\text{dim}_{\mathcal{H}} \, (\overline{M} \setminus M) \leq n - 7$, when $n \geq 8$.
    \end{enumerate}
\end{thm}

In Theorem \ref{thm: Main Theorem} the emphasis is on the fact that $M$ is embedded: this appears to be a new result. The statement of Theorem \ref{thm: Main Theorem} with embedded replaced by (the weaker notion of) quasi-embedded was on the other hand known, as detailed below (with two methods available). We recall that quasi-embedded means that the hypersurface is a smooth immersion, with any self-intersections being tangential, and with local structure around any point of tangential intersection being that of two embedded disks lying on opposite sides of each other (see \cite[Definition 8]{BW-inhomogeneous-allen-cahn}).

\bigskip 

As it will be important for our arguments, we begin by recalling that the existence result in Theorem \ref{thm: Main Theorem}, with embedded replaced by quasi-embedded, follows from the work of the first author and N. Wickramasekera in \cite{BW-inhomogeneous-allen-cahn}. 
In fact, \cite[Theorem 1.1]{BW-inhomogeneous-allen-cahn} proves the following more general result.
Given a compact Riemannian manifold, $(N,g)$, $\text{dim} \, N \geq 3$ (without any curvature assumptions) and a non-negative Lipschitz function $h:N\to \R$, there exists a quasi-embedded, two-sided $C^2$ hypersurface $M_{h}$ such that, for each $x \in M_{h}$, the scalar mean curvature of $M_{h}$ at $x$ is given by $h(x)$; the singular set $\overline{M_{h}} \setminus M_{h}$ satisfies the dimensional estimates listed in Theorem \ref{thm: Main Theorem}.
The construction of $M_{h}$ is carried out in the Allen--Cahn min-max framework, and serves as a starting point for the present work. 
We briefly recall it here in the case $h=\lambda$ constant, with further details in Section \ref{subsec: Allen-Cahn and CMC preliminaries}.

\bigskip 

Consider a sequence of functions $\{u_{i}\}$ in $W^{1,2}(N)$, where each $u_{i}$ is the solution of the appropriate $\varepsilon_{i}$-scaled inhomogeneous Allen--Cahn equation, with $\varepsilon_{i} \rightarrow 0$.
Assuming uniform energy bound, the works of J. Hutchinson--Y. Tonegawa \cite{Hutchinson-Tonegawa2000ConvergenceOP} and M. R\"{o}ger--Y. Tonegawa \cite{roger-Tonegawa2008convergence} give, in the $\varepsilon_{i}\to 0$ limit, an integral varifold $V$ (a ``limit interface"), with generalised mean curvature $H_{V} \in L^{\infty} (\text{supp} \|V\|)$, along with a Caccioppoli set $E$, with $\partial^{*} E \subset \text{supp} \, \|V\|$, such that, 

\begin{equation*}
    \begin{cases}
        H_{V} (x) = \lambda, \, \theta_{V} (x) = 1, & \mathcal{H}^{n} - a.e. \, x \in \partial^{*} E, \\
        H_{V} (x) = 0, \, \theta_{V} (x) \in 2 \mathbb{Z}_{\geq 1}, & \mathcal{H}^{n} - a.e. \, x \in \text{supp} \, \|V\| \setminus \partial^{*} E.
    \end{cases}
\end{equation*}
In the presence of such a sequence $\{ u_{i} \}$, the two major roadblocks to an existence result for a $\lambda$-CMC are (i) $\partial^{*} E$ may be empty, in which case the limit interface is actually minimal (ii) even if $\partial^{*} E \not= \emptyset$, it may not have sufficient regularity (\cite[Figure 1]{BW-inhomogeneous-allen-cahn} illustrates how lack of regularity could prevent $\partial^{*} E$ from being an admissible candidate).
In \cite{BW-inhomogeneous-allen-cahn} a (first) sequence $u_{i}$ is produced by means of a classical mountain pass lemma; the Morse index of $u_{i}$ is at most $1$ (as a consequence of the fact that the min-max has one parameter). It is moreover shown (see \cite[Remark 6.7]{BW-inhomogeneous-allen-cahn}) that in the case of ambient manifold with positive Ricci curvature (and with $h=\lambda$ constant), occurrence (i) cannot arise, that is, $\partial^{*} E$ is non-trivial when $u_{i}$ is the sequence obtained from the min-max. For arbitrary ambient manifolds, in the event that $u_{i}$ leads to occurrence (i), \cite{BW-inhomogeneous-allen-cahn} implements a gradient flow that yields a (second) sequence $\{v_{i}\}$, for which $\partial^{*} E \not= \emptyset$ and with Morse index $0$. The matter is thus reduced to a regularity question for the limit interface arising from a sequence $u_{i}$ with uniformly bounded Morse index. This index control is used in a key way to obtain regularity (\cite[Theorem 1.2]{BW-inhomogeneous-allen-cahn}), whose proof relies on extensions of Y. Tonegawa's work \cite{Tonegawa-Stable-2005} and Y. Tonegawa--N. Wickramasekera's work \cite{TW-stable-2010}, and crucially on the (non-variational) varifold regularity result \cite[Theorem 9.1]{BW-Stable-PMC} (see also \cite[Theorem 3.2]{BW-inhomogeneous-allen-cahn}). In conclusion, \cite{BW-inhomogeneous-allen-cahn} obtains that $V = V_{\lambda} + V_{0}$, where $\text{supp} \, \|V_{\lambda}\| = \partial E = \overline{M_{\lambda}}$ and $\text{supp} \, \|V_{0}\| = \overline{M_{0}}$; here $M_{\lambda}$ is a two-sided, quasi-embedded $\lambda$-CMC hypersurface, and $M_{0}$ an embedded minimal hypersurface, both satisfying the dimensional estimates listed in Theorem \ref{thm: Main Theorem}.
Furthermore, any intersections between $M_{\lambda}$ and $M_{0}$, and self-intersections of $M_{\lambda}$, are always tangential intersections of $C^{2}$ graphs lying on one side of each other.

\bigskip 

With this as a starting point, our first step in establishing Theorem \ref{thm: Main Theorem} is to show that when $\text{Ric}_{g} > 0$, the one-parameter Allen--Cahn min-max just recalled does not produce any minimal components in the limit interface, i.e.~$V_{0} = 0$. (As mentioned earlier, in this case \cite{BW-inhomogeneous-allen-cahn} establishes already that $V_{\lambda} \not= 0$ for the  $u_{i}$ produced by min-max.)

\begin{thm}\label{thm: Allen-Cahn minmax limit}
    Let $(N,g)$ be a compact Riemannian manifold of dimension $\geq 3$, with positive Ricci curvature, and $\lambda > 0$. 
    The one-parameter Allen--Cahn min-max in \cite{BW-inhomogeneous-allen-cahn}, with prescribing function set to $\lambda$, produces a two-sided $\lambda$-CMC hypersurface and no minimal hypersurface. 
\end{thm}

Theorem \ref{thm: Allen-Cahn minmax limit} is achieved by exhibiting a suitable continuous path, admissible in the min-max construction (which employs paths that are continuous in $W^{1,2}(N)$). 
This path will move through functions that are each modelled on a level set of the signed distance to $M_{\lambda}$. 
The idea is to try to place a 1-dimensional Allen--Cahn profile along the normal direction to a given level set and thus produce a function (a point in the path). 
This might appear problematic due to the presence of points where the level sets are not smoothly embedded in $N$ (which, for example, may be caused by the presence of the singular set $\overline{M_{\lambda}} \setminus M_{\lambda}$, or by the fact that $M_\lambda$ has quasi-embedded points).  
We handle this after observing that all “problematic points” are contained in a closed $n$-rectifiable set. 
The open complement (in $N$) of this $n$-rectifiable set is described (via a diffeomorphism) as an open subset of $\tilde{M} \times \mathbb{R}$, where $\tilde{M}$ is a (abstract) $n$-manifold whose immersion into $N$ gives $M_\lambda$. We will refer to this open subset as the Abstract Cylinder (which is endowed with a metric pulled back from $N$). 
Each level set of the distance function becomes a subset of $\tilde{M} \times \{s\}$, where $s$ is the chosen distance value. 
The sought path is then defined by “sliding” the 1-dimensional Allen--Cahn profiles in the $\mathbb{R}$-direction in the whole cylinder $\tilde{M} \times \mathbb{R}$, then restricting these functions to the Abstract Cylinder, and passing them to $N$. 
We check that this indeed produces a continuous path in $W^{1,2} (N)$. 
Furthermore, performing the energy calculations on the Abstract Cylinder, we see that the potentially ``problematic points" do not cause any issues.
The sliding argument yields a path with the (key) property that the relevant Allen-Cahn energy attains a maximum (along the path) at the function obtained in correspondence of $M_\lambda$ (signed distance equal to $0$); this relies on the positivity of the Ricci curvature. 
This property of the path easily implies that $V_{0}= 0$ (no minimal component), for otherwise the min-max characterisation of $V$ would be contradicted.
Theorem \ref{thm: Main Theorem} is then proven by showing that the $\lambda$-CMC hypersurface arising in Theorem \ref{thm: Allen-Cahn minmax limit} is, in fact, embedded.
This is again done by exhibiting a suitable path (admissible in the min-max). This path is constructed by editing the previous one about its maximum, under the contradiction assumption that a non-embedded point exists in $M_\lambda$.
The modification requires the identification of suitable hypersurfaces obtained by deforming $M_\lambda$ about the non-embedded point. 
This construction ensures that the modified path attains a maximum that is strictly smaller than the maximum obtained for the path used in the proof of Theorem \ref{thm: Allen-Cahn minmax limit}. This contradicts the min-max characterisation.
We stress that these path constructions capitalise on the a priori knowledge (from \cite{BW-inhomogeneous-allen-cahn}) that $M_\lambda$ and $M_0$ are sufficiently regular.

\bigskip 

We remark that Theorem \ref{thm: Allen-Cahn minmax limit} is somewhat interesting in its own sake: it is an open question whether (and under what assumptions) a sequence of solutions to the inhomogeneous Allen--Cahn equation with nowhere vanishing inhomogeneous term, and with a uniform bound on the Morse index, can produce minimal components. 
(The regularity result in \cite{BW-inhomogeneous-allen-cahn} recalled earlier allows us to refer to the minimal and prescribed-mean-curvature components as hypersurfaces that are separately smooth, except for a possible small singular set when the ambient dimension is $8$ or higher.)
Theorem \ref{thm: Allen-Cahn minmax limit} rules out minimal components in the special instance in which the solutions come from a one-parameter min-max (in $N$ compact with $\text{Ric}_N>0$) and the inhomogeneous term is constant. 

\bigskip 

The absence of minimal components and of non-embedded points established by Theorem \ref{thm: Main Theorem} has, among its consequences, a Morse index estimate: 

\begin{cor}
\label{cor:index}
The $\lambda$-CMC hypersurface in Theorem \ref{thm: Allen-Cahn minmax limit} has Morse index equal to 1.
\end{cor}

This follows directly from C. Mantoulidis \cite{Mantoulidis_2022}. Alternatively, the arguments of F. Hiesmayr \cite{Fritz-index-paper} apply verbatim. (We refer to Section \ref{sec: index} for the definition of Morse index.)

\bigskip 

As we recalled, \cite{BW-inhomogeneous-allen-cahn} employs an Allen--Cahn approximation scheme to construct the $\lambda$-CMC quasi-embedded hypersurface. 
The statement of Theorem \ref{thm: Main Theorem} with embedded replaced by quasi-embedded can also be obtained (without any curvature assumption on $N$) using the so-called Almgren--Pitts method for the min-max, see the combined works of X. Zhou--J. Zhu \cite{ZhouZhuCMC} ($2 \leq n \leq 6$) and A. Dey in \cite{DeyCMC} (for $n\geq 7$, relying on the compactness theory in \cite{BW-Stable-CMC, BW-Stable-PMC}).

\bigskip 

Regardless of the method used for the min-max construction, and without the need of curvature assumptions, if $2 \leq n \leq 6$ the $\lambda$-CMC hypersurface obtained is closed and immersed (completely smooth). 
In B. White's work \cite[Theorem 35]{White-Generic-Transversality-of-minimal-submanifolds} it is proven that for each $\lambda \in \mathbb{R}$, there exists a generic set (in the sense of Baire category) of smooth metrics on the ambient manifold such that any closed, codimension-1 (completely smooth) immersion with constant mean curvature $\lambda$, is self-transverse. 
Therefore, combining the existence of quasi-embedded $\lambda$-CMC (\cite{BW-inhomogeneous-allen-cahn} or \cite{ZhouZhuCMC}) with \cite[Theorem 35]{White-Generic-Transversality-of-minimal-submanifolds}, one obtains:
when $2\leq n \leq 6$, for any $\lambda$, there exists a generic set of metrics on $N$, such that each admits an embedded $\lambda$-CMC hypersurface.\footnote{A more general version of this statement is available for $h$-PMC hypersurfaces, $2 \leq n \leq 6$, by again combining \cite[Theorem 35]{White-Generic-Transversality-of-minimal-submanifolds} with either \cite{BW-inhomogeneous-allen-cahn} or \cite{ZhouZhuPMC}. 
Note that the class of prescribing functions, $h$, is different in these two results.}

\bigskip 

This argument relies however on the complete smoothness of the $\lambda$-CMC hypersurface, which is not available for $n\geq 7$ in the existence results. The flavour of Theorem \ref{thm: Main Theorem} differs from the statement just given in that it allows a singular set and can handle all dimensions; moreover the class of metrics (Ricci positive metrics) is the same for all $\lambda \in \mathbb{R}$. 
We also stress that the proof of embeddedness in Theorem \ref{thm: Main Theorem} exploits the min-max characterisation of the $\lambda$-CMC, while one can apply \cite[Theorem 35]{White-Generic-Transversality-of-minimal-submanifolds} to any smooth CMC immersion, not necessarily one coming from a min-max.
Theorem \ref{thm: Main Theorem} and \ref{thm: Allen-Cahn minmax limit} may also hold with other assumptions on the metric on $N$, or other choices on the set of prescribing functions. 
(In these different scenarios an alternative approach to the sliding argument mentioned above could be a gradient flow, for example, along the lines of \cite[Section 5.4]{Bellettini-multiplicity-one-bumpy-metric} and \cite[Section 6.9]{BW-inhomogeneous-allen-cahn}.)

\subsection*{Acknowledgments}

C.B.'s research was partially supported by the Engineering and Physical Sciences Research Council [EP/S005641/1]. 

\bigskip 

M.W. was supported by the Engineering and Physical Sciences Research Council [EP/S021590/1]. 
The EPSRC Centre for Doctoral Training in Geometry and Number Theory (The London School of Geometry and Number Theory), University College London.
He would like to thank Kobe Marshall-Stevens for many helpful and insightful discussions.

\tableofcontents

\section{Preliminaries}

\subsection{Allen-Cahn and Construction of CMC Immersion}\label{subsec: Allen-Cahn and CMC preliminaries}

We recall the min-max construction in \cite{BW-inhomogeneous-allen-cahn}, of critical points to the inhomogeneous Allen--Cahn energy,
\begin{equation}\label{eqn: inhomo Allen Cahn energy}
    \mathcal{F}_{\varepsilon, \lambda} (u) = \int_{N} \frac{\varepsilon}{2}|\nabla u|^2 + \frac{W(u)}{\varepsilon} - \sigma \int_{N} \lambda u, \hspace{1cm} \varepsilon \in (0,1), \, u \in W^{1,2} (N).
\end{equation}
Where $W$ is a smooth function on $\mathbb{R}$, with $W (\pm 1) = 0$ being non-degenerate minima, and $W(t) > 0$, for $t \in \mathbb{R} \setminus \{\pm 1\}$. 
Furthermore, we impose that $W$ has only three critical points, $t = 0, \, \pm 1$, and quadratic growth outside some compact interval. 
For example $W(t) = (1 - t^{2})^{2} / 4$, for $t \in [-2,2]$ and has linear growth outside $[-3,3]$.
The constant $\sigma$ is given by, 
\begin{equation*}
    \sigma = \int_{-1}^{1} \sqrt{W(s) / 2} \, ds.
\end{equation*}
Moreover, we take $\lambda > 0$.

\bigskip 

Consider the first and second variations of (\ref{eqn: inhomo Allen Cahn energy}) with respect to $\varphi \in C^{\infty} (N)$,
\begin{eqnarray}
    \delta \mathcal{F}_{\varepsilon, \lambda}(u)(\varphi) &=& \int_{N} \varepsilon \nabla u \cdot \nabla \varphi + \left( \frac{W'(u)}{\varepsilon} - \sigma \lambda \right) \varphi, \label{eqn first variation of F} \\
    \delta^{2} \mathcal{F}_{\varepsilon, \lambda}(u)(\varphi, \varphi) &=& \int_{N} \varepsilon |\nabla \varphi|^2 + \frac{W''(u)}{\varepsilon} \varphi^2.
\end{eqnarray}
We say that $u$ is a critical point of (\ref{eqn: inhomo Allen Cahn energy}), if $\delta \mathcal{F}_{\varepsilon, \lambda} (u) (\varphi) = 0$, for all $\varphi \in C^{\infty} (N)$, and then by standard elliptic theory we have that $u \in C^{\infty} (N)$, and strongly solves, 
\begin{equation}\label{eqn: inhomo Allen Cahn equation}
    \varepsilon \Delta u = \frac{W'(u)}{\varepsilon} - \sigma \lambda.
\end{equation}
If $\delta^{2} \mathcal{F}_{\varepsilon, \lambda} (u) (\varphi, \varphi) \geq 0$, for all $\varphi \in C^{\infty} (N)$, then we say that $u$ is a stable solution to (\ref{eqn: inhomo Allen Cahn equation}). 
By Figure \ref{figure on graph of W'}, we see that there exists two stable constant solutions, $a_{\varepsilon} > -1$, and $b_{\varepsilon} > 1$. 
Furthermore, as $\varepsilon \rightarrow 0$, we have that $a_{\varepsilon} \rightarrow - 1$, and $b_{\varepsilon} \rightarrow 1$.
As $\text{Ric}_{g} > 0$, \cite[Proposition 7.1]{bellettini2020multiplicity1} shows that these are the only stable critical points of (\ref{eqn: inhomo Allen Cahn energy}).

\bigskip 

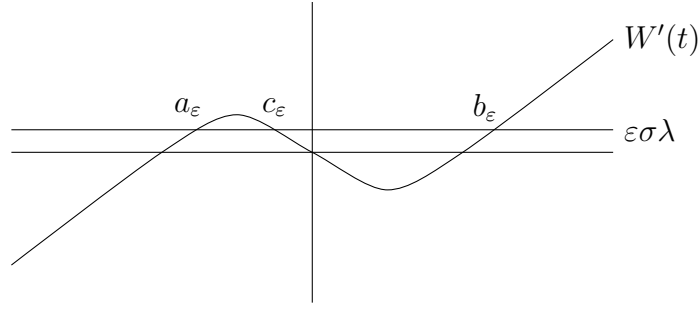
\begin{figure}
    \begin{center}
        \begin{tikzpicture}
            \draw (-4,0) -- (4,0);
            \draw (0,-2) -- (0,2);
            \draw (-4,0.3) -- (4,0.3) node[anchor=west] {$\varepsilon \sigma \lambda$};
            \draw plot [smooth] coordinates {(-4,-1.5) (-2,0) (-1, 0.5) (0,0) (1, -0.5) (2,0) (4,1.5)} node[anchor=west] {$W'(t)$};
            \node at (-1.65, 0.6) {$a_{\varepsilon}$};
            \node at (2.3, 0.6) {$b_{\varepsilon}$};
            \node at (-0.5,0.6) {$c_{\varepsilon}$};
        \end{tikzpicture}
        \caption{Intersection points, $a_{\varepsilon}$, $b_{\varepsilon}$, and $c_{\varepsilon}$, are the solutions to $W'(t) = \varepsilon \sigma \lambda$.}
        \label{figure on graph of W'}
    \end{center}
\end{figure}

\bigskip 

The existence of these isolated, stable solutions permits us to find non-trivial critical points of (\ref{eqn first variation of F}) via a min-max argument. 

\begin{proposition}\label{prop mountain pass solution} (Existence of Min-Max Solution, \cite[Proposition 5.1]{BW-inhomogeneous-allen-cahn})
    For $\varepsilon > 0$, let $\Gamma$ denote the collection of all continuous paths $\gamma \colon [-1,1] \rightarrow W^{1,2}(N)$, such that $\gamma(-1) = a_{\varepsilon}$, and $\gamma(1) = b_{\varepsilon}$. 
    Then there exists an $\varepsilon_0 > 0$, such that for all $\varepsilon < \varepsilon_{0}$, 
    \begin{equation}
        \inf_{\gamma \in \Gamma} \sup_{u \in \gamma([-1,1])} \mathcal{F}_{\varepsilon,\lambda} = \beta_{\varepsilon} > \mathcal{F}_{\varepsilon, \lambda} (a_{\varepsilon}) > \mathcal{F}_{\varepsilon, \lambda} (b_{\varepsilon}),
    \end{equation}
    is a critical value, i.e. there exists $u_{\varepsilon} \in W^{1,2}(N)$, critical point of $\mathcal{F}_{\varepsilon, \lambda}$, with $\mathcal{F}_{\varepsilon, \lambda} (u_{\varepsilon}) = \beta_{\varepsilon}$; moreover, $u_{\varepsilon}$ has Morse index $\leq 1$.
\end{proposition}

In our Ricci positive setting, as $a_{\varepsilon}$ and $b_{\varepsilon}$ are the only stable critical points we actually have that $u_{\varepsilon}$ has Morse index equal to 1.

\bigskip 

Now taking a sequence $\{ \varepsilon_{i} \}_{i \in \mathbb{N}} \subset (0, \varepsilon_{0})$, with $\varepsilon_{i} \rightarrow 0$, and associated critical points from Proposition \ref{prop mountain pass solution}, $\{ u_{i} = u_{\varepsilon_{i}} \}$, we associate the following Radon measures, 
\begin{equation}
    \mu_{i} \coloneqq (2 \sigma)^{-1} \left( \frac{\varepsilon_{i}}{2} |\nabla u_{i}|^{2} + \frac{W(u_{i})}{\varepsilon_{i}} \right) \, d \mu_{g}.
\end{equation}
Where $\mu_{g}$ is the volume measure of $(N, g)$.
Moreover, there exists constants $K$, $L > 0$, such that for all $i$,
\begin{equation}\label{eqn: upper bound for energy for minmax solutions}
    \sup_{N} |u_{i}| + \mu_{i} (N) \leq K, 
\end{equation}
and 
\begin{equation}\label{eqn: lower bound on energy for minmax solutions}
    \mu_{i} (N) \geq L.
\end{equation}

\bigskip 

By the bounds of (\ref{eqn: upper bound for energy for minmax solutions}) and (\ref{eqn: lower bound on energy for minmax solutions}), there exists a subsequence $\{ u_{i'} \} \subset \{ u_{i} \}$, along with a $u_{0} \in BV(N)$, with $u_{0} (y) \in \{ +1, -1\}$ for all $y \in N$, and a non-zero Radon measure $\mu$, such that $u_{i'} \rightarrow u_{0}$ in $L^{1} (N)$, and $\mu_{i'} \rightharpoonup \mu$ as Radon measures.
By \cite[Theorem 1]{Hutchinson-Tonegawa2000ConvergenceOP} and \cite[Theorem 3.2]{roger-Tonegawa2008convergence}, we have that $\mu$ is the weight measure of an integral $n$-varifold $V$, with the following properties:
\begin{enumerate}
    \item \label{pt 1 prop on initial properties of limit varifold, no regularity} $V$, is an integral $n$-varifold with bounded generalised mean curvature $H_{V}$, and first variation $\delta V = - H_{V} \mu_{V}$.
    \item \label{pt 2 prop on initial properties of limit varifold, no regularity} The set $E \coloneqq \{ u_{0} = + 1\}$ is a Caccioppoli set, with reduced boundary $\partial^{*} E \subseteq \text{spt } \, V \subset N \setminus E \not= \emptyset$. 
    \item \label{pt 4 prop on initial properties of limit varifold, no regularity}  For $\mathcal{H}^{n}$-a.e. $x \in \partial^{*}E$, $\Theta(\mu_{V}, x) = 1$, and $H_{V}(x) \cdot \nu(x) = \lambda$; where $\nu$ is the inward pointing unit normal to $\partial^{*} E$, i.e. $\nu = \nabla u_0 / |\nabla u_0|$. 
    \item \label{pt 5 prop on initial properties of limit varifold, no regularity} For $\mathcal{H}^{n}$ a.e. $x \in \text{spt} \, V  \setminus \partial^{*} E$, $\Theta (\mu_{V} , x)$ is an even integer $\geq 2$, and $H_{V} (x) = 0$.
  \end{enumerate}

\bigskip 

Optimal regularity of $V$ was then proven in \cite{BW-inhomogeneous-allen-cahn}.
\begin{enumerate}
    \item $V = V_{0} + V_{\lambda}$
    \item $V_{0}$ is a (possibly zero) stationary integral $n$-varifold with singular set of Hausdorff dimension at most $n - 7$. 
    \item $V_{\lambda} = |\partial^{*} E| \not= \emptyset$, and $\partial^{*} E$ is a quasi-embedded hypersurface with constant mean curvature $\lambda$, with respect to unit normal pointing into $E$. 
    The singular set of $\partial^{*} E$ has Hausdorff dimension at most $n - 7$.
    \item $V$ has a $(\lambda, 0)$-CMC structure.
\end{enumerate}

\bigskip 

By $(\lambda, 0)$-CMC structure we mean that for each point on the support of $V$, potentially away from a closed set of Hausdorff dimension at most $n - 7$, the local picture is one of the following, 
\begin{enumerate}
    \item \label{item: single embedded CMC disk} There is a single embedded $\lambda$-CMC disk.
    \item There are two embedded $\lambda$-CMC disks that lie on either side of each other and only touch tangentially.
    \item There is a single embedded minimal disk 
    \item There is a single embedded $\lambda$-CMC disk and a single embedded minimal disk that only touch tangentially. 
    \item \label{item: two CMC disks and a minimal disk} There are two embedded $\lambda$-CMC disks that lie on either side of each other, along with an embedded minimal disk, such that all three disks only touch tangentially.
\end{enumerate}
For a detailed definition of a $(\lambda, 0)$-CMC structure, see \cite[Definition 8]{BW-inhomogeneous-allen-cahn}.
We define the set $\text{gen-reg} \, V$, to be the set of points on $\text{supp} \, \|V\|$, which satisfy one of the local pictures of \ref{item: single embedded CMC disk} to \ref{item: two CMC disks and a minimal disk}.
For a detailed definition of $\text{gen-reg} \, V$ see \cite[Definition 5]{BW-inhomogeneous-allen-cahn}.

\bigskip 

Therefore, we have the following

\begin{thm}\label{thm existence of CMC hypersurface}
    (Theorem 1.1 \cite{BW-inhomogeneous-allen-cahn})
    Let $N$  be a closed Riemannian manifold of dimension $n+1 \geq 3$, with positive Ricci curvature, and let $\lambda \in (0, \infty)$ be a fixed constant.
    There exists a smooth, quasi-embedded hypersurface $M \subset N$, with;
    \begin{enumerate}
        \item $\overline{M} \setminus M = \emptyset$, if $2 \leq n \leq 6$;
        \item $\overline{M} \setminus M$ is finite if $n = 7$;
        \item $\text{dim}_{\mathcal{H}}(\overline{M} \setminus M) \leq n - 7$, if $n \geq 8$. 
    \end{enumerate}
    Moreover, $M$ is the image of a two-sided immersion with mean curvature $H_{M} = \lambda \nu$, for a choice $\nu$ of continuous unit normal to the immersion.
\end{thm}

We restate Theorems \ref{thm: Main Theorem} and \ref{thm: Allen-Cahn minmax limit} with our new notation. 

\begin{thm}\label{thm: restatement of Allen-Cahn minmax limit}
    Consider a closed Riemannian manifold $(N, g)$, with positive Ricci curvature and $\text{dim} \, N = n + 1 \geq 3$.
    Take $\lambda \in (0, \infty)$.
    The limiting varifold $V = V_{\lambda} + V_{0}$ from Section \ref{subsec: Allen-Cahn and CMC preliminaries} has the following properties
    \begin{enumerate}
        \item $M \coloneqq \text{gen-reg} \, V_{\lambda}$ is embedded, connected and has index 1. 
        \item $V_{0} = 0$.
    \end{enumerate}
\end{thm}

This says that only case \ref{item: single embedded CMC disk} can occur.

\subsection{One Dimensional Allen--Cahn Solution}\label{subsec: one dim allen cahn solution}

We refer to \cite[Section 2.2]{bellettini2020multiplicity1} as a reference for this section. 

\bigskip 

We define the function $\mathbb{H}$ on $\mathbb{R}$ to denote the monotonically increasing solution to the ODE $u'' - W'(u) = 0$, with the conditions $\mathbb{H} (0) = 0$ and $\lim_{t \rightarrow \pm \infty} \mathbb{H} (t) = \pm 1$.
We then define $\mathbb{H}_{\varepsilon} (\, \cdot \,) = \mathbb{H} (\varepsilon^{-1} \, \cdot)$, which solves the ODE $\varepsilon u'' - \varepsilon^{-1} W'(u) = 0$. 

\bigskip 

We define an approximation for $\mathbb{H}_{\varepsilon}$. 
Start by considering the following bump function 
\begin{equation*}
    \begin{cases}
        \chi \in C_{c}^{\infty} (\mathbb{R}), \\
        \chi (t) = 1, & t \in (-1,1), \\
        \chi (t) = 0, & t \in \mathbb{R} \setminus (-2,2), \\
        \chi(t) = \chi (-t), & t \in \mathbb{R}, \\
        \chi'(t) \leq 0, & t \geq 0. \\
    \end{cases}
\end{equation*}
For $\varepsilon \in (0, 1)$, we define the truncation of $\mathbb{H}_{\varepsilon}$ by 
\begin{equation*}
    \overHe_{\varepsilon} (t) \coloneqq \begin{cases}
        \chi ((\varepsilon \Lambda)^{-1} t) \mathbb{H}_{\varepsilon} (t) + 1 - \chi ((\varepsilon \Lambda)^{-1} t), & t > 0, \\
        \chi ((\varepsilon \Lambda)^{-1} t) \mathbb{H}_{\varepsilon} (t) - 1 + \chi ((\varepsilon \Lambda)^{-1} t), & t < 0,
    \end{cases}
\end{equation*}
where $\Lambda = 3 |\log \varepsilon|$.
There exists a constant $\beta = \beta (W) < + \infty$, such that for all $\varepsilon \in (0,1/4)$, 
\begin{equation*}
    2 \sigma - \beta \varepsilon^{2} < \int_{\mathbb{R}} \frac{\varepsilon}{2} |(\overHe_{\varepsilon})' (t)|^{2} + \frac{W(\overHe_{\varepsilon} (t))}{\varepsilon} \, dt < 2 \sigma + \beta \varepsilon^{2}.
\end{equation*}

\section{Idea of Proof} 

We first prove Theorem \ref{thm: Main Theorem} for the case $\lambda > 0$. 
To then prove for $\lambda < 0$, we take $\tilde{\lambda} = - \lambda > 0$, and reverse the direction of the unit normal on the resulting $\tilde{\lambda}$-CMC hypersurface. 
From here on we take $\lambda > 0$.

\bigskip 

For Caccioppoli sets $\Omega \subset N$, we define the following functional, 
\begin{equation*}
    \mathcal{F}_{\lambda} (\Omega) \coloneqq \mathcal{H}^{n} (\partial^{*} \Omega) - \lambda \mu_{g} (\Omega).
\end{equation*}
Recall our converging sequence of critical points $\{u_{\varepsilon_{j}}\}$, along with our limiting varifold $V = V_{\lambda} + V_{0}$, and Caccioppoli set $E$ from Section \ref{subsec: Allen-Cahn and CMC preliminaries}.
We have, as $\varepsilon_{j} \rightarrow 0$, 
\begin{equation*}
    \mathcal{F}_{\varepsilon_{j}, \lambda} (u_{\varepsilon_{j}}) \rightarrow 2 \sigma \mathcal{F}_{\lambda} (E) + 2 \sigma \mathbb{M} (V_{0}) + \sigma \lambda \mu_{g} (N)
\end{equation*}
Therefore, constructing minimising paths between $\emptyset$ and $N$ for $\mathcal{F}_{\lambda}$, may provide insight to minimising paths from $a_{\varepsilon}$ to $b_{\varepsilon}$ for $\mathcal{F}_{\varepsilon, \lambda}$.

\bigskip 

As $N$ is compact, one obvious path that includes $E$, is $\{ E_{t} \}$ for $t \in [-2 \, \text{diam} (N), 2 \, \text{diam} (N)]$, where, 
\begin{equation*}
    E_{t} \coloneqq \{ y \colon \tilde{d} (y) > t\}.
\end{equation*}
Here $\tilde{d}$ is the signed distance function to $M \coloneqq \partial^{*} E$, taking positive values in $E$, and negative values in $N \setminus E$. 
We also denote, 
\begin{equation*}
    \Gamma_{t} \coloneqq \{ y \colon \tilde{d} (y) = t \} = \partial E_{t}.
\end{equation*}
Assuming sufficient regularity on the sets $\Gamma_{t}$ and $E_{t}$, and the functions $t \mapsto \mathcal{H}^{n} (\Gamma_{t})$ and $t \mapsto \mu_{g} (E_{t})$, we have for $t > 0$, 
\begin{equation}\label{eqn: F lambda E t minus F lambda E}
    \begin{split} 
    \mathcal{F}_{\lambda} (E_{t}) - \mathcal{F}_{\lambda} (E) &= \int_{0}^{t} \frac{d}{ds} \mathcal{H}^{n} (\Gamma_{s}) \, ds - \lambda \int_{0}^{t} \frac{d}{ds} \mu_{g} (E_{s}) \, ds, \\
    &= \int_{0}^{t} \int_{\Gamma_{s}} \lambda - H_{\Gamma_{s}} (x) \, d \mathcal{H}^{n} (x) \, ds,
    \end{split}
\end{equation}
where $H_{\Gamma_{s}}$ is the scalar mean curvature of $\Gamma_{s}$ with respect to unit normal $\nabla \tilde{d}$.
Recalling that $H_{\Gamma_{0}} = \lambda$, a straightforward calculation yields the following inequalities.
\begin{equation*}
    \begin{cases}
        H_{\Gamma_{t}} \geq \lambda + m t, & t \geq 0, \\
        H_{\Gamma_{t}} \leq \lambda + m t, & t \leq 0,
    \end{cases} 
\end{equation*}
where $m = \min_{|X| = 1} \, Ric_{g} (X, X) > 0$. 
Therefore, by (\ref{eqn: F lambda E t minus F lambda E}) for $t \geq 0$, 
\begin{equation*}
    \mathcal{F}_{\lambda} (E_{t}) \leq \mathcal{F}_{\lambda} (E).
\end{equation*}
The same inequality holds for $t \leq 0$. 
Here we see the importance of the assumption on the Ricci curvature.
Therefore,
\begin{equation*}
    \gamma \colon t \in [-2 \, \text{diam} (N), 2 \, \text{diam} (N)] \mapsto E_{-t} \in \{ \text{Caccioppoli sets of} \, N \},
\end{equation*}
is a path from $\emptyset$ to $N$, that has maximum height $\mathcal{F}_{\lambda} (E)$. 

\bigskip 

We look to replicate this path in $W^{1,2} (N)$.
Consider the Lipschitz function on $N$, 
\begin{equation*}
    v_{\varepsilon}^{t} = \overHe_{\varepsilon} (\tilde{d} (x) - t),
\end{equation*} 
which can be thought of as placing the truncated one dimensional Allen-Cahn solution from Section \ref{subsec: one dim allen cahn solution} along the normal profile of $\Gamma_{t}$.
By the Co-Area formula, we have, 
\begin{equation*}
    \mathcal{F}_{\varepsilon, \lambda} (v_{\varepsilon}^{t}) = \int_{\mathbb{R}} Q_{\varepsilon} (s - t) \mathcal{H}^{n} (\Gamma_{s}) \, ds - \sigma \lambda \int_{\mathbb{R}} \overHe_{\varepsilon} (s - t) \mathcal{H}^{n} (\Gamma_{s}) \, ds, 
\end{equation*}
were, 
\begin{equation*}
    Q_{\varepsilon} (t) = \frac{\varepsilon}{2} | \left( \overHe_{\varepsilon} \right)' (t) |^{2} + \frac{W\left( \overHe_{\varepsilon} (t) \right)}{\varepsilon}
\end{equation*}
The functions
\begin{equation*}
    t \mapsto \int_{\mathbb{R}} Q_{\varepsilon} (s - t) \mathcal{H}^{n} (\Gamma_{s}) \, ds, \hspace{0.5cm} and \hspace{0.5cm} t \mapsto \sigma \lambda \int_{\mathbb{R}} \overHe_{\varepsilon} (s - t) \mathcal{H}^{n} (\Gamma_{s}) \, ds,
\end{equation*}
act as smooth approximations to $t \mapsto 2 \sigma \mathcal{H}^{n} (\Gamma_{t})$, and $t \mapsto 2 \sigma \lambda \mu_{g} (E_{t}) - \sigma \lambda \mu_{g} (N)$, respectively.

\bigskip 

We say that $v_{\varepsilon}^{0}$ is an Allen--Cahn approximation of $M$ as, 
\begin{equation*}
    \mathcal{F}_{\varepsilon, \lambda} (v_{\varepsilon}^{0}) \rightarrow 2 \sigma \mathcal{H}^{n} (M) - \sigma \lambda \mu_{g} (E) + \sigma \lambda \mu_{g} (N \setminus E) \eqqcolon A_{2},
\end{equation*}
as $\varepsilon \rightarrow 0$, Section \ref{subsec: Approximating function for CMC}.
Carrying out a calculation which replicates the previous one, we deduce that for all $\tau > 0$, there exists an $\varepsilon_{\tau} > 0$, such that for all $\varepsilon \in (0, \varepsilon_{\tau})$, 
\begin{equation*}
    \max_{t \in [- 2 \, \text{diam} (N), 2 \, \text{diam} (N)]}\mathcal{F}_{\varepsilon, \lambda} (v_{\varepsilon}^{t}) < A_{2} + \tau = A_{1} - 2 \sigma \mathbb{M} (V_{0}) + \tau,
\end{equation*}
where $A_{1} \coloneqq 2 \sigma \mathcal{H}^{n} (M) + 2 \sigma \mathbb{M} (V_{0}) - \sigma \mu_{g} (E) + \sigma \mu_{g} (N \setminus E)$.
Connecting $v_{\varepsilon}^{2 \, \text{diam}(N)} = -1$ to $a_{\varepsilon}$, and $v_{\varepsilon}^{- 2 \, \text{diam} (N)} = + 1$ to $b_{\varepsilon}$, by constant functions, we see that we have an appropriate min-max path in $W^{1,2} (N)$.

\bigskip 

This path proves that we cannot have a minimal piece $V_{0}$. 
We also get criterion for $M$. 
Indeed, as there exists a 'Wall', \cite[Lemma 5.1]{BW-inhomogeneous-allen-cahn}, that all min-max paths must climb over, we have that
\begin{equation*}
    2 \sigma \lambda \mathcal{H}^{n} (M) - \sigma \lambda \mu_{g} (E) + \sigma \lambda \mu_{g} (N \setminus E) > \sigma \lambda \mu_{g} (N).
\end{equation*}
Rearranging yields, 
\begin{equation*}
    \mathcal{H}^{n} (M) > \lambda \mu_{g} (E).
\end{equation*}
We note that the above path can be constructed for any suitable $\lambda$-CMC hypersurface which encloses a volume. 
Therefore, for any such pair $(M, E)$, the above inequality holds, and our min-max must choose the pair that minimises the positive quantity $\mathcal{H}^{n} (M) - \lambda \mu_{g} (E)$.
From this, we can deduce that $E$ must be connected.

\bigskip 

We turn our attention to proving that $M$ is embedded.
We prove by contradiction, exploiting the min-max characterisation of $M$.
We now know that, given our sequence of critical points $\{ u_{\varepsilon_{j}} \}$, and potentially after taking a subsequence, 
\begin{equation*}
    \mathcal{F}_{\varepsilon_{j}, \lambda} (u_{\varepsilon_{j}}) \rightarrow 2 \sigma \mathcal{H}^{n} (M) - \sigma \lambda \mu_{g} (E) + \sigma \lambda \mu_{g} (N \setminus E) = A_{2},
\end{equation*}
as $\varepsilon_{j} \rightarrow 0$.
Assume that $M$ has a non-embedded point $z_{0}$. 
Then for every $\varepsilon_{j} > 0$, we construct a continuous path, 
\begin{equation*}
    \gamma_{\varepsilon_{j}} \colon [-1,1] \rightarrow W^{1,2} (N),
\end{equation*}
where, $\gamma_{\varepsilon_{j}} (-1) = a_{\varepsilon_{j}}$, and $\gamma_{\varepsilon_{j}} (1) = b_{\varepsilon_{j}}$. 
This path satisfies the following, there exists a $J$ in $\mathbb{N}$, and $\varsigma > 0$, independent of $j$, such that for all $j \geq J$,  
\begin{equation*}
    \max_{t \in [-1,1]} \mathcal{F}_{\varepsilon_{j}, \lambda} (\gamma_{\varepsilon_{j}} (t)) < 2 \sigma \mathcal{H}^{n} (M) - \sigma \lambda \mu_{g} (E) + \sigma \lambda \mu_{g} (N \setminus E) - \varsigma,
\end{equation*}
This is a contradiction of the min-max characterisation of $u_{\varepsilon_{j}}$. 

\bigskip 

We sketch the main ideas of the path in the $\varepsilon$-limit, Figure \ref{fig: The Path}.

\bigskip 

The picture at $z_{0}$ is Figure \ref{subfig: point 1 of path, Disks}. 
The limiting energy for this structure is $A_{2}$.
The starting point for building this path is to construct a competitor with lower limiting energy. 
Then we wish to connect this competitor to $+1$ and $-1$, with energy always remaining a fixed amount below $A_{2}$. 

\bigskip 

\textbf{\textit{Step 1}:} Construction of Competitor, (1) $\rightarrow$ (2) in Figure \ref{fig: The Path}, Section \ref{sec: Path From Disks to Crab}

\bigskip 

The structure at $z_{0}$ is two smooth, embedded CMC disks, that touch tangentially at $z_{0}$ and lie either side of each other. 
To construct the competitor, we push these disks together, and delete portions of the disks that are pushed past each other. 
This reduces the area of our structure while also increasing the size of $E$, leading to a drop in energy. 

\bigskip 

\textbf{\textit{Idea 1}:} Push the whole of $M$ by some fixed distance $\rho$. 

\bigskip 

This equates to pushing $M$ to the level $\Gamma_{- \rho}$. 
As seen previously, this will lead to a drop in energy. 
Furthermore, there is an obvious path to $+ 1$, namely we keep pushing along level sets, $\Gamma_{- r}$ for $r$ in $[\rho, 2 \, \text{diam} (N)]$.
However, there is no obvious path to $-1$. 
Pushing $\Gamma_{- \rho}$ in the direction of $E$, will increase the energy and bring us back to $M$, undoing the energy drop that the competitor created. 

\bigskip 

\textbf{\textit{Idea 2}:} Push the disks together locally. 

\bigskip 

Consider open balls $B_{1} \subset \subset B_{2}$ about $z_{0}$. 
We smoothly bump the disks at $z_{0}$ such that inside $B_{1}$ we move the disks of distance $\rho > 0$, and outside $B_{2}$ we remain fixed. 
The balls $B_{1}$ and, $B_{2}$ along with $\rho$, are chosen so that the area inside $B_{1}$ gets deleted, Figure \ref{subfig: point 2 of path, Crab}.
Letting, 
\begin{equation*}
    \varsigma = \frac{\sigma}{2} \mathcal{H}^{n} (B_{1} \cap M),
\end{equation*}
we see that our competitor has energy lying below, $A_{2} - \varsigma$.  

\bigskip

\textbf{\textit{Step 2}:} Path to +1, Section \ref{sec: path to plus 1}

\bigskip

To connect to the competitor, $+ 1$ we look to copy the successful path to $+ 1$ of the competitor in \textbf{\textit{Idea 1}}. 
To construct the competitor, we only edited $M$ locally about $z_{0}$. 
Therefore, pushing the competitor to the level set $\Gamma_{- \rho}$ will correspond to a similar drop in energy from pushing $M$ to $\Gamma_{- \rho}$. 
This is (2) $\rightarrow$ (6) in Figure \ref{fig: The Path}.
See Figures \ref{subfig: point 2 of path, Crab} and \ref{subfig: point 6 of path} for local pictures about $z_{0}$.
From $\Gamma_{- \rho}$ we can easily connect to $+ 1$ by pushing along level sets $\Gamma_{- r}$, as previously discussed.

\bigskip

\textbf{\textit{Step 3}:} Path to $-1$, Section \ref{sec: path to minus 1}

\bigskip 

We look to follow a similar method as in \textbf{\textit{Step 2}} by connecting our competitor to a level set $\Gamma_{r_{0}}$, for $r_{0} >0$, then push this along level sets $\Gamma_{r}$ for $r$ in $[r_{0}, 2 \, \text{diam} (N)]$ to connect it to $-1$.
By pushing our competitor straight to $\Gamma_{r_{0}}$ we run the risk of pushing through $M$ and increasing our energy back up to $A_{2}$. 
Therefore, we carry out our path in stages, again making use of the fact that our edit about $z_{0}$ was local. 

\bigskip 

The first stage is (2) $\rightarrow$ (3) in Figure \ref{fig: The Path}.
We fix our competitor in $B_{2}$, and outside we push forward, so that outside some larger ball $B_{3}$, we line up with $\Gamma_{r_{0}}$. 
See Figures \ref{subfig: point 2 of path, Crab} and \ref{subfig: point 3 of path} for local pictures about $z_{0}$.
Again, as our edit is local about $z_{0}$, this corresponds to a similar drop in energy of pushing $M$ to $\Gamma_{r_{0}}$, and the drop will be of order $r_{0}^{2}$. 
For a large enough $r_{0}$ this will give us a large enough energy drop to be able to undo the edit inside $B_{2}$, and still have our energy remain below $A_{2} - \varsigma$.
This is the second stage from (3) $\rightarrow$ (4), in Figure \ref{fig: The Path}.
See Figure \ref{subfig: point 4 of path}, for local picture about non-embedded point.
From here we push up inside $B_{3}$ to line up with $\Gamma_{r_{0}}$, (4) $\rightarrow$ (5) in Figure \ref{fig: The Path}, Figure \ref{subfig: point 5 of path}. 
Finally, we connect to $- 1$ by sliding along level sets as previously stated. 

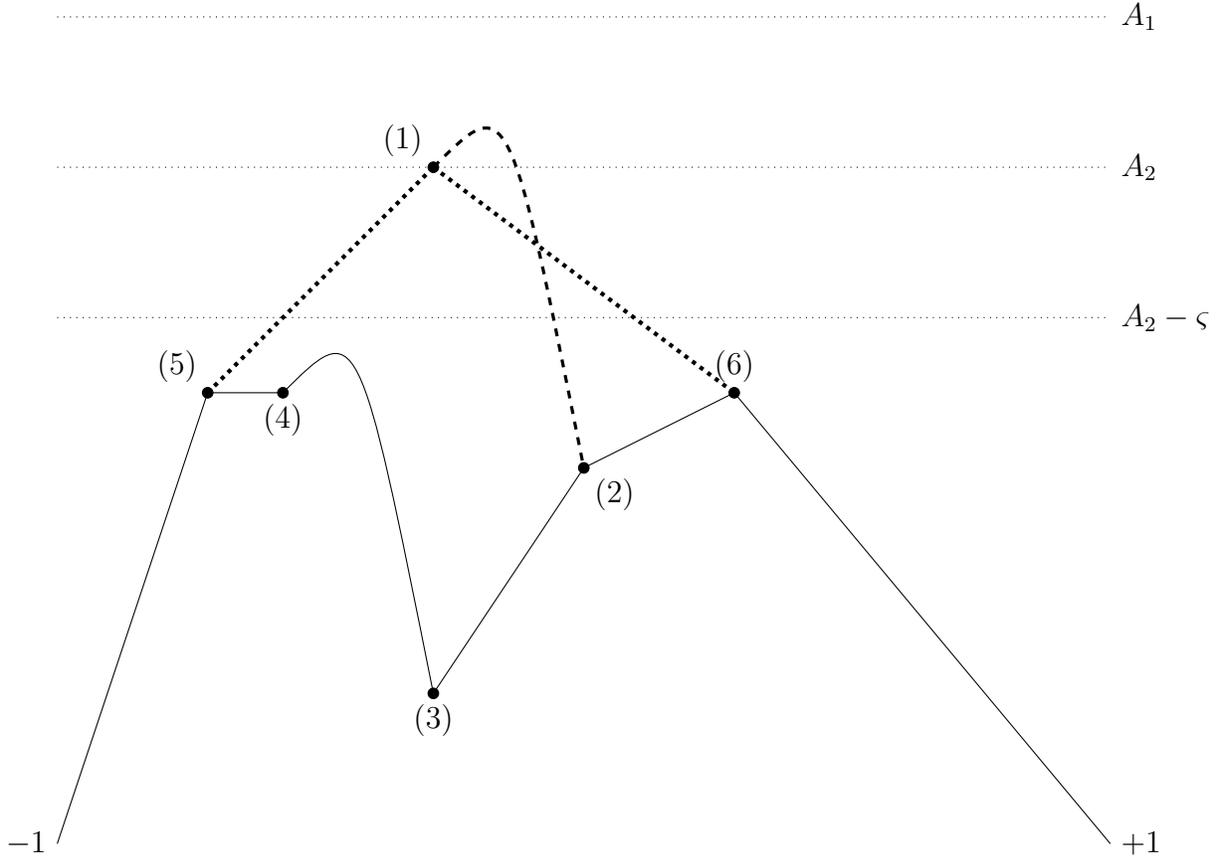
\begin{figure}
    \begin{center}
        \begin{tikzpicture}
            \draw[dotted] (-7,6) -- (7,6) node[anchor=west]{$A_{1}$}; 
            \draw[dotted] (-7,4) -- (7,4) node[anchor=west]{$A_{2}$}; 
            \draw[dotted] (-7,2) -- (7,2) node[anchor=west]{$A_{2} - \varsigma$};
            \filldraw[black] (-2,4) circle (2pt) node[above left]{(1)};
            \draw[dashed, very thick] (-2,4) .. controls (-1,5) .. (0,0);
            \filldraw[black] (0,0) circle (2pt) node[below right]{(2)};
            \draw (0,0) -- (-2,-3); 
            \filldraw[black] (-2,-3) circle (2pt) node[below]{(3)};
            \draw (-2, - 3) .. controls (-3, 2) .. (-4,1);
            \filldraw[black] (-4,1) circle (2pt) node[below]{(4)};
            \draw (-4, 1) -- (-5, 1); 
            \filldraw[black] (-5,1) circle (2pt) node[above left]{(5)};
            \draw (-5,1) -- (-7, -5);
            \node[left] at (-7,-5) {$- 1$};
            \draw (0,0) -- (2,1); 
            \filldraw[black] (2,1) circle (2pt) node[above]{(6)};
            \draw (2,1) -- (7, -5);
            \node[right] at (7, -5) {$+ 1$};
            \draw[dotted, ultra thick] (-5,1) -- (-2, 4); 
            \draw[dotted, ultra thick] (-2,4) -- (2,1);
        \end{tikzpicture}
        \caption{The Paths.
        To prove $V_{0} = 0$, we follow the path from $-1$ to (5), then the dotted line to (1), dotted line to (6), then complete the path to $+1$.
        The dashed line from (1) to (2) is the construction of the competitor.
        Then, to prove that $M$ is embedded, we follow the path from $-1$ to $+1$ given by the solid lines.
        Refer to Figure \ref{fig: Local pictures about non-emebedded point in the Path} for the local picture about the non-embedded point $z_{0}$ at each numbered stage on the paths.}
        \label{fig: The Path}
    \end{center}
    \end{figure}
    
    \bigskip
    
    \begin{figure}
        \begin{center}
            \begin{subfigure}{0.4\textwidth}
                \begin{center}
                {
                \begin{tikzpicture}[scale=1.5]
                    \draw plot [smooth] coordinates {(-2,2) (-1,0.5) (0,0) (1,0.5) (2,2)};
                    \draw plot [smooth] coordinates {(-2,-2) (-1,-0.5) (0,0) (1,-0.5) (2,-2)};
                    \node at (0,1) {$+1$};
                    \node at (0,-1) {$+1$};
                    \node[left] at (-1,0) {$-1$};
                    \node[right] at (1,0) {$-1$};
                \end{tikzpicture}
                }
                \caption{(1): Non-embedded point $z_{0}$}
                \label{subfig: point 1 of path, Disks}
            \end{center}
                
            \end{subfigure}
            \hspace{2cm}
            \begin{subfigure}{0.4\textwidth}
                \begin{center}
                \begin{tikzpicture}[scale=1.5]
                    \draw plot [smooth] coordinates {(-2,2) (-1, 0.5) (-0.5, 0)};
                    \draw plot [smooth] coordinates {(0.5,0) (1,0.5) (2,2)};
                    \draw plot [smooth] coordinates {(-2,-2) (-1,-0.5) (-0.5, 0)}; 
                    \draw plot [smooth] coordinates {(0.5,0) (1,-0.5) (2,-2)};
                    \draw [dashed] plot [smooth] coordinates {(-1,0.5) (0,0) (1,0.5)};
                    \draw [dashed] plot [smooth] coordinates {(-1,-0.5) (0,0) (1,-0.5)};
                    \node at (0,1) {$+1$};
                    \node at (0,-1) {$+1$};
                    \node[left] at (-1,0) {$-1$};
                    \node[right] at (1,0) {$-1$};
                \end{tikzpicture}
                \caption{(2): Competitor}
                \label{subfig: point 2 of path, Crab}
            \end{center}
            \end{subfigure}
    
            \hfill 
    
            \begin{subfigure}{0.4\textwidth}
                \begin{center}
                \begin{tikzpicture}[scale=1.5]
                    \draw plot [smooth] coordinates {(-1.5,2) (-1, 0.5) (-0.5, 0)};
                    \draw plot [smooth] coordinates {(0.5,0) (1,0.5) (1.5,2)};
                    \draw plot [smooth] coordinates {(-1.5,-2) (-1,-0.5) (-0.5, 0)}; 
                    \draw plot [smooth] coordinates {(0.5,0) (1,-0.5) (1.5,-2)};
                    \draw [dashed] plot [smooth] coordinates {(-2,2) (-1,0.5) (0,0) (1,0.5) (2,2)};
                    \draw [dashed] plot [smooth] coordinates {(-2,-2) (-1,-0.5) (0,0) (1,-0.5) (2,-2)};
                    \node at (0,1) {$+1$};
                    \node at (0,-1) {$+1$};
                    \node[left] at (-1,0) {$-1$};
                    \node[right] at (1,0) {$-1$};
                \end{tikzpicture}
                \caption{(3): Move competitor to $\Gamma_{r_{0}}$ outside ball $B_{3}$ centred at $z_{0}$.}
                \label{subfig: point 3 of path}
            \end{center}
            \end{subfigure}
            \hspace{2cm}
        \begin{subfigure}{0.4\textwidth}
            \begin{center}
                \begin{tikzpicture}[scale=1.5]
                    \draw plot [smooth] coordinates {(-1.5,2) (-1,0.5) (0,0) (1,0.5) (1.5,2)};
                    \draw plot [smooth] coordinates {(-1.5,-2) (-1,-0.5) (0,0) (1,-0.5) (1.5,-2)};
                    \draw [dashed] plot [smooth] coordinates {(-2,2) (-1,0.5) (0,0) (1,0.5) (2,2)};
                    \draw [dashed] plot [smooth] coordinates {(-2,-2) (-1,-0.5) (0,0) (1,-0.5) (2,-2)};
                    \node at (0,1) {$+1$};
                    \node at (0,-1) {$+1$};
                    \node[left] at (-1,0) {$-1$};
                    \node[right] at (1,0) {$-1$};
                \end{tikzpicture}
                \caption{(4): Undo the edit inside $B_{2}$.}
                \label{subfig: point 4 of path}
            \end{center}
        \end{subfigure}
    
        \hfill 
    
        \begin{subfigure}{0.4\textwidth}
            \begin{center}
                \begin{tikzpicture}[scale=1.5]
                    \draw plot [smooth] coordinates {(-1.5,2) (-1,1) (0,0.5) (1,1) (1.5,2)};
                    \draw plot [smooth] coordinates {(-1.5,-2) (-1,-1) (0,-0.5) (1,-1) (1.5,-2)};
                    \draw [dashed] plot [smooth] coordinates {(-2,2) (-1,0.5) (0,0) (1,0.5) (2,2)};
                    \draw [dashed] plot [smooth] coordinates {(-2,-2) (-1,-0.5) (0,0) (1,-0.5) (2,-2)};
                    \node[above] at (0,1) {$+1$};
                    \node[below] at (0,-1) {$+1$};
                    \node[left] at (-1,0) {$-1$};
                    \node[right] at (1,0) {$-1$};
                \end{tikzpicture}
                \caption{(5): Push up in $B_{3}$ to come into line with $\Gamma_{r_{0}}$.}
                \label{subfig: point 5 of path}
            \end{center}
        \end{subfigure}
        \hspace{2cm}
        \begin{subfigure}{0.4\textwidth}
            \begin{center}
                \begin{tikzpicture}[scale=1.5]
                    \draw plot [smooth] coordinates {(-2,1.5) (-1, 0.25) (-0.5, 0)};
                    \draw plot [smooth] coordinates {(0.5,0) (1,0.25) (2,1.5)};
                    \draw plot [smooth] coordinates {(-2,-1.5) (-1,-0.25) (-0.5, 0)}; 
                    \draw plot [smooth] coordinates {(0.5,0) (1,-0.25) (2,-1.5)};
                    \draw [dashed] plot [smooth] coordinates {(-2,2) (-1,0.5) (0,0) (1,0.5) (2,2)};
                    \draw [dashed] plot [smooth] coordinates {(-2,-2) (-1,-0.5) (0,0) (1,-0.5) (2,-2)};
                    \node at (0,1) {$+1$};
                    \node at (0,-1) {$+1$};
                    \node[left] at (-1,0) {$-1$};
                    \node[right] at (1,0) {$-1$};
                \end{tikzpicture}
                \caption{(6): Push Competitor to come in line with $\Gamma_{- \rho}$.}
                \label{subfig: point 6 of path}
            \end{center}
        \end{subfigure}
        \caption{Stages of the Path at the non-embedded point $z_{0}$. 
        In each image, the dashed lines represent the original $\lambda$-CMC disks, as a reference to what we are changing at each step.
        Furthermore, in each image, it is the solid lines that are the boundaries between the '$+1$' and '$-1$' regions.}
        \label{fig: Local pictures about non-emebedded point in the Path}
    \end{center}
    \end{figure}
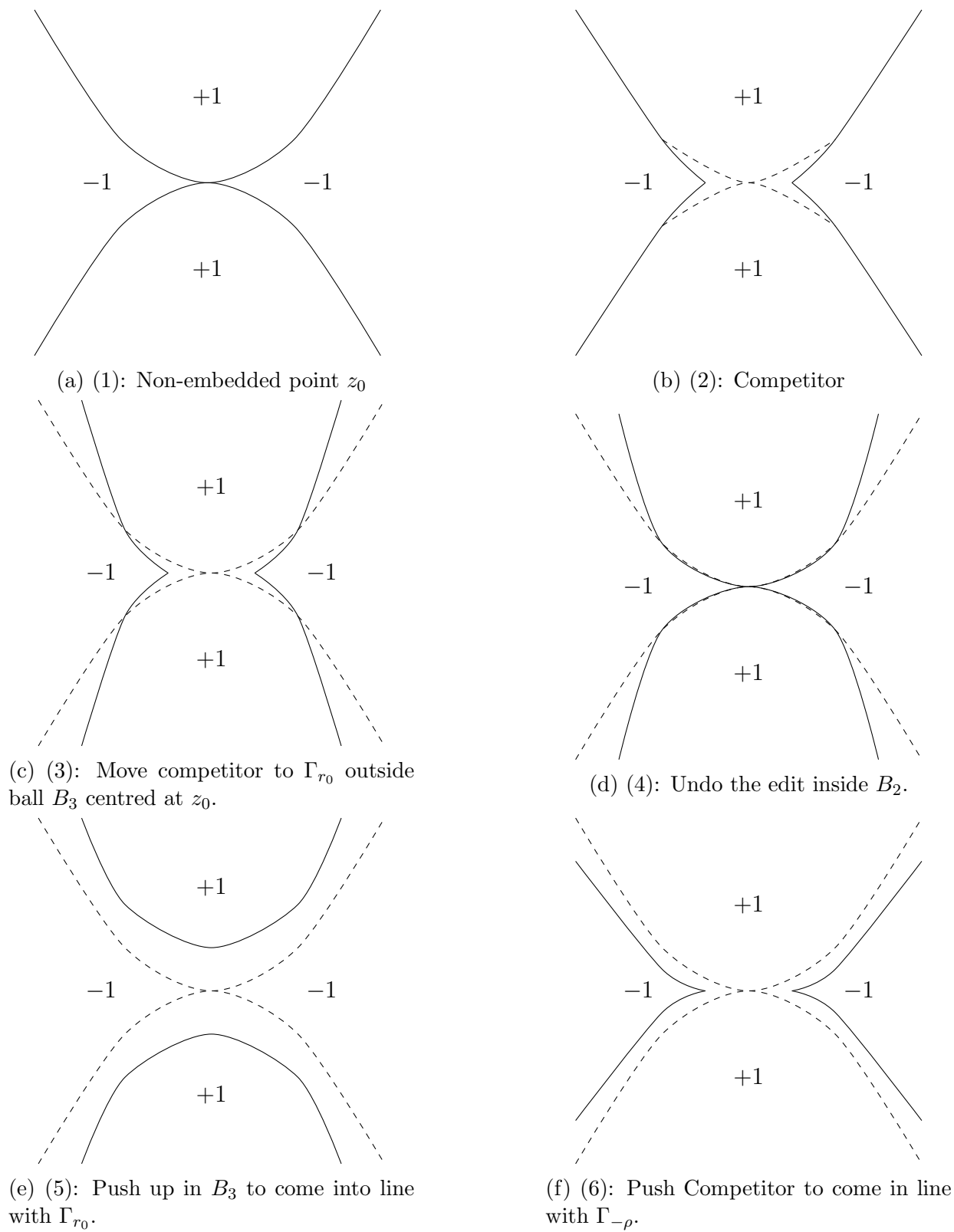

\bigskip 

\textbf{\textit{Path at $\varepsilon$ Level}}

\bigskip 

We carry out this 'pushing', on what we refer to as our abstract cylinder, $\tilde{M} \times \mathbb{R}$.
See Section \ref{subsec: abstract cylinder}.
Here $\tilde{M}$ is an n-dimensional manifold and $\iota \colon \tilde{M} \rightarrow M$ is a smooth immersion.
We define the following map, 
\begin{eqnarray*}
    F \colon \tilde{M} \times \mathbb{R} &\rightarrow& N, \\
    (x,t) &\mapsto& \exp_{\iota(x)} (t \nu(x))),
\end{eqnarray*}
with $\nu$ being a smooth choice of unit normal to immersion, pointing into $E$.
Therefore, we view points $(x,t)$ on our cylinder $\tilde{M} \times \mathbb{R}$ as having base point $\iota(x)$ and moving length $t$ along the geodesic with initial direction $\nu(x)$.
See Figure \ref{fig: Local pictures of M and the corresponding local pictures in abstract cylinder}.

\bigskip 

Recall our function $v_{\varepsilon}^{0} = \overHe_{\varepsilon} \circ \tilde{d}$, then by the Co-Area formula, 
\begin{eqnarray*}
    \mathcal{F}_{\varepsilon, \lambda} (v_{\varepsilon}^{0}) &=& \int_{\mathbb{R}} \left( \frac{\varepsilon}{2} |(\overHe_{\varepsilon})' (t)|^{2} + \frac{W(\overHe_{\varepsilon} (t))}{\varepsilon} - \sigma \lambda \overHe_{\varepsilon} (t) \right) \, \mathcal{H}^{n} (\Gamma_{t}) \, dt, \\
    &=& \int_{\mathbb{R}} \int_{\tilde{M}} \left( \frac{\varepsilon}{2} |(\overHe_{\varepsilon})' (t)|^{2} + \frac{W(\overHe_{\varepsilon} (t))}{\varepsilon} - \sigma \lambda \overHe_{\varepsilon} (t) \right) \, \theta_{t} (x) \, d \mathcal{H}^{n} (x) \, dt,
\end{eqnarray*}
where $\theta_{t} \colon \tilde{M} \rightarrow \mathbb{R}$, is defined by the Area Formula to be such that for a.e $t \in \mathbb{R}$, and any $\mathcal{H}^{n}$-measurable function on $N$, 
\begin{equation*}
    \int_{\Gamma_{t}} g \, d \mathcal{H}^{n} = \int_{\tilde{M}} (g \circ F_{t}) \, \theta_{t} \, d \mathcal{H}^{n},
\end{equation*}
with $F_{t} ( \, \cdot \, ) = F ( \, \cdot \, , t)$.
Then we carry out the relevant 'pushings' by considering a continuous family of functions $\{g_{r}\}_{r \in [0,r']}\subset C(\tilde{M})$,   
\begin{eqnarray*}
    \mathcal{F}_{\varepsilon, \lambda} (v_{\varepsilon}^{r}) &=& \int_{\mathbb{R}} \int_{\tilde{M}} \biggl( \frac{\varepsilon}{2} |(\overHe_{\varepsilon})' (t - g_{r}(x))|^{2} + \frac{W(\overHe_{\varepsilon} (t - g_{r}(x)))}{\varepsilon} \\
    && \hspace{3cm} - \sigma \lambda \overHe_{\varepsilon} (t - g_{r}(x)) \biggr) \, \theta_{t} (x) \, d\mathcal{H}^{n} (x,t) \, dt.
\end{eqnarray*}
See Figure \ref{fig: local pictures showing construction of competitor in abstract cylinder}.

\begin{figure}
    \begin{center}
        \begin{subfigure}{0.7\textwidth}
                \begin{tikzpicture}
                \draw plot [smooth] coordinates {(-2,2) (-1,0.5) (0,0) (1,0.5) (2,2)};
                \node[left] at (-2,2) {$D_{1}$};
                \node[left] at (-2,-2) {$D_{2}$};
                \draw plot [smooth] coordinates {(-2,-2) (-1,-0.5) (0,0) (1,-0.5) (2,-2)};
                \draw[dotted, ultra thick] (-2,0) -- (2,0);
                \draw[->, thick] (0,0) -- (0, 0.5);
                \node[right] at (0,0.5) {$\nu_{1}$};
                \node[left] at (0,-0.5) {$\nu_{2}$};
                \draw[->, thick] (0,0) -- (0, -0.5);
                \node at (0,1.5) {$+1$};
                \node at (0,-1.5) {$+1$};
                \node at (-1.5,0.5) {$-1$}; 
                \node at (-1.5, -0.5) {$-1$};
                \node at (1.5,0.5) {$-1$}; 
                \node at (1.5, -0.5) {$-1$};
                \draw[->] (4,0) -- (3,0);
                \node[above] at (3.5,0) {$F$};
                \draw[dotted, ultra thick] plot [smooth] coordinates {(5, 0.5) (7,1.5) (9, 0.5)};
                \draw (5,1.5) -- (9,1.5); 
                \draw[dotted, ultra thick] plot [smooth] coordinates {(5, -0.5) (7, -1.5) (9, -0.5)};
                \draw (5, -1.5) -- (9, -1.5); 
                \draw[->, thick] (7,1.5) -- (7,2);
                \node[right] at (7,2) {$\partial_{t}$};
                \draw[->, thick] (7,-1.5) -- (7,-2);
                \node[right] at (7,-2) {$\partial_{t}$};
                \draw[->] (9.5, 2) -- (9.5,2.5);
                \node[right] at (9.5,2.5) {$t$};
                \draw[->] (9.5, -2) -- (9.5,-2.5);
                \node[right] at (9.5,-2.5) {$t$};
                \node[right] at (9,1.5) {$t = 0$};
                \node[right] at (9,-1.5) {$t = 0$};
                \node[left] at (5,1.5) {$\tilde{D}_{1}$};
                \node[left] at (5,-1.5) {$\tilde{D}_{2}$};
                \node[below] at (7, 1.5) {$x_{0}^{1}$}; 
                \node[above] at (7,-1.5) {$x_{0}^{2}$};
                \end{tikzpicture}        
                \caption{On the left we have a local picture about a non-emebedded point $z_{0}$ of $M$. 
                On the right the two local pictures about $x_{0}^{1}$ and $x_{0}^{2}$ in $\tilde{M} \times \mathbb{R}$, where $\iota(x_{0}^{1}) = z_{0} = \iota(x_{0}^{2})$. 
                We have, $F(\tilde{D}_{i}) = D_{i}$, and $dF_{x_{0}^{i}} (\partial_{t}) = \nu_{i}$, for $i = 1$ and $2$.
                The dotted line on the left picture represents points in $N$ which are of equal distance to $D_{1}$ and $D_{2}$. 
                The dotted lines on the right-hand picture are the preimages of the dotted line on the left, under the map $F$, and these can be seen as acting as the boundary to the open set $\tilde{T}$ in $\tilde{M} \times \mathbb{R}$.}
                \label{subfig: local picture about non-embedded point}
        \end{subfigure}

        
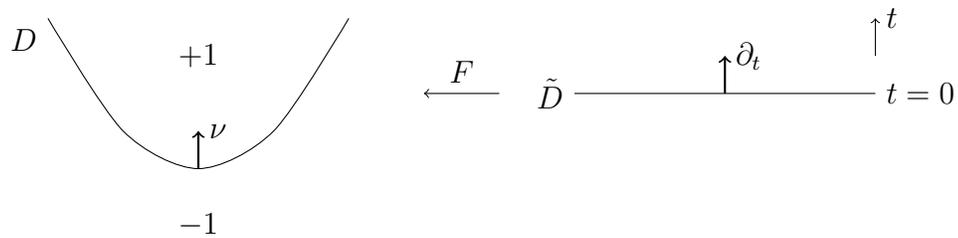
\bigskip 

        \begin{subfigure}{0.7\textwidth} 
            \begin{tikzpicture}
            \draw plot [smooth] coordinates {(-2,2) (-1,0.5) (0,0) (1,0.5) (2,2)}; 
            \node[below left] at (-2,2) {$D$}; 
            \node at (0,1.5) {$+1$}; 
            \node at (0,-0.75) {$-1$}; 
            \draw[->, thick] (0,0) -- (0,0.5); 
            \node[right] at (0,0.5) {$\nu$}; 
            \draw[->] (4,1) -- (3,1);
            \node[above] at (3.5,1) {$F$};
            \draw (5,1) -- (9,1); 
            \node[left] at (5,1) {$\tilde{D}$}; 
            \node[right] at (9,1) {$t = 0$}; 
            \draw[->] (9, 1.5) -- (9,2); 
            \node[right] at (9,2) {$t$}; 
            \draw[->, thick] (7,1) -- (7,1.5); 
            \node[right] at (7,1.5) {$\partial_{t}$};
            \end{tikzpicture}
            \caption{On the left, a local picture about an embedded point of $M$.
            On the right is its preimage in $\tilde{T}$ under the map $F$.}
            \label{subfig: Local picture about embedded point}
        \end{subfigure}
        \caption{Local pictures about points in $M$}
        \label{fig: Local pictures of M and the corresponding local pictures in abstract cylinder}
    \end{center}
\end{figure}

\begin{figure}
    \begin{center}
        \begin{tikzpicture}
            \draw plot [smooth] coordinates {(-2,2) (-1, 0.5) (-0.5, 0)};
            \draw plot [smooth] coordinates {(0.5,0) (1,0.5) (2,2)};
            \draw plot [smooth] coordinates {(-2,-2) (-1,-0.5) (-0.5, 0)}; 
            \draw plot [smooth] coordinates {(0.5,0) (1,-0.5) (2,-2)};
            \draw [dashed] plot [smooth] coordinates {(-1,0.5) (0,0) (1,0.5)};
            \draw [dashed] plot [smooth] coordinates {(-1,-0.5) (0,0) (1,-0.5)};
            \node[left] at (-2,2) {$D_{1}$};
            \node[left] at (-2,-2) {$D_{2}$};
            \draw[dotted, ultra thick] (-2,0) -- (2,0);
            \node at (0,1.5) {$+1$};
            \node at (0,-1.5) {$+1$};
            \node at (-1.5,0.5) {$-1$}; 
            \node at (-1.5, -0.5) {$-1$};
            \node at (1.5,0.5) {$-1$}; 
            \node at (1.5, -0.5) {$-1$};
            \draw[->] (4,0) -- (3,0);
            \node[above] at (3.5,0) {$F$};
            \draw[dotted, ultra thick] plot [smooth] coordinates {(5, 0.5) (7,1.5) (9, 0.5)};
            \draw plot [smooth] coordinates {(5,1.5) (5.75, 1.5) (6.5,1) (7.5,1) (8.25, 1.5) (9,1.5)}; 
            \draw[dashed] (5.75, 1.5) -- (8.25,1.5);
            \draw[dotted, ultra thick] plot [smooth] coordinates {(5, -0.5) (7, -1.5) (9, -0.5)};
            \draw plot [smooth] coordinates {(5,-1.5) (5.75, -1.5) (6.5,-1) (7.5,-1) (8.25, -1.5) (9,-1.5)}; 
            \draw[dashed] (5.75, -1.5) -- (8.25, -1.5);
            \draw[->] (9.5, 2) -- (9.5,2.5);
            \node[right] at (9.5,2.5) {$t$};
            \draw[->] (9.5, -2) -- (9.5,-2.5);
            \node[right] at (9.5,-2.5) {$t$};
            \node[right] at (9,1.5) {$t = 0$};
            \node[right] at (9,-1.5) {$t = 0$};
            \node[left] at (5,1.5) {$\tilde{D}_{1}$};
            \node[left] at (5,-1.5) {$\tilde{D}_{2}$};
            \end{tikzpicture}   
            \caption{How the competitor is constructed as the graph of bump functions about points $x_{0}^{1}$ and $x_{0}^{2}$ over $\tilde{M}$. 
            Whatever is bumped out beyond the dotted line, on the right-hand side, is not considered in $N$. 
            In other words, it is deleted.}
            \label{fig: local pictures showing construction of competitor in abstract cylinder}
    \end{center}
\end{figure}
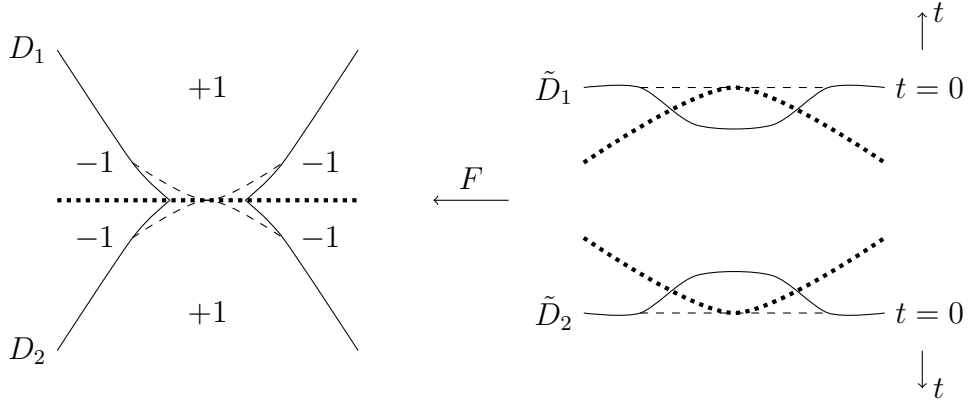

\bigskip 

\subsection{Structure of the Paper}

The paper is organised as follows.
We start with setup:
\begin{itemize}
    \item Section \ref{sec: Construction of Objects} is devoted to set up of objects used in the main computation. 
    \item In Section \ref{sec: Baseline Computation} we carry out the main computation.
    The constructions that follow are carried out by plugging explicitly defined functions into this computation.
\end{itemize}
To prove Theorem \ref{thm: Allen-Cahn minmax limit}:
\begin{itemize}
    \item In Section \ref{subsec: Recovery Path} we build the dotted path (5) $\rightarrow$ (1) $\rightarrow$ (6) in Figure \ref{fig: The Path}.
    Theorem \ref{thm: Allen-Cahn minmax limit} then follows upon combining this with computations in Sections \ref{sec: Completing Path to -1} and \ref{sec: Completing Path to +1}; in these sections we build the paths (5) $\rightarrow$ $-1$, and (6) $\rightarrow$ $+1$, in Figure \ref{fig: The Path}.
\end{itemize}
To prove Theorem \ref{thm: Main Theorem} we argue by contradiction, assuming that $M$ has a non-embedded point $z_{0}$:
\begin{itemize}
    \item In Section \ref{sec: Path From Disks to Crab} we construct our competitor about $z_{0}$. 
    This is the dashed path (1) $\rightarrow$ (2) in Figure \ref{fig: The Path}.
    \item In Section \ref{sec: path to minus 1} we construct a path from the competitor to the stable constant $a_{\varepsilon}$. 
    This is the solid path (2) $\rightarrow$ (6) $\rightarrow$ $+1$, in Figure \ref{fig: The Path}.
    \item In Section \ref{sec: path to plus 1} we construct a path from the competitor to the stable constant $b_{\varepsilon}$.
    This is the solid path (2) $\rightarrow$ (3) $\rightarrow$ (4) $\rightarrow$ (5) $\rightarrow$ $-1$ in Figure \ref{fig: The Path}.
    \item In Section \ref{sec: Pieceing Path Together} we piece together this continuous path from $a_{\varepsilon}$ to $b_{\varepsilon}$, in $W^{1,2}(N)$.
    The energy $\mathcal{F}_{\varepsilon, \lambda}$, is less than $A_{2} - \varsigma$ for every point along this path, Figure \ref{fig: The Path}.
    This contradicts the min-max construction, proving that $M$ is embedded.
\end{itemize}
Finally, in Section \ref{sec: index} we prove Corollary \ref{cor:index} (the Morse index of $M$ is equal to $1$, which also implies that $M$ must be connected).

\subsection{A Note on Choice of Constants}

The biggest subtlety in the Construction of the path in Sections \ref{sec: Path From Disks to Crab}, \ref{sec: path to minus 1} and \ref{sec: path to plus 1} is the choice of constants, and the order that we choose them in. 
We explicitly list the order of choices here, and reference where they have been chosen. 
\begin{enumerate}
    \item We first choose a non-embedded point $z_{0}$
    \item We choose $\delta = \delta (z_{0}, N, M, g, \lambda, W) > 0$, in Remarks \ref{rem: defining tilde d 1 and tilde d 2}, \ref{rem: diffeos F i}, \ref{rem: choice of delta for set d 1 = d 2}, \ref{rem: choice of delta so that sigma minus has lower quadratic bound}, \ref{rem: choice of delta for upper and lower mean curvature bounds in ball about z 0}. 
    \item We choose $L = L (z_{0}, N, M, g, \delta, \lambda, W) > 0$, in Remarks \ref{rem: Choice of L and r 0 to fit in ball B delta}, \ref{rem: second choice of L}.
    \item We choose $k = k (z_{0}, N, M, g, \delta, L, \lambda, W)$, in Remark \ref{rem: Choice of gamma and k}.
    \item We choose $r_{0} = r_{0} (z_{0}, N, M, g, \delta, L, k, \lambda, W) > 0$, in Remarks \ref{rem: Choice of L and r 0 to fit in ball B delta}, \ref{rem: second choice of L}, \ref{rem: Choice of r 0 from disks push up claim}.
    \item We choose $\rho = \rho (z_{0}, N, M, g, \delta, L, k, r_{0}, \lambda, W) > 0$, in Remarks \ref{rem: choice of rho so that B l fits into D 1}, \ref{rem: choice of rho so that B 2 l x 0 1 lies in W 1}, \ref{rem: choice of rho so that B l times -rho rho fits into B delta}, \ref{rem: choice of rho based on r 0 and L}, \ref{rem: choice of rho for path from disks to crab}, \ref{rem: choice of rho based on r 0}. 
    \item We define $l = l (\rho)$ in (\ref{eqn: def of l}).
    \item We choose $\tau > 0$. 
    \item We finally choose $\varepsilon_{\tau} = \varepsilon_{\tau} (z_{0}, N, M, g, \delta, L, k, r_{0}, \rho, \tau, \lambda, W) > 0$, in Remarks \ref{rem: choice in epsilon 1}, \ref{rem: Choice of epsilon 2 for function P epsilon}, \ref{rem: Choice of epsilon 3} and Sections \ref{subsec: Recovery Path} and \ref{sec: Pieceing Path Together}.
\end{enumerate}

\section{Construction of Objects}\label{sec: Construction of Objects}

\subsection{Signed Distance Function}

Let $d_{\overline{M}} \colon N \rightarrow \mathbb{R}$ be the distance function to the closed set $\overline{M} \subset N$.
As $\overline{M}$ is closed, and $N$ is complete, Hopf--Rinow tells us that, for each $z$ in $N$, the value, $d_{\overline{M}}(z)$, is obtained by a geodesic from $z$ to a point on $\overline{M}$.
Furthermore, $d_{\overline{M}}$ is Lipschitz, with Lipschitz constant 1.

\bigskip

The set $E = \{u_{0} = 1\}$ is an open in $N$, and $\overline{M} = \partial E$. 
This allows us to define the signed distance function, $\tilde{d} \colon N \rightarrow \mathbb{R}$, to $\overline{M}$, which takes positive values in $E$, and negative values in $N \setminus E$,
\begin{equation*}
    \tilde{d}(y) = \begin{cases}
        d_{\overline{M}} (z), & x \in E, \\
        - d_{\overline{M}} (z), & x \not\in E.
    \end{cases}
\end{equation*}
This is a $1$-Lipschitz function on $N$.

\subsection{Abstract Surface} 

$M$ is a quasi-embedded $\lambda$-CMC hypersurface, \cite[Definition 8]{BW-inhomogeneous-allen-cahn}.

\begin{rem}\label{rem: graphical characterisation of points on M}
    For a point $z \in M$, there exists an $n$-dimensional linear subspace $T = T_{z} \subset T_{z} N$, and a unit vector $\nu_{z} \in T^{\perp}$, along with $r = r(z) >0$, $s = s(z) > 0$, and $S = S (z) >0$, such that $S < \text{inj} \, (N)$.
    We define the cylinder
    \begin{equation*}
        C_{z, T, r, s} \coloneqq \exp_{z} \left( \{ X + t \nu_{z} \colon X \in B_{r}^{T_{z} N} (0) \cap T, \, t \in (-s,s) \} \right) \subset B_{S}^{N} (z),
    \end{equation*}
    and, one of the following holds:
    \begin{enumerate}
        \item \label{rem point: Graphical rep of embedded point} 
        (See Figure \ref{subfig: Local picture about embedded point})
        There exists a smooth function, 
        \begin{equation*}
            f \colon B_{z, T, r} \coloneqq B_{r}^{T_{z} N} (0) \cap T \rightarrow (-s,s), 
        \end{equation*}
        which satisfies, 
        \begin{equation*}
            \begin{cases}
                f(0) = 0, \\
                \nabla^{T} f (0) = 0, \\
                \Delta_{T} f (0) = \lambda,
            \end{cases}
        \end{equation*}
        and, 
        \begin{equation*}
            \overline{M} \cap C_{z, T, r, s} = \exp_{z} ( \text{Graph} \, (f) ) = \exp_{z} ( \{ X + f(X) \nu_{z} \colon X \in B_{z,T,r} \} )
        \end{equation*}
        Furthermore, we have that, 
        \begin{equation*}
            E \cap C_{z, T, r, s} = \exp_{z} ( \{ X + t \nu_{z} \colon X \in B_{z, T, r}, \, f(X) < t < s \} ),
        \end{equation*}
        and we can define a smooth choice of unit normal to $\exp_{z} (\text{Graph} \, (f))$, 
        \begin{equation*}
            \nu \colon \exp_{z} (\text{Graph} \, (f)) \rightarrow T (\exp_{z} (\text{Graph} \, (f)))^{\perp},
        \end{equation*}
        such that $\nu (z) = \nu_{z}$.
        \item \label{rem point: graphical rep of non-embedded point} 
        (See Figure \ref{subfig: local picture about non-embedded point})
        There exists two smooth functions, 
        \begin{equation*}
            f_{1}, \, f_{2} \colon B_{z, T, r} \rightarrow (-s,s),
        \end{equation*}
        which satisfy, 
        \begin{equation*}
            \begin{cases}
                f_{1} (0) = 0 = f_{2} (0), \\
                f_{1} \geq f_{2}, \\
                \nabla^{T} f_{1} (0) = 0 = \nabla^{T} f_{2} (0), \\
                \Delta_{T} f_{1} (0) = \lambda = - \Delta_{T} f_{2} (0),
            \end{cases}
        \end{equation*}
        and, 
        \begin{equation*}
            \overline{M} \cap C_{z,T, r, s} = \bigcup_{i= 1,2} \exp_{z} (\text{Graph} \, (f_{i})) = \bigcup_{i = 1, 2} \exp_{z} (\{ X + f_{i} (X) \nu_{z} \colon X \in B_{z, T, r} \}).
        \end{equation*}
        Furthermore, we have that, 
        \begin{eqnarray*}
            E \cap C_{z, T, r, s} &=& \exp_{z} (\{ X + t \nu_{z} \colon X \in B_{z, T, r}, \, f_{1} (X) < t < s \}) \\
            && \hspace{1cm} \cup \, \exp_{z} (\{ X + t \nu_{z} \colon X \in B_{z, T, r}, \, - s < t < f_{2} (X) \}),
        \end{eqnarray*}
        and we can define smooth choices of unit normals, 
        \begin{equation*}
            \nu_{i} \colon \exp_{z} (\text{Graph} \, (f_{i})) \rightarrow T (\exp_{z} (\text{Graph} \, (f)))^{\perp},
        \end{equation*}
        such that $\nu_{1} (z) = \nu_{z}$, and $\nu_{2} (z) = - \nu_{z}$.
    \end{enumerate}
    If Case \ref{rem point: Graphical rep of embedded point} holds, then we say that $z$ is an embedded point of $M$. 
    Alternatively, if Case \ref{rem point: graphical rep of non-embedded point} holds, we say that $z$ is a non-embedded point of $M$.
    In either case, the tangent space of $M$ at $z$ is given by, $T_{z} M \coloneqq T_{z}$.
\end{rem}

\begin{claim}\label{claim: CMC disks intersect on a set of measure 0 M, and both unit normals point into E}
    (Remark 2.6 of \cite{BW-Stable-CMC})
    The set of non-embedded points of $M$ has $\mathcal{H}^{n}$-measure $0$.
\end{claim}

We define our abstract surface $\tilde{M}$ by
\begin{equation*}
    \tilde{M} = \{ (z, \nu) \colon z \in M, \, \nu \in T_{z} M^{\perp}, \, \text{with} \, |\nu| = 1, \, \text{and points into} \, E \}.
\end{equation*}
Locally $\tilde{M}$ is a smooth, embedded CMC disk in $N$, therefore, $\tilde{M}$ is a smooth $n$-dimensional manifold.

\subsection{Abstract Cylinder}\label{subsec: abstract cylinder}

Consider $x$ in $\tilde{M}$, then $x = (z, X)$, for some $z$ in $M$ and $X$ in $T_{z}M^{\perp}$. 
We define two, smooth projections, first from $\tilde{M}$ to $TM^{\perp}$, 
\begin{equation*}
    \nu \colon (z,X) \mapsto X,
\end{equation*}
and secondly, from $\tilde{M}$ to $M$, 
\begin{equation*}
    \iota \colon (z,X) \mapsto z.
\end{equation*}
From these we define the following map,
\begin{eqnarray*}
    F \colon \tilde{M} \times \mathbb{R} &\rightarrow& N, \\ 
    (x,t) &\mapsto& \exp_{\iota(x)} (t \nu(x)),
\end{eqnarray*}
which, as $N$ is complete, is well-defined.
For a fixed $x$ in $\tilde{M}$, $F$ is a unit parametrisation of a geodesic which, at time $0$, passes through $\iota(x)$, with velocity $\nu (x)$. 
The set $\{ t \colon d_{\overline{M}}(F(x,t)) = |t| \}$, is the set of times $t$, at which this geodesic achieves the shortest distance from $F(x,t)$ to $\overline{M}$.
Consider the subset $\{t \colon \tilde{d} (F(x,t)) = t\} \subset \{ t \colon d_{\overline{M}}(F(x,t)) = |t| \}$, and its endpoints,
\begin{eqnarray*}
    \sigma^{+}(x) &=& \sup \{ t \colon \tilde{d}(F(x,t)) = t \} \geq 0, \\ 
    \sigma^{-}(x) &=& \inf \{ t \colon \tilde{d}(F(x,t)) = t \} \leq 0.
\end{eqnarray*}

These are uniformly bounded functions on $\tilde{M}$, and in fact as the next claim shows, $\{ t \colon \tilde{d}(F(x,t)) = t \}$ is a closed and connected interval on $\mathbb{R}$.

\begin{claim}\label{claim: interval from sigma plus to sigma minus contains all t which equality holds for}
    We have that 
    \begin{equation*}
        [\sigma^{-}(x), \sigma^{+}(x)] = \{ t : \tilde{d}(F(x,t)) = t\}.
    \end{equation*}
\end{claim}

\begin{proof}
    Consider the geodesic, $\gamma : t \mapsto F(x,t)$, and define the following function, 
    \begin{eqnarray*}
        f : t \mapsto \tilde{d}(F(x,t)).
    \end{eqnarray*}
    This is a $1$-Lipschitz function with $f(0) = 0$. 
    Indeed, 
    \begin{equation*}
        |f(t_{1}) - f(t_{2})| \leq |d(F(x,t_{1}), F(x,t_{2}))| \leq \text{Length} (\gamma_{|[t_{1}, t_{2}]}) = |t_{1} - t_{2}|. 
    \end{equation*}

    For $t_{0} \geq 0$, such that $f(t_{0}) \not= t_{0}$, we must have $f(t_{0}) < t_{0}$.
    By Lipschitz constant $1$, for any $t > t_{0}$,
    \begin{eqnarray*}
        f(t) &=& f(t) - f(t_{0}) + f(t_{0}), \\
        &\leq& t- t_{0} + f(t_{0}), \\
        &<&  t.
    \end{eqnarray*}
    Similarly, if we have $t_{0} \leq 0$, such that $f(t_{0}) \not= t_{0}$, then $f(t) \not= t$, for all $t < t_{0}$.

    \bigskip 

    By continuity, we have that $\tilde{d}(F(x, \sigma^{+} (x))) = \sigma^{+}(x)$, and therefore by above, for all $t \in [0, \sigma^{+}(x)]$, we must have that $\tilde{d}(F(x,t)) = t$.
    By definition of $\sigma^{+}(x)$, for all $t > \sigma^{+}(x)$, $\tilde{d}(F(x,t)) < t$. 
    Therefore, 
    \begin{equation*}
        [0, \sigma^{+}(x)] = \{t \geq 0 \colon \tilde{d}(F(x,t)) = t \}.
    \end{equation*}
    Similarly, $[\sigma^{-}(x), 0] = \{ t \leq 0 \colon \tilde{d}(F(x,t)) = t \}.$
\end{proof}

\bigskip 

We define the abstract cylinder, 
\begin{equation*}
    \tilde{T} = \{(x,t) \colon x \in \tilde{M}, \, t \in (\sigma^{-}(x), \sigma^{+}(x))\} \subset \tilde{M} \times \mathbb{R}.
\end{equation*}
Defining the projection map from $\tilde{M} \times \mathbb{R}$ onto $\mathbb{R}$, $p \colon (x,t) \mapsto t$, then on $\tilde{T}$ we have that $\tilde{d} \circ F = p$.

\bigskip 

We wish to work on $\tilde{T}$ instead of $N$.
The following Lemma is crucial in that respect.

\begin{lem}\label{lem: end points of length minimising geodesics to M are smooth points}
    (Geodesic Touching Lemma)
    For all $y$ in $N \setminus \overline{M}$, there exists a geodesic from $y$ to $\overline{M}$ that achieves the length of $d_{\overline{M}}(y)$. 
    The end point of this geodesic on $\overline{M}$ must in fact be a quasi-embedded point of $M$, and the geodesic will hit $M$ orthogonally.
\end{lem}

\begin{proof}
    Identical argument to \cite[Lemma 3.1]{bellettini2020multiplicity1}, except we replace the Sheeting Theorem of \cite{Wickramasekera-general-regularity-varifolds} with the Sheeting Theorem of \cite{BW-Stable-CMC}.
\end{proof}

From this Lemma, the following result is immediate,

\begin{proposition}\label{prop: F is surjective onto N minus singular set of M}
    For all $y$ in $N \setminus (\overline{M} \setminus M)$, there exists an $x$ in $\tilde{M}$, such that $F(x, \tilde{d}(y)) = y$. 
\end{proposition}

\begin{claim}
    The functions, $\sigma^{+}, \, \sigma^{-} \colon \tilde{M} \rightarrow \mathbb{R}$, are continuous.
\end{claim}

\begin{proof}
    We prove by contradiction. 
    Suppose there exists an $\hat{x} \in \tilde{M}$ such that, $\liminf_{x \rightarrow \hat{x}} \sigma^{+} (x) = \alpha < \sigma^{+} (\hat{x})$. 
    Choose $0 < \delta < \sigma^{+} (\hat{x}) - \alpha$, then there exists $x_{n} \rightarrow \hat{x}$ in $\tilde{M}$ such that $\sigma^{+} (x_{n}) < \alpha + \delta$.
    Now consider the points,
    \begin{equation*}
        z_{n} = F(x_{n}, \alpha + \delta) \rightarrow z \coloneqq F(\hat{x}, \alpha + \delta).
    \end{equation*}
    By Claim \ref{claim: interval from sigma plus to sigma minus contains all t which equality holds for}, $\tilde{d} (z_{n}) < \alpha + \delta$.
    By Proposition \ref{prop: F is surjective onto N minus singular set of M} there exists a sequence $\tilde{x}_{n}$, such that, 
    \begin{equation*}
        F(\tilde{x}_{n}, \tilde{d}(z_{n})) = z_{n}.
    \end{equation*}
    After potentially taking a subsequence and renumerating we have that there exists a $y \in \overline{M}$, such that $\iota(\tilde{x}_{n}) \rightarrow y$, then note $d(y, z) = \tilde{d} (z) = \alpha + \delta$. 
    Therefore, by Lemma \ref{lem: end points of length minimising geodesics to M are smooth points}, $y \in M$, and as $t \mapsto F(\hat{x}, t)$ is the unique length minimising geodesic from $M$ to $z$, $\tilde{x}_{n} \rightarrow \hat{x}$ in $\tilde{M}$. 
    Now we have that, 
    \begin{equation*}
        F(x_{n}, \alpha + \delta) = z_{n} = F(\tilde{x}_{n}, \tilde{d}(z_{n})).
    \end{equation*}
    However, $(x_{n}, \alpha + \delta) \not= (\tilde{x}_{n}, \tilde{d}(z_{n}))$, and 
    \begin{equation*}
        \lim_{n \rightarrow \infty} (x_{n}, \alpha + \delta) = (\hat{x}, \alpha + \delta) = \lim_{n \rightarrow \infty} (\tilde{x}_{n}, \tilde{d}(z_{n})).
    \end{equation*}
    This implies that $F$ is not a diffeomorphism about the point $(\hat{x}, \alpha + \delta)$, and therefore by classical theory of geodesics, \cite[Lemma 2.11]{sakai1996riemannian}, $t \mapsto F(\hat{x}, t)$ is no longer length minimising to $M$ beyond time $t = \alpha + \delta$. 
    This contradicts $\alpha + \delta < \sigma^{+} (\hat{x})$.

    \bigskip 

    Now suppose that $\sigma^{+} (\hat{x}) < \limsup_{x \rightarrow \hat{x}} \sigma^{+} (x) = \beta < + \infty$. 
    Choose $0 < \delta < \beta - \sigma^{+} (\hat{x})$, and sequence $x_{n} \rightarrow \hat{x}$, such that, 
    \begin{equation*}
        \sigma^{+} (x_{n}) > \sigma^{+} (\hat{x}) + \delta.
    \end{equation*}
    Define, 
    \begin{equation*}
        z_{n} = F(x_{n}, \sigma^{+} (\hat{x}) + \delta),
    \end{equation*}
    then $\tilde{d}(z_{n}) = \sigma^{+} (\hat{x}) + \delta$. 
    By continuity of $F$, 
    \begin{equation*}
        z_{n} \rightarrow z \coloneqq F(\hat{x}, \sigma^{+} (\hat{x}) + \delta).
    \end{equation*}
    However, by definition of $\sigma^{+} (\hat{x})$, $\tilde{d} (z) < \sigma^{+} (\hat{x}) + \delta = \tilde{d} (z_{n})$. 
    This contradicts continuity of $\tilde{d}$. 

    \bigskip 

    Similar arguments show that $\sigma^{-}$ is also continuous. 
\end{proof}

\bigskip 

We define the Cut Locus of $M$ to be the following points in $N$, 
\begin{equation*}
    \text{Cut} \, (M) = \{ F(x, \sigma^{+} (x)) \colon x \in \tilde{M} \} \cup \{ F(x, \sigma^{-} (x)) \colon x \in \tilde{M} \} \subset N.
\end{equation*}
and by Proposition \ref{prop: F is surjective onto N minus singular set of M}, we have that,
\begin{equation*}
    N \setminus (\overline{M} \setminus M) = F(\tilde{T}) \cup \text{Cut} \, (M).
\end{equation*}

\begin{proposition}\label{prop: nabla tilde d lies in SBV}
    $\text{Cut} \, (M)$ is an $n$-rectifiable set.
\end{proposition}

To prove Proposition \ref{prop: nabla tilde d lies in SBV}, we first classify points in $\text{Cut} (M)$, 
\begin{proposition}\label{prop: Classification of points in Cut M}
    A point $y$ in $N \setminus (\overline{M} \setminus M)$, lies in $\text{Cut}(M)$ if and only if at least one of the following conditions holds: 
    \begin{enumerate}
        \item \label{point: F is non-invertible} $y$ lies in $N \setminus \overline{M}$, and there exists an $x$ in $\tilde{M}$ such that $F(x, \tilde{d}(y)) = y$, and $dF_{(x, \tilde{d}(y))} \colon T_{x} \tilde{M} \times \mathbb{R} \rightarrow T_{y} N$, is non-invertible. 
        \item \label{point: there exists atleast two unique geodesics} $y$ lies in $N \setminus \overline{M}$, and there exists at least two unique geodesics from $y$ to $\overline{M}$ which achieve the length $d_{\overline{M}}(y)$. 
        \item \label{point: non-embedded point of M} $y$ is a non-embedded point of $M$. 
    \end{enumerate}
\end{proposition}

\begin{proof}
    Consider a point $y = F(x, 0) \in M$. 
    If $y$ is an embedded point of $M$, then case \ref{rem point: Graphical rep of embedded point} of Remark \ref{rem: graphical characterisation of points on M} holds, and there exists an $S > 0$, such that $\overline{M} \cap B_{S}(y)$ is a smooth, embedded CMC disk. 
    Therefore, (\cite[Proposition 4.2]{Mantegazza-Mennucci_2002}) there exists an $r$ in $(0, S /2)$, such that for all $t$ in $(-r,r)$, $\tilde{d}(F(x,t)) = t$.
    Therefore, if $y \in M \cap \text{Cut} (M)$, then $y$ must be a non-embedded point. 

    \bigskip 

    Alternatively, if $y$ is a non-embedded point then case \ref{rem point: graphical rep of non-embedded point} of Remark \ref{rem: graphical characterisation of points on M} holds, and $(y, \nu)$ and $(y, - \nu)$ both lie in $\tilde{M}$.  
    Moreover, for $t \in (- s, 0)$, $t < f_{2} (0)$, implying that $F((y, \nu), t) = \exp_{y} (t \nu)$ lies in $E$. 
    Therefore, $\tilde{d} (F((y, \nu), t)) \geq 0$, implying that $\sigma^{-} (y, \nu) = 0$, and thus $y$ is a point in $\text{Cut} (M)$. 

    \bigskip 

    For $y \in N \setminus \overline{M}$, the conclusion follows from standard theory of geodesics, see \cite{sakai1996riemannian}. 
    We can use this classical theory in our setting by Lemma \ref{lem: end points of length minimising geodesics to M are smooth points}. 
    This observation is seen \cite[Proposition 3.1]{bellettini2020multiplicity1}.
\end{proof}

\begin{rem}\label{rem: Cut M and F tilde T are disjoint in N}
    By point \ref{point: there exists atleast two unique geodesics} of Proposition \ref{prop: Classification of points in Cut M}, $F(\tilde{T})$ and $\text{Cut} \, (M)$ must be disjoint.
    Therefore, by point \ref{point: F is non-invertible} of Proposition \ref{prop: Classification of points in Cut M}, $F$ must be a local diffeomorphism on $\tilde{T}$. 
    Moreover, by point \ref{point: there exists atleast two unique geodesics}, $F \colon \tilde{T} \rightarrow F(\tilde{T})$ is a bijection.
\end{rem}

\begin{proof}
    (of Proposition \ref{prop: nabla tilde d lies in SBV})
    As $\text{Cut} (M) \cap M$ consists of non-embedded points of $M$, by Claim \ref{claim: CMC disks intersect on a set of measure 0 M, and both unit normals point into E} we have $\mathcal{H}^{n} (\text{Cut} (M) \cap M) = 0$. 
    Therefore, to prove that $\text{Cut} (M)$ is rectifiable, we just need to concern ourselves with $\text{Cut} (M) \setminus M$. 
    This follows from the observation made in the proof of \cite[Proposition 3.1]{bellettini2020multiplicity1}, that as Lemma \ref{lem: end points of length minimising geodesics to M are smooth points} holds, then the arguments in \cite[Theorem 4.10]{Mantegazza-Mennucci_2002} hold verbatim. 
\end{proof}

\begin{rem}
    As $M$ is smooth, we have that $\tilde{d}$ is smooth in $F(\tilde{T})$, \cite[Proposition 4.2]{Mantegazza-Mennucci_2002}.
\end{rem}

Denoting $h = F^{*} g$, we have that $F \colon (\tilde{T}, h) \rightarrow (F(\tilde{T}), g)$, is a bijective, local isometry.

\bigskip 

Consider the projection map, 
\begin{eqnarray*}
    p \colon \tilde{M} \times \mathbb{R} &\rightarrow& \mathbb{R}, \\
    (x,t) &\mapsto& t.
\end{eqnarray*}
In $\tilde{T}$, we have that $p = \tilde{d} \circ F$, and
\begin{equation*}
    |\nabla p (x,t)|_{h} = |\nabla \tilde{d} (F(x,t))|_{g} = 1.
\end{equation*}
We denote the sets, 
\begin{equation*}
    \tilde{\Gamma}_{t} = p^{- 1} (t) \cap \tilde{T},
\end{equation*}
and, 
\begin{equation*}
    \Gamma_{t} = \tilde{d}^{-1} (t) \subset N.
\end{equation*}
Note,
\begin{equation*}
    F(\tilde{\Gamma}_{t}) = \begin{cases}
        \Gamma_{t} \cap F(\tilde{T}) = \Gamma_{t} \setminus \text{Cut}(M), & t \not= 0, \\
        \{ \text{embedded points of } M \}, & t = 0.
    \end{cases}
\end{equation*}
Denote $H_{\tilde{\Gamma}_{t}} (x,t)$ as the scalar mean curvature of $\tilde{\Gamma}_{t}$, at $(x,t)$, with respect to unit normal $\nabla p (x,t)$, and define the following function, 
\begin{eqnarray*}
    H_{t} \colon \tilde{M} &\rightarrow& \mathbb{R}, \\
    x &\mapsto& \begin{cases} 
        H_{\tilde{\Gamma}_{t}} (x,t), & (x,t) \in \tilde{T}, \\
        0, & (x,t) \not\in \tilde{T},
    \end{cases}
\end{eqnarray*}
For $(x,t)$ in $\tilde{T}$, we have, 
\begin{equation*}
    H_{t} (x) = - \text{tr}_{T_{(x,t)} \tilde{\Gamma}_{t}} h(\nabla_{\cdot} \, \nabla p (x,t), \, \cdot \,) = - \Delta_{\tilde{\Gamma}_{t}} p (x,t).
\end{equation*}
However, as $\nabla p$ is a geodesic vector field
\begin{equation*}
    \nabla_{\nabla p} \nabla p = 0,
\end{equation*}
and as $|\nabla p| = 1$,
\begin{equation*}
    h(\nabla_{X} \nabla p, \nabla p) = \frac{1}{2} X(|\nabla p|) = 0.
\end{equation*}
Therefore, $\Delta_{\tilde{\Gamma}_{t}} p (x,t) = \Delta_{\tilde{T}} p (x,t)$, and thus for $(x,t)$ in $\tilde{T}$,
\begin{equation*}
    H_{t} (x) = - \Delta p (x,t).
\end{equation*}

\begin{proposition}\label{prop: derivative of H t is negative}
    (\cite[Corollary 3.6]{gray2012tubes})
    For $(x,t)$ in $\tilde{T}$, 
    \begin{equation*}
        \partial_{t} H_{t} (x) = - \nabla p (\Delta p) (x,t) \geq m,
    \end{equation*}
    where $m = \inf_{|X| = 1} \text{Ric}_{g} (X,X) > 0$.
\end{proposition}

\begin{rem}\label{rem: Bounds on mean curvature of level sets}
    Consider fixed $x$ in $\tilde{M}$. 
    For $\sigma^{-} (x) < 0$, we have $H_{0} (x) = \lambda$.
    If $\sigma^{-} (x) = 0$, we still have, 
    \begin{equation*}
        \lim_{t \searrow 0} H_{t} (x) = \lambda.
    \end{equation*}
    Thus, by Proposition \ref{prop: derivative of H t is negative}, we have for $(x,t) \in \tilde{T}$, 
    \begin{equation*}
        \begin{cases}
            H_{t} (x) \geq \lambda + m \, t, & t > 0, \\
            H_{0} (x) = \lambda, & \\
            H_{t} (x) \leq \lambda + m \, t, & t < 0.
        \end{cases}
    \end{equation*}
\end{rem}

\subsection{Area Element}\label{subsec: Area Element}

We define the function on $\tilde{M}$, 
\begin{eqnarray*}
    \theta_{t} (x) = \begin{cases}
        J_{\Pi_{t}} (x), & (x,t) \in \tilde{T}, \\
        0, & (x,t) \not\in \tilde{T},
    \end{cases}
\end{eqnarray*}
where, $J_{\Pi_{t}}$ is the Jacobian of the map $\Pi_{t} \colon x \in \tilde{M} \mapsto (x,t) \in \tilde{M} \times \mathbb{R}$.
By the Area Formula, 
\begin{equation*}
    \int_{\tilde{M}} \theta_{t} \, d \mathcal{H}^{n} = \mathcal{H}^{n} (\tilde{\Gamma}_{t}).
\end{equation*}

\bigskip

\begin{proposition}\label{claim: derivative of log area element}
    (\cite[Theorem 3.11]{gray2012tubes})
    For $(x_{0}, t_{0})$ in $\tilde{T}$,
    \begin{equation*}
        \partial_{t} \, \log(\theta_{t}) (x_{0})_{|t = t_{0}} = - H_{t_{0}}(x_{0}).
    \end{equation*}
\end{proposition}

\bigskip 

Consider a fixed point $(x_{0}, t_{0})$ in $\tilde{T}$. 
First, consider $t_{0} \geq 0$.
For all $t$ in $(0, t_{0}]$, $(x_{0}, t)$ lies in $\tilde{T}$, which implies that the function $t \mapsto \theta_{t} (x_{0})$ is smooth on the interval $(0, t_{0}]$.
Furthermore, $\lim_{t \rightarrow 0^{+}} \theta_{t} (x_{0}) = 1$, and applying the Fundamental Theorem of Calculus,
\begin{equation*}
    \log (\theta_{t_{0}}(x_{0})) = - t_{0} \left( \lambda + \frac{1}{2} m t_{0} \right).
\end{equation*} 
Therefore,
\begin{equation*}
    \theta_{t_{0}}(x_{0}) \leq e^{-t_{0} \left( \lambda + \frac{1}{2} m t_{0} \right)}.
\end{equation*}
Identical inequality holds for $t_{0} \leq 0$.

\bigskip 

The term $- t (\lambda + \frac{1}{2} mt)$ achieves a global maximum at $t = - \frac{\lambda}{m}$.
Noting that for $(x_{0}, t_{0})$ not in $\tilde{T}$, $\theta_{t_{0}}(x_{0}) = 0$, we have that,
\begin{equation*}
   0 \leq \theta_{t_{0}}(x_{0}) \leq e^{\frac{\lambda^{2}}{2 m}},
\end{equation*}
for all $(x_{0}, t_{0})$ in $\tilde{M} \times \mathbb{R}$.

\subsection{Construction About Non-Embedded point}\label{subsec: construction about non-embed point}

Let $z_{0}$ in $M$ be a non-embedded point. 

\begin{rem}\label{rem: defining tilde d 1 and tilde d 2}
    We are in case \ref{rem point: graphical rep of non-embedded point} of Remark \ref{rem: graphical characterisation of points on M}.
    We can choose $\delta = \delta(z_{0}, M, N, g)$ such that, 
    \begin{equation*}
        B_{2 \delta} (z_{0}) \subset C_{z_{0}, T, r, s}.
    \end{equation*}
    We have three disjoint sets, 
    \begin{eqnarray*}
        && E_{1} \coloneqq \exp_{z} (\{ X + t \nu \colon X \in B_{z_{0}, T, r}, \, f_{1} (X) < t < s \}) \cap B_{2 \delta}^{N} (z_{0}), \\
        && F \coloneqq \exp_{z} (\{ X + t \nu \colon X \in B_{z_{0}, T, r}, \, f_{2} (X) \leq t \leq f_{1} (X) \}) \cap B_{2 \delta}^{N} (z_{0}), \\
        && E_{2} \coloneqq \exp_{z} (\{ X + t \nu \colon X \in B_{z_{0}, T, r},\, - s <t < f_{2} (X) \}) \cap B_{2 \delta}^{N} (z_{0}). \\
    \end{eqnarray*}
    As $\partial E_{i} \cap B_{2 \delta}^{N} (z_{0}) = \exp_{z} (\{ \text{Graph} \, (f_{i}) \}) \cap B_{2 \delta}^{N} (z_{0}) \eqqcolon D_{i}$, the following signed distance functions are well-defined for $i = 1,2$,
    \begin{eqnarray*}
        \tilde{d}_{i} \colon B_{2 \delta}^{N} (z_{0}) &\rightarrow& \mathbb{R}, \\
        y &\mapsto& \begin{cases}
            d_{D_{i}} (y), & y \in E_{i}, \\
            - d_{D_{i}} (y), & y \in B_{2 \delta}^{N} (z_{0}) \setminus E_{i}.
        \end{cases} 
    \end{eqnarray*}
    For $y$ in $B_{\delta}^{N} (z_{0}) \subset \subset B_{2 \delta}^{N} (z_{0})$, 
    \begin{equation*}
        \tilde{d} (y) = \max \{ \tilde{d}_{1} (y), \tilde{d}_{2} (y) \}. 
    \end{equation*}
    Furthermore, by \cite[Proposition 4.2]{Mantegazza-Mennucci_2002}, we may choose $\delta > 0$, such that $\tilde{d}_{1}$ and $\tilde{d}_{2}$ will be smooth on $B^{N}_{2 \delta} (z_{0})$.
\end{rem}

For $i = 1, \, 2$, we define 
\begin{equation*}
    \tilde{D}_{i} \coloneqq \{ (z, \nu_{i} (z)) \colon z \in D_{i}\} \subset \tilde{M},
\end{equation*}
and points $x_{0}^{i} = (z_{0}, \nu_{i} (z_{0}))$.

\begin{rem}\label{rem: diffeos F i}
    We make a choice of $\delta = \delta (N, M, g, z_{0}) > 0$ small enough such that, for each $i = 1, \, 2$, we have open sets $\tilde{V}_{i} \subset \tilde{M} \times \mathbb{R}$, and maps, 
    \begin{equation*}
        F_{i} \colon \tilde{V}_{i} \rightarrow B_{2 \delta}^{N} (z_{0}),
    \end{equation*}
    such that $\tilde{D}_{i} = \tilde{V}_{i} \cap \{t = 0\}$, and  $F_{i} = F_{|\tilde{V}_{i}}$, is a diffeomorphism. 
    We also insist that $\delta = \delta (N, M, g, z_{0}) > 0$, is chosen small enough such that $\text{Cut}(D_{1})$ and $\text{Cut}(D_{2})$ are empty in $B_{2 \delta} (z_{0})$.
    We know we can pick such a $\delta > 0$ by \cite[Proposition 4.2]{Mantegazza-Mennucci_2002}
\end{rem}

By choice of $\delta > 0$ in Remark \ref{rem: diffeos F i}, and Proposition \ref{prop: Classification of points in Cut M}, 
    \begin{equation*}
        \text{Cut} (M) \cap B_{\delta}^{N} (z_{0}) = \{ y \in B_{\delta}^{N} (z_{0}) \colon \tilde{d}_{1} (y) = \tilde{d}_{2} (y) \} \subset B_{\delta}^{N} (z_{0}) \setminus E.
    \end{equation*}

\begin{rem}\label{rem: choice of delta for set d 1 = d 2}
    Denote the set, 
    \begin{equation*}
        A = \{ y \in B_{2 \delta}^{N} (z_{0}) \colon \tilde{d}_{1} (y) = \tilde{d}_{2} (y) \}.
    \end{equation*}
    For $i = 1, \, 2$, we consider the functions, 
    \begin{eqnarray*}
        \psi_{i} \colon \tilde{V}_{i} &\rightarrow& \mathbb{R}, \\
        (x,t) &\mapsto& \tilde{d}_{1} (F_{i} (x,t)) - \tilde{d}_{2}(F_{i} (x,t)).
    \end{eqnarray*}
    Therefore, $A = F_{i} (\{\psi_{i} = 0\})$.
    Moreover, 
    \begin{equation*}
        \partial_{t} \psi_{i} (x_{0}^{i}, 0) = dF_{i}^{-1} (\nabla \tilde{d}_{1} (z_{0})) - d F_{i}^{-1} (\nabla \tilde{d}_{2} (z_{0})) = 2 \partial_{t} \not= 0.
    \end{equation*}
    Thus, by Implicit Function Theorem we may choose $\delta = \delta (z_{0}, N, M, g) > 0$, such that set $A = \text{Cut} \, (M) \cap B_{2 \delta}^{N} (z_{0})$ is a smooth $n$-submanifold in $B_{2 \delta}^{N} (z_{0})$, and $\sigma^{-}$ is smooth on $\tilde{D}_{1} \cup \tilde{D}_{2}$.
\end{rem} 

\bigskip 

We now look to define the push out function to construct our competitor, Figure \ref{fig: local pictures showing construction of competitor in abstract cylinder}.

\bigskip

We wish to determine the amount we want to push out by, and the set we wish to push out on.
Fix $\rho > 0$, and we set $l = l(\rho)$, to be,
\begin{equation}\label{eqn: def of l}
    l (\rho) = \sup \{ t \colon \, for \, all \, x \, in \, B_{t}^{\tilde{M}}(x_{0}^{1}), \, |\sigma^{-}(x)| < \rho\}.
\end{equation}
Here, $B_{t}^{\tilde{M}}(x)$ is the geodesic ball in $\tilde{M}$, about point $x$, of radius $t$.
As $\sigma^{-}$ is smooth about $x_{0}^{1}$, and $\sigma^{-}(x_{0}^{1}) = 0$, this implies that $l(\rho) > 0$ for all $\rho > 0$.
Furthermore, $l (\rho)$ is increasing in $\rho$, implying that the limit of $l(\rho)$, as $\rho \rightarrow 0$, exists. 
Therefore, as $\sigma^{-}(x) = 0$ if and only if $\iota(x)$ is a non-embedded point, and such points have $\mathcal{H}^{n}$-measure 0 in $\tilde{M}$, we have that this limit must be 0.

\begin{rem}\label{rem: choice of delta so that sigma minus has lower quadratic bound}
    As $\sigma^{-}$ is smooth on $\tilde{D}_{1}$, $\sigma^{-} \leq 0$, and $\sigma^{-} (x_{0}^{1}) = 0$, then there exists a $C_{1} = C_{1} (N, M, g, z_{0}) < + \infty$, and a $\delta = \delta (N, M, g, z_{0})$, such that for all $x$ in $\tilde{D}_{1}$,
    \begin{equation*}
        \sigma^{-} (x) \geq - C_{1} d_{\tilde{M}}^{2} (x, x_{0}^{1}).
    \end{equation*}
\end{rem}

\bigskip

As $l(\rho) \rightarrow 0$, as $\rho \rightarrow 0$, this implies that we can choose $\rho > 0$, Remark \ref{rem: choice of rho so that B l fits into D 1}, such that 
\begin{equation*}
    B_{l (\rho)}^{\tilde{M}} (x_{0}^{1}) \subset \subset \tilde{D}_{1}.
\end{equation*} 
There exists an $x'$ in $\tilde{D}_{1}$, such that $d_{\tilde{M}} (x', x_{0}^{1}) = l$, and $\sigma^{-}(x') = - \rho$. 
Therefore, by Remark \ref{rem: choice of delta so that sigma minus has lower quadratic bound}, 
\begin{equation*}
    \rho \leq C_{1} l^{2}.
\end{equation*}

\begin{rem}\label{rem: choice of rho so that B l fits into D 1}
    Note that we have made our first choice of $\rho = \rho (z_{0}, N, M, g, \delta)$.
\end{rem}

We push out on disks $D_{1}$ and $D_{2}$ equally, so that they meet on the Cut Locus in $B_{\delta}^{N} (z_{0})$, which is our previously denoted set $A$, as seen in Figure \ref{fig: local pictures showing construction of competitor in abstract cylinder}.
We consider the open sets $\tilde{W}_{i} \subset \tilde{D}_{i}$, defined by, 
\begin{equation*}
    \tilde{W}_{i} = \{ x \colon F_{i}(x, \sigma^{-}(x)) \in B_{\delta}(z_{0})\}.
\end{equation*}
Clearly $x_{0}^{i}$ lies in $\tilde{W}_{i}$, therefore these sets are non-empty.
We can then define a diffeomorphism between $\tilde{W}_{1}$, and $\tilde{W}_{2}$. 
\begin{eqnarray*}
    \Psi \colon \tilde{W}_{1} &\rightarrow& \tilde{W}_{2}, \\
    x &\mapsto& (\pi \circ F_{2}^{-1} \circ F_{1} \circ (\text{id}, \sigma^{-})) (x),
\end{eqnarray*}
where we define, $\pi$ by, 
\begin{eqnarray*}
    \pi \colon \tilde{M} \times \mathbb{R} &\rightarrow& \tilde{M}, \\
    (x,t) &\mapsto& x,
\end{eqnarray*}
and $(\text{id}, \sigma^{-})$, by 
\begin{eqnarray*}
    (\text{id}, \sigma^{-}) \colon \tilde{M} &\rightarrow& \tilde{M} \times \mathbb{R}, \\
    x &\mapsto& (x, \sigma^{-} (x)).
\end{eqnarray*}
The function $\Psi$ is smooth and has smooth inverse given by  
\begin{eqnarray*}
    \Psi^{-1} \colon \tilde{W}_{2} &\rightarrow& \tilde{W}_{1}, \\
    x &\mapsto& (\pi \circ F_{1}^{-1} \circ F_{2} \circ (\text{id}, \sigma^{-})) (x).
\end{eqnarray*}
We note that, $d \Psi_{x_{0}^{1}} = Id$.

\begin{rem}\label{rem: choice of rho so that B 2 l x 0 1 lies in W 1}
    We choose $\rho = \rho (z_{0}, N, M, g, \delta) > 0$, such that, 
\begin{equation*}
    B_{2 l}^{\tilde{M}}(x_{0}^{1}) \subset \subset \tilde{W}_{1}.
\end{equation*}
\end{rem}

\bigskip

Consider a \textit{push out function}, which lies in $C_{c}^{\infty}(\tilde{D}_{1})$, and has the following properties,
\begin{equation*}
    f_{1} (x) = 
    \begin{cases}
        -1, & x \in B_{l}^{\tilde{M}}(x_{0}^{1}), \\
        [-1,0], & x \in B_{2l}^{\tilde{M}} (x_{0}^{1}) \setminus B_{l}^{\tilde{M}} (x_{0}^{1}), \\
        0, & x \in \tilde{D}_{1} \setminus B_{2 l}^{\tilde{M}}(x_{0}^{1}).
    \end{cases}
\end{equation*} 
We further impose the condition, 
\begin{equation*}
    |\nabla f_{1}| \leq \frac{2}{l}.
\end{equation*}
Define $f_{2}$ in $C^{\infty}_{c}(\tilde{D}_{2})$, by 
\begin{equation*}
    f_{2} (x) = \begin{cases} 
        (f_{1} \circ \Psi^{-1}) (x), & x \in \tilde{W}_{2}, \\
        0, & x \in \tilde{D}_{2} \setminus \tilde{W}_{2}.
    \end{cases}
\end{equation*}
The support of $f_{2}$ will lie in $\Psi(B_{2l}^{\tilde{M}} (x_{0}^{1})) \subset \subset \Psi (\tilde{W}_{1}) = \tilde{W}_{2}$.
We then define the function $f$ in $C_{c}^{\infty} (\tilde{M})$, by $f = f_{1} + f_{2}$. 

\bigskip

Define the sets, 
\begin{eqnarray*}
    B_{2 l} &=& B_{2 l}^{\tilde{M}} (x_{0}^{1}) \cup \Psi (B_{2 l}^{\tilde{M}} (x_{0}^{1})), \\
    B_{l} &=& B_{l}^{\tilde{M}} (x_{0}^{1}) \cup \Psi (B_{l}^{\tilde{M}} (x_{0}^{1})), \\
    A_{l} &=& B_{2l} \setminus B_{l}.
\end{eqnarray*}
We will similarly define the sets, 
\begin{equation*}
    B_{t} = B_{t}^{\tilde{M}} (x_{0}^{1}) \cup \Psi (B_{t}^{\tilde{M}} (x_{0}^{1})),
\end{equation*}
for $t > 0$, such that $B_{t}^{\tilde{M}} (x_{0}^{1}) \subset \tilde{W}_{1}$.

\bigskip 

\begin{rem}\label{rem: choice of rho so that B l times -rho rho fits into B delta}
    We choose $\rho = \rho (z_{0}, M, N, g, \delta, W, \lambda) > 0$, such that 
    \begin{equation*}
        F_{1} \left( B_{2 l} \times (- 2 \rho, 2 \rho) \right) \subset \subset B_{\delta}^{N}(z_{0})
    \end{equation*}
\end{rem}

\bigskip 

We now look to define the function that will 'push out away from non-embedded point'. 
This function will define the path from (2) to (3) in Figure \ref{fig: The Path}.

\begin{rem}\label{rem: Choice of L and r 0 to fit in ball B delta}
    (Choice of $L$ and $r_{0}$)
    We choose $L = L(z_{0}, N, M, g, \delta) > 0$ and $r_{0} = r_{0} (z_{0}, N, M, g, \delta) > 0$, such that, 
    \begin{equation*}
        B_{L}^{\tilde{M}} (x_{0}^{1}) \subset \subset \tilde{W}_{1},
    \end{equation*}
    and, 
    \begin{equation*}
        F \left( B_{L} \times (- 2 r_{0}, 2 r_{0}) \right) \subset \subset B_{\delta}^{N} (z_{0}).
    \end{equation*}
\end{rem}

\bigskip 

For a sets $\tilde{\Omega}$ and $\Omega$, were $\Omega$ is open and $\tilde{\Omega} \subset \subset \Omega$, we define the $2$-Capacity of $\tilde{\Omega}$ in $\Omega$ as the value, 
\begin{equation*}
    \text{Cap}_{2} (\tilde{\Omega}, \Omega) = \inf \left\{ \int_{\Omega} |\nabla \varphi|^{2} \, d \mathcal{H}^{n} \colon \varphi \in C_{c}^{\infty} (\Omega), \, \varphi \geq \chi_{\tilde{\Omega}} \right\}.
\end{equation*}
For $n \geq 3$, by \cite[Theorem 4.15 (ix), Section 4.7.1 and Theorem 4.16, Section 4.7.2]{EG1991measure}, 
\begin{equation*}
    \lim_{k \rightarrow \infty} \text{Cap}_{2} (B_{\frac{L}{k}}^{\tilde{M}} (x_{0}^{1}), B_{L}^{\tilde{M}} (x_{0}^{1})) = \text{Cap}_{2} (\{x_{0}^{1}\}, B_{L}^{\tilde{M}} (x_{0}^{1})) = 0.
\end{equation*}
Identical proofs show that this also holds for $n = 2$. 
Therefore, for all $\gamma > 0$, there exists a function $\varphi_{\gamma, k}$, such that, 
\begin{equation*}
    \begin{cases}
        \varphi_{\gamma, k} \in C_{c}^{\infty} (B_{L}^{\tilde{M}} (x_{0}^{1})), \\
        \varphi_{\gamma, k} \colon \tilde{M} \rightarrow [0,1], \\
        \varphi_{\gamma, k} (x) = 1, \, x \in B_{\frac{L}{k}}^{\tilde{M}} (x_{0}^{1}),
    \end{cases}
\end{equation*}
and, defining $\tilde{\varphi}_{\gamma, k} = \varphi_{\gamma, k} + \varphi_{\gamma, k} \circ \Psi^{-1}$, we have
\begin{equation*}
    \int_{\tilde{M}} |\nabla \tilde{\varphi}_{\gamma, k}|^{2} \, d \mathcal{H}^{n} (x) < \gamma.
\end{equation*}
We consider the function $\tilde{f} = 1 - \tilde{\varphi}_{\gamma, k}$ in $C^{\infty} (\tilde{M})$, and $\|\nabla \tilde{f}\|_{L^{2} (\tilde{M})}^{2} < \gamma$. 

\begin{rem}\label{rem: later choices of r 0 and L and gamma and k}
    We will later make fixed choices for $L = L (z_{0}, M, N, g, \delta, W, \lambda)$, $r_{0} = r_{0} (z_{0}, M, N, g, \delta, W, \lambda, L)$, $\gamma = \gamma (z_{0}, N, M, g, \delta, r_{0}, L)$, and $k = k(z_{0}, N, M, g, \delta, L, \gamma)$.
\end{rem}

\begin{rem}\label{rem: choice of rho based on r 0 and L}
    We make a further choice of $\rho = \rho (z_{0}, N, M, g, \delta, L, r_{0}, k)$, such that, 
    \begin{equation*}
        B_{2l} \subset \subset B_{\frac{L}{k}},
    \end{equation*}
    We will make a further choice of $\rho$ later on, so that $\rho = \rho (z_{0}, N, M, g, \delta, L, k, r_{0})$.
\end{rem}

\subsection{Approximating Function for CMC}\label{subsec: Approximating function for CMC}

We use the tools we have constructed to give a simple proof that function, 
\begin{equation*}
    v_{\varepsilon} (y) = \overHe_{\varepsilon} (\tilde{d} (y)),
\end{equation*}
is suitable approximation of $M$, i.e. 
\begin{equation*}
    \lim_{\varepsilon \rightarrow 0} \mathcal{F}_{\varepsilon, \lambda} (v_{\varepsilon}) = 2 \sigma \mathcal{H}^{n} (M) - \sigma \lambda \mu_{g} (E) - \sigma \lambda \mu_{g} (N \setminus E). 
\end{equation*}
By the Co-Area formula on the function $\tilde{d}$, 
\begin{eqnarray*}
    \mathcal{F}_{\varepsilon, \lambda} (v_{\varepsilon}) &=& \int_{N} \frac{\varepsilon}{2} |\nabla v_{\varepsilon}|^{2} + \frac{W(v_{\varepsilon})}{2} - \sigma \lambda \int_{N} v_{\varepsilon}, \\
    &=& \int_{\mathbb{R}} \int_{\Gamma_{t}} Q_{\varepsilon} (t) \, d\mathcal{H}^{n} \, dt - \sigma \lambda \int_{\mathbb{R}} \int_{\Gamma_{t}} \overHe_{\varepsilon} (t) \, d \mathcal{H}^{n} \, dt,
\end{eqnarray*}
where, 
\begin{equation*}
    Q_{\varepsilon} (t) = \frac{\varepsilon}{2} ((\overHe_{\varepsilon})'(t))^{2} + \frac{W(\overHe_{\varepsilon} (t))}{\varepsilon}.
\end{equation*}
Using the fact that $N \setminus F(\tilde{T})$ is a set of 0 $\mu_{g}$-measure, and that $F \colon (\tilde{T}, h) \rightarrow (F(\tilde{T}), g)$, is a bijective local isometry, we have, 
\begin{equation*}
    \mathcal{F}_{\varepsilon, \lambda} (v_{\varepsilon}) = \int_{\mathbb{R}} Q_{\varepsilon} (t) \, \mathcal{H}^{n} (\tilde{\Gamma}_{t}) \, dt - \sigma \lambda \int_{\mathbb{R}} \overHe_{\varepsilon} (t) \, \mathcal{H}^{n} (\tilde{\Gamma}_{t}) \, dt. 
\end{equation*}
From Analysis of $\overHe_{\varepsilon}$, we have that, $\text{supp} \, Q_{\varepsilon} \subset [- 2 \varepsilon \Lambda, 2 \varepsilon \Lambda]$, and
\begin{equation*}
    2 \sigma - \beta \varepsilon^{2} \leq \int_{\mathbb{R}} Q_{\varepsilon} (t) \, dt \leq 2 \sigma + \beta \varepsilon^{2}.
\end{equation*}
Furthermore, 
\begin{equation*}
    \overHe_{\varepsilon} (t) \leq \begin{cases}
        1, & t > - 2 \varepsilon \Lambda, \\
        -1, & t \leq - 2 \varepsilon \Lambda,
    \end{cases}
\end{equation*}
and, 
\begin{equation*}
    \overHe_{\varepsilon} (t) \geq \begin{cases}
        1, & t > 2 \varepsilon \Lambda, \\
        -1, & t \leq 2 \varepsilon \Lambda.
    \end{cases}
\end{equation*}
Therefore, 
\begin{equation*}
    \mathcal{F}_{\varepsilon, \lambda} (v_{\varepsilon}) \leq (2 \sigma + \beta \varepsilon^{2}) \esssup_{t \in [-2 \varepsilon \Lambda, 2 \varepsilon \Lambda]} \mathcal{H}^{n} (\tilde{\Gamma}_{t}) - \sigma \lambda \int_{2 \varepsilon \Lambda}^{+ \infty} \mathcal{H}^{n} (\tilde{\Gamma}_{t}) \, dt + \sigma \lambda \int_{- \infty}^{2 \varepsilon \Lambda} \mathcal{H}^{n} (\tilde{\Gamma}_{t}) \, dt.
\end{equation*}
Similarly, 
\begin{equation*}
    \mathcal{F}_{\varepsilon, \lambda} (v_{\varepsilon}) \geq (2 \sigma - \beta \varepsilon^{2}) \essinf_{t \in [-2 \varepsilon \Lambda, 2 \varepsilon \Lambda]} \mathcal{H}^{n} (\tilde{\Gamma}_{t}) - \sigma \lambda \int_{- 2 \varepsilon \Lambda}^{+ \infty} \mathcal{H}^{n} (\tilde{\Gamma}_{t}) \, dt + \sigma \lambda \int_{- \infty}^{2 \varepsilon \Lambda} \mathcal{H}^{n} (\tilde{\Gamma}_{t}) \, dt.
\end{equation*}
We have, 
\begin{equation*}
    \mathcal{H}^{n} (\tilde{\Gamma}_{t}) = \int_{\tilde{M}} \theta_{t} (x) \, d \mathcal{H}^{n} (x),
\end{equation*}
and by applying Dominated Convergence Theorem to $\theta_{t}$, we have that, 
\begin{equation*}
    \lim_{t \rightarrow 0} \mathcal{H}^{n} (\tilde{\Gamma}_{t}) = \lim_{t \rightarrow 0} \int_{\tilde{M}} \theta_{t} (x) \, d \mathcal{H}^{n} (x) = \mathcal{H}^{n} (\tilde{M} \cap \tilde{T}) = \mathcal{H}^{n} (\tilde{M}).
\end{equation*}
This implies that, 
\begin{equation*}
    \lim_{\varepsilon \rightarrow 0} \esssup_{t \in [-2 \varepsilon \Lambda, 2 \varepsilon \Lambda]} \mathcal{H}^{n} (\tilde{\Gamma}_{t}) = \mathcal{H}^{n}(\tilde{M}) = \mathcal{H}^{n} (M),
\end{equation*}
and, 
\begin{equation*}
    \lim_{\varepsilon \rightarrow 0} \essinf_{t \in [-2 \varepsilon \Lambda, 2 \varepsilon \Lambda]} \mathcal{H}^{n} (\tilde{\Gamma}_{t}) = \mathcal{H}^{n}(\tilde{M}) = \mathcal{H}^{n} (M).
\end{equation*}

\bigskip 

The function $t \mapsto \mathcal{H}^{n} (\tilde{\Gamma}_{t})$ is measurable, implying that, 
\begin{equation*}
    \lim_{\varepsilon \rightarrow 0} \int_{\pm 2 \varepsilon \Lambda}^{+ \infty} \mathcal{H}^{n} (\tilde{\Gamma}_{t}) \, dt = \int_{0}^{+ \infty} \mathcal{H}^{n} (\tilde{\Gamma}_{t}) \, dt = \mathcal{H}^{n+1} (\{y \in N \colon \tilde{d} (y) > 0\}) = \mu_{g} (E),
\end{equation*}
and, 
\begin{equation*}
    \lim_{\varepsilon \rightarrow 0} \int^{\pm 2 \varepsilon \Lambda}_{- \infty} \mathcal{H}^{n} (\tilde{\Gamma}_{t}) \, dt = \int^{0}_{- \infty} \mathcal{H}^{n} (\tilde{\Gamma}_{t}) \, dt = \mathcal{H}^{n + 1} (\{y \in N \colon \tilde{d} (y) < 0\}) = \mu_{g} (N \setminus E).
\end{equation*}

\bigskip 

Therefore we have, 
\begin{equation*}
    \lim_{\varepsilon \rightarrow 0} \mathcal{F}_{\varepsilon, \lambda} (v_{\varepsilon}) = 2 \sigma \mathcal{H}^{n} (M) - \sigma \lambda \mu_{g} (E) + \sigma \lambda \mu_{g} (N \setminus E).
\end{equation*}

\section{Base Computation}\label{sec: Baseline Computation}
 
Consider a smooth function, 
\begin{equation*}
    g \colon \mathbb{R} \times \tilde{M}  \rightarrow \mathbb{R}.
\end{equation*}
and define the following
\begin{eqnarray*}
    v_{\varepsilon}^{r,g} \colon \tilde{M} \times \mathbb{R} &\rightarrow& \mathbb{R}, \\
    (x,t) &\mapsto& \overHe_{\varepsilon}(t - g(r,x)).
\end{eqnarray*}
By Gauss Lemma, 
\begin{equation*}
    |\nabla v_{\varepsilon}^{r,g}(x,t)|^{2} = ( (\overHe_{\varepsilon})'(t - g(r,x)))^{2} (1 + |\nabla_{x} g(r,x)|_{(x,t)}^{2}).
\end{equation*}
By the co-area formula on $p$, 
\begin{eqnarray*}
    \mathcal{F}_{\varepsilon, \lambda} (v_{\varepsilon}^{r,g}) &=& \int_{\tilde{T}} \frac{\varepsilon}{2} |\nabla v_{\varepsilon}^{r,g}|^{2} + \frac{W(v_{\varepsilon}^{r,g})}{\varepsilon} - \sigma \lambda v_{\varepsilon}^{r,g} \, d \mu_{h}, \\
    &=& \int_{\mathbb{R}} \int_{\tilde{\Gamma}_{t}} \frac{\varepsilon}{2} ( (\overHe_{\varepsilon})'(t - g(r,x)))^{2}|\nabla_{x} g(r,x)|_{(x,t)}^{2} \, d \mathcal{H}^{n}(x,t) \, dt \\
    && + \int_{\mathbb{R}} \int_{\tilde{\Gamma}_{t}} \frac{\varepsilon}{2} ( (\overHe_{\varepsilon})'(t - g(r,x)))^{2} + \frac{W(\overHe_{\varepsilon}(t - g(r,x)))}{\varepsilon} - \sigma \lambda \overHe_{\varepsilon} (t - g(r,x)) \, d \mathcal{H}^{n} (x,t) \, dt, \\
    &=& \int_{\tilde{M}} \int_{\sigma^{-}(x)}^{\sigma^{+}(x)} \frac{\varepsilon}{2} ( (\overHe_{\varepsilon})'(t - g(r,x)))^{2}|\nabla_{x} g(r,x)|_{(x,t)}^{2} \, \theta_{t}(x) \, dt \, d \mathcal{H}^{n}(x) \\
    && \hspace{1cm} + \int_{\tilde{M}} \int_{\sigma^{-}(x)}^{\sigma^{+}(x)} \bigg(\frac{\varepsilon}{2} ( (\overHe_{\varepsilon})'(t - g(r,x)))^{2} + \frac{W(\overHe_{\varepsilon}(t - g(r,x)))}{\varepsilon} \\
    && \hspace{3cm} - \sigma \lambda \overHe_{\varepsilon} (t - g(r,x)) \bigg) \, \theta_{t} (x) \, dt \, d \mathcal{H}^{n} (x), 
\end{eqnarray*}
In the last equality we use Fubini's Theorem to switch the integrals. 

\bigskip 

We have, 
\begin{eqnarray*}
    \mathcal{F}_{\varepsilon, \lambda} (v_{\varepsilon}^{r,g}) - \mathcal{F}_{\varepsilon, \lambda} (v_{\varepsilon}^{0,g}) &=& \int_{\tilde{M}} \int_{\sigma^{-}(x)}^{\sigma^{+}(x)} \frac{\varepsilon}{2} ( (\overHe_{\varepsilon})'(t - g(r,x)))^{2}|\nabla_{x} g(r,x)|_{(x,t)}^{2} \, \theta_{t}(x) \, dt \, d \mathcal{H}^{n}(x) \\
    && \hspace{1cm} - \int_{\tilde{M}} \int_{\sigma^{-}(x)}^{\sigma^{+}(x)} \frac{\varepsilon}{2} ( (\overHe_{\varepsilon})'(t - g(r,x)))^{2}|\nabla_{x} g(0,x)|_{(x,t)}^{2} \, \theta_{t}(x) \, dt \, d \mathcal{H}^{n}(x), \\
    && + \int_{\tilde{M}} \int_{\sigma^{-}(x)}^{\sigma^{+}(x)} (Q_{\varepsilon} (t - g(r,x)) - Q_{\varepsilon} (t - g(0,x))) \theta_{t}(x) \, dt \, d \mathcal{H}^{n}(x) \\
    && - \int_{\tilde{M}} \int_{\sigma^{-}(x)}^{\sigma^{+}(x)} \sigma \lambda (\overHe_{\varepsilon} (t - g(r,x)) - \overHe_{\varepsilon} (t - g(0,x))) \, \theta_{t} (x) \, dt \, d \mathcal{H}^{n} (x), 
\end{eqnarray*}
We have the following two terms,
\begin{eqnarray*}
    \rom{1}_{\varepsilon}^{r,g} &=& \int_{\tilde{M}} \int_{\sigma^{-}(x)}^{\sigma^{+}(x)} \frac{\varepsilon}{2} ( (\overHe_{\varepsilon})'(t - g(r,x)))^{2}|\nabla_{x} g(r,x)|_{(x,t)}^{2} \, \theta_{t}(x) \, dt \, d \mathcal{H}^{n}(x) \\
    && \hspace{1cm} - \int_{\tilde{M}} \int_{\sigma^{-}(x)}^{\sigma^{+}(x)} \frac{\varepsilon}{2} ( (\overHe_{\varepsilon})'(t - g(r,x)))^{2}|\nabla_{x} g(0,x)|_{(x,t)}^{2} \, \theta_{t}(x) \, dt \, d \mathcal{H}^{n}(x), 
\end{eqnarray*}
and, by Fundamental Theorem of Calculus and Fubini's Theorem,
\begin{eqnarray*}
    \rom{2}_{\varepsilon}^{r,g} &=& \int_{\tilde{M}} \int_{\sigma^{-}(x)}^{\sigma^{+}(x)} (Q_{\varepsilon} (t - g(r,x)) - Q_{\varepsilon} (t - g(0,x))) \theta_{t}(x) \, dt \, d \mathcal{H}^{n}(x) \\
    && - \int_{\tilde{M}} \int_{\sigma^{-}(x)}^{\sigma^{+}(x)} \sigma \lambda (\overHe_{\varepsilon} (t - g(r,x)) - \overHe_{\varepsilon} (t - g(0,x))) \, \theta_{t} (x) \, dt \, d \mathcal{H}^{n} (x), \\
    &=& - \int_{0}^{r} \int_{\tilde{M}} \partial_{s} g (s,x) \int_{\sigma^{-}(x)}^{\sigma^{+}(x)} Q_{\varepsilon}'(t - g(s,x)) \, \theta_{t}(x) \, dt \, d\mathcal{H}^{n}(x) \, ds \\
    && + \int_{0}^{r} \int_{\tilde{M}} \partial_{s} g (s,x) \int_{\sigma^{-}(x)}^{\sigma^{+}(x)} \sigma \lambda (\overHe_{\varepsilon})'(t - g(s,x)) \, \theta_{t}(x) \, dt \, d\mathcal{H}^{n}(x) \, ds, \\
    &=& - \int_{0}^{r} \int_{\tilde{M}} \partial_{s} g (s,x) Q_{\varepsilon}(\sigma^{+}(x) - g(s,x)) \, \theta^{+}(x) \, d\mathcal{H}^{n}(x) \, ds \\
    && + \int_{0}^{r} \int_{\tilde{M}} \partial_{s} g (s,x) Q_{\varepsilon}(\sigma^{-}(x) - g(s,x)) \, \theta^{-}(x) \, d\mathcal{H}^{n}(x) \, ds \\
    && + \int_{0}^{r} \int_{\tilde{M}} \partial_{s} g (s,x) \int_{\sigma^{-}(x)}^{\sigma^{+}(x)} Q_{\varepsilon}(t - g(s,x)) \, \partial_{t}\theta_{t}(x) \, dt \, d\mathcal{H}^{n}(x) \, ds \\
    &&  + \int_{0}^{r} \int_{\tilde{M}} \partial_{s} g (s,x) \int_{\sigma^{-}(x)}^{\sigma^{+}(x)} \sigma \lambda (\overHe_{\varepsilon})'(t - g(s,x)) \, \theta_{t}(x) \, dt \, d\mathcal{H}^{n}(x) \, ds, \\
    &=& - \int_{0}^{r} \int_{\tilde{M}} \partial_{s} g (s,x) Q_{\varepsilon}(\sigma^{+}(x) - g(s,x)) \, \theta^{+}(x) \, d\mathcal{H}^{n}(x) \, ds \\
    && + \int_{0}^{r} \int_{\tilde{M}} \partial_{s} g (s,x) Q_{\varepsilon}(\sigma^{-}(x) - g(s,x)) \, \theta^{-}(x) \, d\mathcal{H}^{n}(x) \, ds \\
    && + \int_{0}^{r} \int_{\tilde{M}} \partial_{s} g (s,x) \int_{\sigma^{-}(x)}^{\sigma^{+}(x)} Q_{\varepsilon}(t - g(s,x)) \, (\lambda - H_{t}(x)) \, \theta_{t}(x) \, dt \, d\mathcal{H}^{n}(x) \, ds \\
    &&  + \lambda \int_{0}^{r} \int_{\tilde{M}} \Theta_{\varepsilon, g}^{1} (s,x) - \Theta_{\varepsilon, g}^{2} \, d\mathcal{H}^{n}(x) \, ds, \\
\end{eqnarray*}
Where, 
\begin{eqnarray*}
    \theta^{+} (x) &=& \lim_{t \nearrow \sigma^{+}(x)} \theta_{t}(x), \\
    \theta^{-} (x) &=& \lim_{t \searrow \sigma^{-}(x)} \theta_{t}(x), \\
    \Theta_{\varepsilon, g}^{1} (s,x) &=& \sigma \int_{\sigma^{-}(x)}^{\sigma^{+}(x)} \partial_{s} g(s,x) (\overHe_{\varepsilon})' (t - g(s,x)) \theta_{t}(x) \, dt, \\
    \Theta_{\varepsilon, g}^{2} (s,x) &=& \int_{\sigma^{-}(x)}^{\sigma^{+}(x)} \partial_{s} g(s,x) Q_{\varepsilon} (t - g(s,x)) \theta_{t} (x) \, dt.
\end{eqnarray*}
For the last equality of $\rom{2}_{\varepsilon}^{r,g}$ we are using $\partial_{t} \theta_{t} (x) = - H_{t} (x) \theta_{t} (x)$, for $t$ in $(\sigma^{-}(x), \sigma^{+} (x) )$.

\section{Competitor}\label{sec: Path From Disks to Crab}

\subsection{Calculation on $\tilde{M} \times \mathbb{R}$}\label{subsec: Path From Disks to Crab, calculation}

Here we construct the path in Figure \ref{fig: The Path} from (1) to (2).

\bigskip

Set $g_{1}(r,x) = r f(x)$, take $r$ in $[0, \rho]$, where $\rho \in (0,1)$ will be chosen later and $f \colon \tilde{M} \rightarrow \mathbb{R}$ as defined in Section \ref{subsec: construction about non-embed point}.

\begin{rem}\label{rem: choice in epsilon 1}
    (Choice in $\varepsilon_{1}$)
    We choose $\varepsilon_{1} = \varepsilon_{1} (\rho) \in (0, 1/4)$, such that, 
    \begin{equation*}
        2 \varepsilon_{1} \Lambda = 6 \varepsilon_{1} |\log \, \varepsilon_{1}| < < \rho.
    \end{equation*} 
    From here we consider $\varepsilon$ in $(0, \varepsilon_{1})$.
\end{rem}

We have, 
\begin{equation}\label{eqn: rom 2 for g 1}
    \begin{split}
    \rom{2}_{\varepsilon}^{r,g_{1}} = & - \int_{0}^{r} \int_{\tilde{M}} f(x) Q_{\varepsilon}(\sigma^{+}(x) - sf(x)) \, \theta^{+}(x) \, d\mathcal{H}^{n}(x) \, ds \\
    & + \int_{0}^{r} \int_{\tilde{M}} f(x) Q_{\varepsilon}(\sigma^{-}(x) - sf(x)) \, \theta^{-}(x) \, d\mathcal{H}^{n}(x) \, ds \\
    & + \int_{0}^{r} \int_{\tilde{M}} f(x) \int_{\sigma^{-}(x)}^{\sigma^{+}(x)} Q_{\varepsilon}(t - sf(x)) \, (\lambda - H_{t}(x)) \, \theta_{t}(x) \, dt \, d\mathcal{H}^{n}(x) \, ds \\
    & + \lambda \int_{0}^{r} \int_{\tilde{M}} \Theta_{\varepsilon, g}^{1} (s,x) - \Theta_{\varepsilon, g}^{2} \, d\mathcal{H}^{n}(x) \, ds, 
    \end{split}
\end{equation}

\bigskip 

Concentrate on the second term of the right-hand side of (\ref{eqn: rom 2 for g 1}).
As the integrand is non-positive, $f = -1$ on $B_{l}$ and $\text{supp} \, Q_{\varepsilon} \subset [-2 \varepsilon \Lambda, 2 \varepsilon \Lambda]$, we have
\begin{eqnarray*}
    && \int_{0}^{r} \int_{\tilde{M}} f(x) Q_{\varepsilon}(\sigma^{-}(x) - sf(x)) \, \theta^{-}(x) \, d\mathcal{H}^{n}(x) \, ds \\
    && \hspace{2cm} \leq - (2 \sigma - \beta \varepsilon^{2}) \int_{B_{l} \cap \{- r + 2 \varepsilon \Lambda \leq \sigma^{-}(x) \leq - 2 \varepsilon \Lambda\} } \theta^{-} (x) \, d \mathcal{H}^{n}(x)
\end{eqnarray*}
We look for lower bounds on $\theta^{-}$.

\begin{rem}\label{rem: choice of delta for upper and lower mean curvature bounds in ball about z 0}
    Choose $\delta = \delta (z_{0}, N, M, g) > 0$, such that, 
\begin{equation*}
    \min_{y \in B_{\delta}^{N} (z_{0})} \{ \Delta \tilde{d}_{1} (y), \Delta \tilde{d}_{2} (y) \} \geq \frac{\lambda}{2},
\end{equation*}
and, 
\begin{equation*}
    \max_{y \in B_{\delta}^{N} (z_{0})} \{ \Delta \tilde{d}_{1} (y), \Delta \tilde{d}_{2} (y) \} \leq \frac{3\lambda}{2}.
\end{equation*}
\end{rem}

\bigskip 

Therefore, for $(x,t)$ in $\tilde{T}$, such that, $F(x,t)$ lies in $B_{\delta}^{N} (z_{0})$, we have that, 
\begin{equation*}
    \frac{\lambda}{2} \leq H_{t} (x) \leq \frac{3 \lambda}{2}.
\end{equation*}
Thus, by similar calculations carried out in Section \ref{subsec: Area Element}, for all $(x,t)$ in $\tilde{T}$, such that $F(x,t)$ lies in $B_{\delta}^{N} (z_{0})$, we have,
\begin{equation*}
    \theta_{t} (x) \geq \begin{cases} 
        e^{- \frac{3\lambda t}{2}}, & t \geq 0, \\
        e^{- \frac{\lambda t}{2}}, & t \leq 0.
    \end{cases}
\end{equation*}

For $x$ in $B_{l}$, we have $\sigma^{-} (x) > - \rho$, and by choice of $\rho$ in Remark \ref{rem: choice of rho so that B l times -rho rho fits into B delta}, we have that $F(\{x\} \times (\sigma^{-}(x) ,0)) \subset B_{\delta}^{N} (z_{0})$.
Thus,
\begin{equation*}
    \theta^{-} (x) = \lim_{t \searrow \sigma^{-}(x)} \theta_{t} (x) \geq e^{- \frac{\lambda \, \sigma^{-} (x)}{2}} \geq 1,
\end{equation*}
for all $x$ in $B_{l}$.
Therefore, 
\begin{eqnarray*}
    && \int_{0}^{r} \int_{\tilde{M}} f(x) Q_{\varepsilon}(\sigma^{-}(x) - sf(x)) \, \theta^{-}(x) \, d\mathcal{H}^{n}(x) \, ds \\
    && \hspace{2cm} \leq - 2 \sigma \mathcal{H}^{n} (\{ x \colon x \in B_{l}, \, - r + 2 \varepsilon \Lambda \leq \sigma^{-}(x) \leq - 2 \varepsilon \Lambda \}) + C_{2} \varepsilon^{2},
\end{eqnarray*}
for $C_{2} = C_{2} (N, M, g, \lambda, W) < + \infty$.
This is a lower bound for the area deleted in pushing the disks together. 

\bigskip 

Concentrate on First term on the right-hand side of (\ref{eqn: rom 2 for g 1}).
By choice of $\delta > 0$ in Remark \ref{rem: diffeos F i} and $\rho > 0$ in Remark \ref{rem: choice of rho so that B l times -rho rho fits into B delta} we have that for $x$ in $ \text{supp} \, (f) \subset B_{2l}$, $\sigma^{+} (x) > 2 \rho > > 2 \varepsilon \Lambda$. 
Thus,
\begin{equation*}
    \int_{0}^{r} \int_{\tilde{M}} f(x) Q_{\varepsilon}(\sigma^{+}(x) - sf(x)) \, \theta^{+}(x) \, d\mathcal{H}^{n}(x) \, ds = 0.
\end{equation*}

\bigskip 

Concentrate on the third term on the right-hand side of (\ref{eqn: rom 2 for g 1}).
Consider $s > 0$, and $x$ in $\tilde{M}$, such that $s f (x) < -2 \varepsilon \Lambda$. 
Using the fact that $\text{supp} \, Q_{\varepsilon} \subset [-2 \varepsilon \Lambda, 2 \varepsilon \Lambda]$, and the inequalities on $H_{t}$ in Remark \ref{rem: Bounds on mean curvature of level sets},
\begin{eqnarray*}
    \int_{\sigma^{-}(x)}^{\sigma^{+}(x)} Q_{\varepsilon}(t - sf(x)) \, (\lambda - H_{t}(x)) \, \theta_{t}(x) \, dt &=& \int_{- 2 \varepsilon \Lambda}^{2 \varepsilon \Lambda} Q_{\varepsilon} (\xi) (\lambda - H_{\xi + s f (x)}) \theta_{\xi + s f(x)} d \xi, \\
    &\geq& 0.
\end{eqnarray*}
For $s f (x) \geq -2 \varepsilon \Lambda$, we have, 
\begin{eqnarray*}
    \int_{\sigma^{-}(x)}^{\sigma^{+}(x)} Q_{\varepsilon}(t - sf(x)) \, (\lambda - H_{t}(x)) \, \theta_{t}(x) \, dt &=& \int_{- 2 \varepsilon \Lambda}^{2 \varepsilon \Lambda} Q_{\varepsilon} (\xi) (\lambda - H_{\xi + s f (x)}) \theta_{\xi + s f(x)} d \xi, \\
    &\geq& C_{2} \min_{t \in [-4 \varepsilon \Lambda, 2 \varepsilon \Lambda]} (\lambda - H_{t}(x)) \theta_{t}(x), 
\end{eqnarray*}
potentially rechoosing $C_{2} = C_{2} (M, N, g, \lambda, W)$.
Therefore, we have that for all $r$ in $[0, \rho]$, 
\begin{eqnarray*}
    \rom{2}_{\varepsilon}^{r,g_{1}} &\leq& - 2 \sigma \mathcal{H}^{n} (\{ x \colon x \in B_{l}, \, - r + 2 \varepsilon \Lambda \leq \sigma^{-}(x) \leq - 2 \varepsilon \Lambda \}) \\
    && + C_{2} \left( r \int_{B_{2 l}} q_{\varepsilon}^{1} (x) \, d \mathcal{H}^{n}(x) + \int_{0}^{r} \int_{\tilde{M}} \Theta_{\varepsilon, g}^{1} (s,x) - \Theta_{\varepsilon, g}^{2} \, d\mathcal{H}^{n}(x) \, ds + \varepsilon^{2} \right).
\end{eqnarray*}
where, 
\begin{equation*}
    q^{1}_{\varepsilon} (x) = \max_{t \in [-4 \varepsilon \Lambda, 2 \varepsilon \Lambda]} (H_{t} (x) - \lambda) \theta_{t} (x) \geq 0,
\end{equation*}
and we have potentially rechosen $C_{2} = C_{2} (M, N, g, \lambda, W)$.
Therefore, for $r < 4 \varepsilon \Lambda$, 
\begin{equation*}
    \rom{2}_{\varepsilon}^{r,g_{1}} = C_{2} \bigg( r \int_{B_{2 l}} q_{\varepsilon}^{1} (x) \, d \mathcal{H}^{n}(x) + \int_{0}^{r} \int_{\tilde{M}} \Theta_{\varepsilon, g}^{1} (s,x) - \Theta_{\varepsilon, g}^{2} \, d\mathcal{H}^{n}(x) \, ds + \varepsilon^{2} \bigg),
\end{equation*}
and for $r \geq 4 \varepsilon \Lambda$, 
\begin{eqnarray*}
    \rom{2}_{\varepsilon}^{r,g_{1}} &\leq&  - 2 \sigma \mathcal{H}^{n} (\{ x \colon x \in B_{l}, \, -r + 2 \varepsilon \Lambda \leq \sigma^{-}(x) \leq 0 \}) \\
    && + C_{2} \bigg( \mathcal{H}^{n} ( \{ x \colon x \in B_{l}, -2 \varepsilon \Lambda < \sigma^{-}(x) \leq 0 \}) \\
    && + \int_{B_{2 l}} q_{\varepsilon}^{1} (x) \, d \mathcal{H}^{n}(x) + \int_{0}^{r} \int_{\tilde{M}} \Theta_{\varepsilon, g}^{1} (s,x) - \Theta_{\varepsilon, g}^{2} \, d\mathcal{H}^{n}(x) \, ds + \varepsilon^{2} \bigg).
\end{eqnarray*}
Again we have potentially rechoosing $C_{2} = C_{2} (M, N, g, \lambda, W)$.

\bigskip 

We now turn our attention to the term, 
\begin{eqnarray*}
    \rom{1}_{\varepsilon}^{r,g_{1}} &=& \int_{\tilde{M}} \int_{\sigma^{-}(x)}^{\sigma^{+}(x)} \frac{\varepsilon}{2} ( (\overHe_{\varepsilon})'(t - r f (x)))^{2}|r \nabla f(x)|_{(x,t)}^{2} \, \theta_{t}(x) \, dt \, d \mathcal{H}^{n}(x), \\
    &=& \int_{\tilde{M}} \int_{- 2 \varepsilon \Lambda}^{2 \varepsilon \Lambda} \frac{\varepsilon}{2} ( (\overHe_{\varepsilon})'(\xi) )^{2}|r \nabla f(x)|_{(x,\xi + r f (x))}^{2} \, \theta_{\xi + r f (x)}(x) \, d \xi \, d \mathcal{H}^{n}(x),
\end{eqnarray*}
with, 
\begin{equation*}
    |\nabla f (x)|_{(x,t)}^{2} = g_{\exp_{\iota(x)}(t \nu (x))} (d \exp_{\iota(x)} (t \nu(x))(\iota_{*} (\nabla f (x)) ), d \exp_{\iota(x)} (t \nu(x))(\iota_{*} (\nabla f (x))))
\end{equation*}

\bigskip 

\begin{rem}\label{rem: Choice of delta for normal neighbourhood, and C 1}
    We may choose $\delta = \delta (z_{0}, N, g) > 0$, such that $B_{2 \delta}^{N} (z_{0})$ is a totally normal neighbourhood, and the following holds
    \begin{equation*}
        C_{3}  = C_{3}  (z_{0}, N, g, \delta) = \sup \{ |d \exp_{y} (X)|^{2} \colon y \in B_{\delta}^{N} (z_{0}), \, X \in B_{2 \delta}^{T_{y}N} (0)\} \leq 100 n^{2}.
    \end{equation*}
\end{rem} 

By choices of $\rho$ in Remark \ref{rem: choice of rho so that B l times -rho rho fits into B delta}, and $\varepsilon$ in Remark \ref{rem: choice in epsilon 1}, for all $x$ in $A_{l}$, $r$ in $[0, \rho]$, and $\xi$ in $[-2 \varepsilon \Lambda, 2 \varepsilon \Lambda]$, 
\begin{equation*}
    |r \nabla f(x)|_{(x,\xi + r f (x))}^{2} \leq C_{3}  |r^{2} \nabla f (x)|^{2}_{(x,0)} \leq 2 C_{3}  \frac{r^{2}}{l^{2}}.
\end{equation*}
Note that for $x$ in $\tilde{M} \setminus A_{l}$, $\nabla f (x) = 0$, therefore, $|\nabla f (x)|_{(x,t)} = 0$, for all $t$.
We have, 
\begin{equation*}
    \rom{1}_{\varepsilon}^{r,g_{1}} \leq C_{3} \mathcal{H}^{n}(A_{l}) \frac{r^{2}}{l^{2}},
\end{equation*}
where we have potentially rechosen $C_{3} = C_{3} (z_{0}, N, M, g, \delta, \lambda, W)$.

\bigskip 

For $r$ in $[0, 4 \varepsilon \Lambda)$, we have, 
\begin{eqnarray*}
    \rom{1}_{\varepsilon}^{r,g} + \rom{2}_{\varepsilon}^{r,g} &\leq& C_{3} \frac{(\varepsilon \Lambda)^{2}}{l^{2}} + C_{2} \bigg( \varepsilon \Lambda \int_{B_{2 l}} q_{\varepsilon}^{1} (x) \, d \mathcal{H}^{n}(x) \\
    && \hspace{1cm} + \int_{0}^{r} \int_{\tilde{M}} \Theta_{\varepsilon, g}^{1} (s,x) - \Theta_{\varepsilon, g}^{2} \, d\mathcal{H}^{n}(x) \, ds + \varepsilon^{2} \bigg). 
\end{eqnarray*}
Again, we are potentially rechoosing $C_{3} = C_{3} (z_{0}, N, M, g, \delta, \lambda, W) < + \infty$.

\bigskip 

For $r$ in $[4 \varepsilon \Lambda, \rho]$ we define the following non-decreasing function, 
\begin{equation*}
    P_{\varepsilon} (r) \coloneqq \frac{\mathcal{H}^{n} (\{ x \colon x \in B_{l}, \, -r + 2 \varepsilon \Lambda \leq \sigma^{-} (x) \leq 0\})}{\mathcal{H}^{n} (A_{l})},
\end{equation*}
and we have,
\begin{eqnarray*}
    \rom{1}_{\varepsilon}^{r,g_{1}} + \rom{2}_{\varepsilon}^{r,g_{1}} &\leq& \mathcal{H}^{n} (A_{l}) \left( C_{3} \frac{r^{2}}{l^{2}} - 2 \sigma P_{\varepsilon}(r) \right) \\
    && + C_{2} \bigg( \mathcal{H}^{n} ( \{ x \colon x \in B_{l}, -2 \varepsilon \Lambda < \sigma^{-}(x) \leq 0 \}) + \int_{B_{2 l}} q_{\varepsilon}^{1} (x) \, d \mathcal{H}^{n}(x)\\
    && + \int_{0}^{r} \int_{\tilde{M}} \Theta_{\varepsilon, g}^{1} (s,x) - \Theta_{\varepsilon, g}^{2} \, d\mathcal{H}^{n}(x) \, ds + \varepsilon^{2} \bigg).
\end{eqnarray*}

\bigskip 

We now define the following function on $[0,1]$, 
\begin{equation*}
    \kappa_{\varepsilon} (s) = \begin{cases}
        0, & s \in [0, (4 \varepsilon \Lambda) / \rho), \\
        C_{3} \frac{\rho^{2}}{l^{2}} s^{2} - 2 \sigma P_{\varepsilon} (s \rho), & s \in [(4 \varepsilon \Lambda) / \rho, 1].
    \end{cases}
\end{equation*}
Note that, 
\begin{equation*}
    P_{\varepsilon} (\rho) \xrightarrow{\varepsilon \rightarrow 0} \frac{\mathcal{H}^{n} (B_{l})}{\mathcal{H}^{n} (A_{l})} \xrightarrow{\rho \rightarrow 0} \frac{1}{2^{n} - 1},
\end{equation*}
and furthermore, recalling the bound $\rho \leq C_{1}  l^{2}$, $C_{1} = C_{1} (z_{0}, N, M, g, \delta) < + \infty$, we have, 
\begin{equation*}
    0 < \frac{\rho^{2}}{l^{2}} \leq C_{1}  \rho \xrightarrow{\rho \rightarrow 0} 0.
\end{equation*}

\begin{rem}\label{rem: choice of rho for path from disks to crab}
    Choose $\rho = \rho (z_{0}, N, M, g, \delta, \lambda, W) > 0$, such that 
    \begin{equation*}
        C_{3} \frac{\rho^{2}}{l^{2}} < \frac{\sigma}{2(2^{n} - 1)},
    \end{equation*}
    and, 
    \begin{equation*}
        \frac{\mathcal{H}^{n} (B_{l})}{\mathcal{H}^{n}(A_{l})} > \frac{7}{8(2^{n} - 1)}.
    \end{equation*}
\end{rem}

\bigskip 

\begin{rem}\label{rem: Choice of epsilon 2 for function P epsilon}
    There exists an $\varepsilon_{2} = \varepsilon_{2} (z_{0}, M, N, g, \delta, W, \lambda, \rho) > 0$, such that, $\varepsilon_{2} \leq \varepsilon_{1}$, and for all $\varepsilon$ in $(0, \varepsilon_{2})$,
    \begin{equation*}
        P_{\varepsilon} (\rho) > \frac{3}{4(2^{n} - 1)}.
    \end{equation*}
    From here we always choose $\varepsilon$ in $(0, \varepsilon_{2})$.
\end{rem}

\bigskip 

We have that, 
\begin{equation*}
    \max_{s \in [0,1]} \kappa_{\varepsilon} (s) \leq C_{3} \frac{\rho^{2}}{l^{2}} < \frac{\sigma}{2(2^{n} - 1)},
\end{equation*}
and, 
\begin{equation*}
    \kappa_{\varepsilon} (1) < - \frac{\sigma}{2^{n} - 1}.
\end{equation*}

\bigskip 

We have, for $r$ in $[0, 4 \varepsilon \Lambda)$,
\begin{equation*}
    \mathcal{F}_{\varepsilon, \lambda} (v_{\varepsilon}^{r, g_{1}}) \leq \mathcal{F}_{\varepsilon, \lambda} (v_{\varepsilon}) + \rom{3}_{\varepsilon}^{1, r},
\end{equation*}
where, 
\begin{eqnarray*}
    \rom{3}_{\varepsilon}^{1, r} = C_{4} \bigg( \varepsilon \Lambda \int_{B_{2 l}} q_{\varepsilon}^{1} (x) \, d \mathcal{H}^{n}(x) + \int_{0}^{r} \int_{\tilde{M}} \Theta_{\varepsilon, g}^{1} (s,x) - \Theta_{\varepsilon, g}^{2} \, d\mathcal{H}^{n}(x) \, ds + (\varepsilon \Lambda)^{2} \bigg),
\end{eqnarray*}
and $C_{4} = C_{4} (z_{0}, M, N, g, \delta, W, \lambda, \rho) < +\infty$.

\bigskip 

For $r$ in $[4 \varepsilon \Lambda, \rho]$, 
\begin{equation*}
    \mathcal{F}_{\varepsilon, \lambda} (v_{\varepsilon}^{r, g_{1}}) \leq \mathcal{F}_{\varepsilon, \lambda} (v_{\varepsilon}) + \mathcal{H}^{n} (A_{l}) \kappa_{\varepsilon} \left( \frac{r}{\rho} \right) + \rom{3}_{\varepsilon}^{2, r},
\end{equation*}
where,
\begin{eqnarray*}
    \rom{3}_{\varepsilon}^{2,r} &=& C_{2} \bigg( \mathcal{H}^{n} ( \{ x \colon x \in B_{l}, -2 \varepsilon \Lambda < \sigma^{-}(x) \leq 0 \}) + \int_{B_{2 l}} q_{\varepsilon}^{1} (x) \, d \mathcal{H}^{n}(x)\\
    && + \int_{0}^{r} \int_{\tilde{M}} \Theta_{\varepsilon, g}^{1} (s,x) - \Theta_{\varepsilon, g}^{2} \, d\mathcal{H}^{n}(x) \, ds + \varepsilon^{2} \bigg).
\end{eqnarray*}
Furthermore, 
\begin{equation*}
    \mathcal{F}_{\varepsilon, \lambda} (v_{\varepsilon}^{\rho, g_{1}}) \leq \mathcal{F}_{\varepsilon, \lambda} (v_{\varepsilon}) - \frac{\sigma \mathcal{H}^{n} (A_{l})}{2^{n} - 1} + \rom{3}_{\varepsilon}^{2, \rho}.
\end{equation*}

\subsection{Appropriate Function on Manifold}\label{subsec: Function on Manifold Disks to Crab}

We wish to show that for every $r$ in $[0, \rho]$, there exists a $\tilde{v}_{\varepsilon}^{r,g_{1}}$, in $W^{1, \infty} (N) \subset W^{1,2} (N)$, such that, for every $(x,t)$ in $\tilde{T}$, 
\begin{equation*}
    \tilde{v}_{\varepsilon}^{r,g_{1}} (F(x,t)) = v_{\varepsilon}^{r,g_{1}} (x,t).
\end{equation*}
This implies that $\mathcal{F}_{\varepsilon, \lambda} (\tilde{v}_{\varepsilon}^{r,g_{1}}) (N) = \mathcal{F}_{\varepsilon, \lambda} (v_{\varepsilon}^{r,g_{1}}) (\tilde{T})$.
Indeed, this follows from the fact that $\mu_{g} (N \setminus F(\tilde{T}) ) = \mu_{g} ( \text{Cut} (M) \cup (\overline{M} \setminus M) ) = 0$, and $F \colon \tilde{T} \rightarrow F(\tilde{T})$ is a bijection between open sets, Remark \ref{rem: Cut M and F tilde T are disjoint in N},
\begin{eqnarray*}
    \mathcal{F}_{\varepsilon, \lambda} (\tilde{v}_{\varepsilon}^{r, g_{1}}) (N) &=& \mathcal{F}_{\varepsilon, \lambda} (\tilde{v}_{\varepsilon}^{r, g_{1}}) (N \setminus ( \text{Cut} (M) \cup (\overline{M} \setminus M) ) ), \\
    &=& \mathcal{F}_{\varepsilon, \lambda} (\tilde{v}_{\varepsilon}^{r, g_{1}}) (F (\tilde{T})), \\
    &=& \mathcal{F}_{\varepsilon, \lambda} (v_{\varepsilon}^{r, g_{1}}) (\tilde{T}).
\end{eqnarray*}
We have the following, 
\begin{equation*}
    B_{\delta}^{N} (z_{0}) = \Upsilon_{1} \sqcup A \sqcup \Upsilon_{2},
\end{equation*}
where, 
\begin{equation*}
    \Upsilon_{1} = \{ y \in B_{\delta}^{N} (z_{0}) \colon \tilde{d}_{1} (y) > \tilde{d}_{2} (y)\},
\end{equation*}
and, 
\begin{equation*}
    \Upsilon_{2} = \{ y \in B_{\delta}^{N} (z_{0}) \colon \tilde{d}_{2} (y) > \tilde{d}_{1} (y)\}.
\end{equation*}
Recall Remark \ref{rem: choice of delta for set d 1 = d 2},
\begin{equation*}
    A = \{ y \in B_{\delta}^{N} (z_{0}) \colon \tilde{d}_{1} (y) = \tilde{d}_{2} (y)\},
\end{equation*}
is a smooth $n$-submanifold in $B_{\delta}^{N}(z_{0})$. 
Recall the diffeomorphisms, for $i= 1, \, 2$, defined in Remark \ref{rem: diffeos F i}, 
\begin{equation*}
    F_{i} \colon \tilde{V}_{i} \subset \tilde{M} \times \mathbb{R} \rightarrow B_{2 \delta}^{N} (z_{0}).
\end{equation*}
We then define, $\tilde{v}_{\varepsilon}^{r, g_{1}}$, 
\begin{equation*}
    \tilde{v}_{\varepsilon}^{r,g_{1}} (y) = \begin{cases}
     \overHe_{\varepsilon}(\tilde{d}(y)), & y \not\in B_{\delta}^{N} (z_{0}), \\
     v_{\varepsilon}^{r, g_{1}} (F_{1}^{-1} (y)), & y \in \overline{\Upsilon_{1}} \cap B_{\delta}^{N} (z_{0}), \\
     v_{\varepsilon}^{r,g_{1}} (F_{2}^{-1} (y)), & y \in \overline{\Upsilon_{2}} \cap B_{\delta}^{N} (z_{0}). 
    \end{cases}
\end{equation*}

\bigskip 
 
For $(x,t)$ in $\tilde{T}$, we have $\tilde{v}_{\varepsilon}^{r, g_{1}} (F(x,t)) = v_{\varepsilon}^{r, g_{1}} (x,t)$. 
Indeed, first we consider the case that $F(x,t)$ lies in $\Upsilon_{1} \cup \Upsilon_{2}$. 
In $\Upsilon_{i}$, $F = F_{i}$, and we have, 
\begin{equation*}
    \tilde{v}_{\varepsilon}^{r, g_{1}} (F(x,t)) = v_{\varepsilon}^{r,g_{1}} (F_{i}^{-1} (F(x,t))) = v_{\varepsilon}^{r, g_{1}} (x,t).
\end{equation*}
As $A \subset \text{Cut} (M)$, we know that $F(x,t)$ cannot lie on $A$. 
Last case to consider is $F(x,t)$ lies in $N \setminus B_{\delta}^{N} (z_{0})$. 
By Remark \ref{rem: choice of rho so that B l times -rho rho fits into B delta} $(x,t)$ must lie in $\tilde{T} \setminus (B_{2l} \times (-2 \rho, 2 \rho))$. 
If $x$ lies in $\tilde{M} \setminus B_{2 l}$, then $f (x) = 0$, and, 
\begin{equation*}
    v_{\varepsilon}^{r, g_{1}} (x,t) = \overHe_{\varepsilon} (t) = \overHe_{\varepsilon}(\tilde{d}(F(x,t))) = \tilde{v}_{\varepsilon}^{r, g_{1}} (F(x,t)).
\end{equation*}
If $x$ lies in $B_{2l}$, then $|t| \geq 2 \rho > r |f(x)| + 2 \varepsilon \Lambda$, and therefore, 
\begin{equation*}
    v_{\varepsilon}^{r, g_{1}} (x,t) = \overHe_{\varepsilon}(t - r f(x)) = \begin{cases}
        1, & t \geq 2 \rho > r f(x) + 2 \varepsilon \Lambda, \\
        -1, & t \leq - 2 \rho < r f (x) - 2 \varepsilon \Lambda.
    \end{cases}
\end{equation*}
Also, $\tilde{d} (F(x,t)) = t$, implies that, 
\begin{equation*}
    \tilde{v}_{\varepsilon}^{r, g_{1}} (F(x,t)) = \overHe_{\varepsilon} (t) = \begin{cases}
        1, & t \geq 2 \rho > 2 \varepsilon \Lambda, \\
        -1, & t \leq - 2 \rho < - 2 \varepsilon \Lambda.
    \end{cases} 
\end{equation*}
Therefore, for all $(x,t)$ in $\tilde{T}$, we have that $v_{\varepsilon}^{r,g_{1}} (x,t) = \tilde{v}_{\varepsilon}^{r, g_{1}} (F(x,t))$.

\bigskip 

We now just look to show that $\tilde{v}_{\varepsilon}^{r, g_{1}}$ lies in $W^{1, \infty} (N)$. 
First consider $y$ in $N \setminus F(B_{2l} \times (-2 \rho, 2 \rho))$. 
There exists an $x$ in $\tilde{M}$, such that, $F(x, \tilde{d} (y)) = y$, and $(x,\tilde{d} (y))$ lies in $(\tilde{M} \times \mathbb{R}) \setminus (B_{2 l} \times (-2 \rho, 2 \rho))$. 
By previous argument we see that, 
\begin{equation*}
    \tilde{v}_{\varepsilon}^{r,g_{1}} (y) = \overHe_{\varepsilon} (\tilde{d} (y)).
\end{equation*}
and therefore, $\tilde{v}_{\varepsilon}^{r,g_{1}}$ is Lipschitz on the set $N \setminus F(B_{2 l} \times (-2 \rho, 2 \rho))$. 

\bigskip 

As 
\begin{equation*}
    F(B_{2l} \times (-2 \rho, 2 \rho)) \subset \subset B_{\delta}^{N} (z_{0}),
\end{equation*}
showing that $\tilde{v}_{\varepsilon}^{r, g_{1}}$ is Lipschitz on $B_{\delta}^{N} (z_{0})$, implies that it is Lipschitz on $N$.

\bigskip 

Consider $y$ on $A$, then $\tilde{d}_{1} (y) = \tilde{d}_{2} (y)$, and by construction of $f$ and $\Psi$, 
\begin{equation*}
    f(\pi (F_{1}^{-1} (y))) = f(\pi (F_{2}^{-1} (y))).
\end{equation*}
Therefore, 
\begin{equation*}
    v_{\varepsilon}^{r,g_{1}} (F_{1}^{-1} (y)) = v_{\varepsilon}^{r,g_{1}} (F_{2}^{-1} (y)),
\end{equation*}
and $\tilde{v}_{\varepsilon}^{r,g_{1}}$ is well-defined and continuous across the smooth $n$-submanifold $A$.
Thus, we have that $\tilde{v}_{\varepsilon}^{r, g_{1}}$ lies in $W^{1, \infty} (B_{\delta}^{N} (z_{0}))$.

\subsection{Continuity of the Path}\label{subsec: Continuity of path from disks to crab}

We show that the path, 
\begin{eqnarray*}
    \gamma \colon [0, \rho] &\rightarrow& W^{1,2} (N), \\
    r &\mapsto& \tilde{v}_{\varepsilon}^{r,g_{1}},
\end{eqnarray*}
is continuous in $W^{1,2} (N)$. 

\bigskip 

Take $r$ and $s$ in $[0, \rho]$. 
Recalling that $\mu_{g} (N \setminus F(\tilde{T}) ) = \mu_{g} ( \text{Cut}(M) \cup (\overline{M} \setminus M) ) = 0$,
\begin{eqnarray*}
    \| \tilde{v}_{\varepsilon}^{r, g_{1}} - \tilde{v}_{\varepsilon}^{s, g_{1}} \|_{L^{2} (N)}^{2} &=& \int_{F(\tilde{T})} |\tilde{v}_{\varepsilon}^{r, g_{1}} - \tilde{v}_{\varepsilon}^{s, g_{1}}|^{2}, \\
    &=& \int_{\mathbb{R}} \int_{\tilde{M}} |\overHe_{\varepsilon} (t - r f(x)) - \overHe_{\varepsilon} (t - s f(x))|^{2} \theta_{t} (x) \, d \mathcal{H}^{n} (x) \, dt, \\
    &\xrightarrow{s \rightarrow r}& 0,
\end{eqnarray*}
by Dominated Convergence Theorem.

\bigskip 

Noting that, $\tilde{v}_{\varepsilon}^{r,g_{1}} = \tilde{v}_{\varepsilon}^{r,g_{1}}$ on $N \setminus B_{\delta}^{N} (z_{0})$, for all $r$ in $[0, \rho]$, and $\mu_{g} (B_{\delta}^{N}(z_{0}) \setminus (\Upsilon_{1} \cup \Upsilon_{2})) = \mu_{g} (A) = 0$,
\begin{equation*}
    \| \nabla \tilde{v}_{\varepsilon}^{r, g_{1}} - \nabla \tilde{v}_{\varepsilon}^{s, g_{1}} \|_{L^{2} (N)}^{2} = \int_{\Upsilon_{1} \cup \Upsilon_{2}}  |\nabla \tilde{v}_{\varepsilon}^{r, g_{1}} - \nabla \tilde{v}_{\varepsilon}^{s, g_{1}}| \, d \mu_{g}.
\end{equation*}
As $F_{i}^{-1} \colon (\Upsilon_{i}, g) \rightarrow (F_{i}^{-1} (\Upsilon_{i}), h)$ is an isometry, we have,
\begin{eqnarray*}
    \| \nabla \tilde{v}_{\varepsilon}^{r, g_{1}} - \nabla \tilde{v}_{\varepsilon}^{s, g_{1}} \|_{L^{2} (N)}^{2} &=& \int_{F_{1}^{-1} (\Upsilon_{1}) \cup F_{2}^{-1} (\Upsilon_{2})} |\nabla v_{\varepsilon}^{r,g_{1}} (x,t) - \nabla v_{\varepsilon}^{s,g_{1}} (x,t)|^{2}, \\
    &=& \int_{F_{1}^{-1} (\Upsilon_{1}) \cup F_{2}^{-1} (\Upsilon_{2})} (\overHe_{\varepsilon}' (t - r f(x)) - \overHe_{\varepsilon}' (t - s f(x)))^{2} \\
    && \hspace{2cm} + |\nabla_{x} f(x)|^{2} (r \overHe_{\varepsilon}' (t - r f(x)) - s \overHe_{\varepsilon}' (t - s f(x)))^{2}, \\
    &\xrightarrow{s \rightarrow r}& 0,
\end{eqnarray*}
by Dominated Convergence Theorem.

\section{Path to $a_{\varepsilon}$}\label{sec: path to minus 1}

\subsection{Fixed Energy Gain Away from Non-Embedded Point}\label{sec: Crab Legs Forward Away From Body}

We construct the path from (2) to (3) in Figure \ref{fig: The Path}.

\bigskip 

Recall $\tilde{f}$ from Section \ref{subsec: construction about non-embed point} and set, 
\begin{equation*}
    g_{2}(r, x) = \rho f (x) + r \tilde{f}(x),
\end{equation*}
for $r$ in $[0,r_{0}]$, where $r_{0} \in (0, \min\{1, \text{diam}(N)/2\})$, will be chosen later. 
Denote, $A_{L}^{k} = B_{L} \setminus B_{\frac{L}{k}}$.

\begin{rem}\label{rem: Choice of epsilon 3}
    We choose $0 < \varepsilon_{3} \leq \varepsilon_{2}$, such that $2 \varepsilon_{3} \Lambda = 6 \varepsilon_{3} |\log \, \varepsilon_{3}| < < r_{0}$.
    From here on we consider $\varepsilon$ on $(0, \varepsilon_{3})$.
\end{rem}

We slightly edit the Base Computation in Section \ref{sec: Baseline Computation}. 
Consider $r > 2 \varepsilon \Lambda$,
\begin{equation*}
    \mathcal{F}_{\varepsilon, \lambda} (v_{\varepsilon}^{r,g_{2}}) - \mathcal{F}_{\varepsilon, \lambda} (v_{\varepsilon}^{r, g_{2}}) = \rom{1}_{\varepsilon}^{r, g_{2}} + (\rom{2}_{\varepsilon}^{r, g_{2}} - \rom{2}_{\varepsilon}^{2 \varepsilon \Lambda, g_{2}} ) + \rom{2}_{\varepsilon}^{2 \varepsilon \Lambda, g_{2}}.
\end{equation*}
We have,
\begin{equation}\label{eqn: rom 2 g 2 r minus rom 2 g 2 epsilon}
    \begin{split}
    \rom{2}_{\varepsilon}^{r,g_{2}} - \rom{2}_{\varepsilon}^{2 \varepsilon \Lambda, g_{2}} = & - \int_{2 \varepsilon \Lambda}^{r} \int_{\tilde{M} \setminus B_{\frac{L}{k}}} \tilde{f} (x) Q_{\varepsilon}(\sigma^{+}(x) - s \tilde{f}(x)) \, \theta^{+}(x) \, d\mathcal{H}^{n}(x) \, ds \\
    & + \int_{2 \varepsilon \Lambda}^{r} \int_{\tilde{M} \setminus B_{\frac{L}{k}}} \tilde{f}(x) Q_{\varepsilon}(\sigma^{-}(x) - s \tilde{f}(x)) \, \theta^{-}(x) \, d\mathcal{H}^{n}(x) \, ds \\
    & + \int_{2 \varepsilon \Lambda}^{r} \int_{\tilde{M} \setminus B_{\frac{L}{k}}} \tilde{f}(x) \int_{\sigma^{-}(x)}^{\sigma^{+}(x)} Q_{\varepsilon}(t - s \tilde{f}(x)) \, (\lambda - H_{t}(x)) \, \theta_{t}(x) \, dt \, d\mathcal{H}^{n}(x) \, ds \\
    &  + \lambda \int_{2 \varepsilon \Lambda}^{r} \int_{\tilde{M}} \Theta_{\varepsilon, g_{2}}^{1} (s,x) - \Theta_{\varepsilon, g_{2}}^{2} (s,x) \, d\mathcal{H}^{n}(x) \, ds.
    \end{split}
\end{equation}

\bigskip 

Considering the first term on the right-hand side of (\ref{eqn: rom 2 g 2 r minus rom 2 g 2 epsilon}), 
\begin{equation*}
    - \int_{2 \varepsilon \Lambda}^{r} \int_{\tilde{M} \setminus B_{\frac{L}{k}}} \tilde{f} (x) Q_{\varepsilon}(\sigma^{+}(x) - s \tilde{f}(x)) \, \theta^{+}(x) \, d\mathcal{H}^{n}(x) \, ds \leq 0.
\end{equation*}

Considering the second term on the right-hand side of (\ref{eqn: rom 2 g 2 r minus rom 2 g 2 epsilon}), and by applying similar arguments for when we considered the corresponding term on the right hand side of (\ref{eqn: rom 2 for g 1}) in Section \ref{subsec: Path From Disks to Crab, calculation},
\begin{eqnarray*}
    &&\int_{2 \varepsilon \Lambda}^{r} \int_{\tilde{M} \setminus B_{\frac{L}{k}}} \tilde{f}(x) Q_{\varepsilon}(\sigma^{-}(x) - s \tilde{f}(x)) \, \theta^{-}(x) \, d\mathcal{H}^{n}(x) \, ds \\ 
    && \hspace{5cm} \leq C_{2} \mathcal{H}^{n} ( \{ x \colon x \in \tilde{M} \setminus B_{\frac{L}{k}}, \, \sigma^{-}(x) \geq 2 \varepsilon \Lambda (\tilde{f}(x) - 1) \} ),
\end{eqnarray*}
where we are potentially rechoosing $C_{2} = C_{2} (M, N, g, W, \lambda)< + \infty$.

\bigskip 

Considering the third term on the right-hand side of (\ref{eqn: rom 2 g 2 r minus rom 2 g 2 epsilon}). 
Applying similar arguments in $A_{L}^{k}$ from when we considered the corresponding term on the right-hand side of (\ref{eqn: rom 2 for g 1}) in Section \ref{subsec: Path From Disks to Crab, calculation},
\begin{equation}\label{eqn: Fixed energy gain integral term with mean curvature of level sets}
    \begin{split}
    & \int_{2 \varepsilon \Lambda}^{r} \int_{\tilde{M} \setminus B_{\frac{L}{k}}} \tilde{f}(x) \int_{\sigma^{-}(x)}^{\sigma^{+}(x)} Q_{\varepsilon}(t - s \tilde{f}(x)) \, (\lambda - H_{t}(x)) \, \theta_{t}(x) \, dt \, d\mathcal{H}^{n}(x) \, ds \\
    & \hspace{2cm} \leq \int_{2 \varepsilon \Lambda}^{r} \int_{\tilde{M} \setminus B_{L}} \int_{-2 \varepsilon \Lambda}^{2 \varepsilon \Lambda} Q_{\varepsilon} (\xi) (\lambda - H_{\xi + s} (x)) \theta_{\xi + s} (x) \, d \xi \, d \mathcal{H}^{n}(x) \, ds \\
    & \hspace{3cm} + C_{2} \int_{A_{L}^{k}} q_{\varepsilon}^{2} (x) \, d \mathcal{H}^{n} (x),
    \end{split}
\end{equation}
where, $q_{\varepsilon}^{2} (x) \coloneqq \max_{t \in [-2 \varepsilon \Lambda, 4 \varepsilon \Lambda]} ( \lambda - H_{t} (x)) \theta_{t} (x)$, and we are potentially rechoosing $C_{2} = C_{2} (M, N, g, \lambda, W)$.

\bigskip 

Define the following measurable set, 
\begin{equation*}
    \Omega_{r} = \{ x \in \tilde{M} \colon \, \sigma^{+}(x) > 2 r \, \}.
\end{equation*}

\begin{rem}\label{rem: second choice of L}
    We choose $L = L (z_{0}, N, M, g, \delta) > 0$, such that, 
    \begin{equation*}
        \mathcal{H}^{n} (\tilde{M} \setminus B_{L}) > \frac{3}{4} \mathcal{H}^{n} (\tilde{M}).
    \end{equation*}
    Then we can find an $r_{0} = r_{0}(z_{0}, M, N, g, \delta, L) > 0$, such that, for all $r$ in $[0, r_{0}]$, 
\begin{equation*}
    \mathcal{H}^{n} (\{ (x, 2r) \colon \, x \in \Omega_{r} \setminus B_{L} \}) > \frac{1}{2} \mathcal{H}^{n}(\tilde{M}).
\end{equation*}
\end{rem}

For all $x$ in $\Omega_{r}$, $s$ in $(2 \varepsilon \Lambda, r)$, and $\xi$ in $[- 2 \varepsilon \Lambda, 2 \varepsilon \Lambda]$, $s + \xi$ lies in $(0, \sigma^{+}(x))$. 
Therefore, recalling bounds on $H_{t}$ and $\theta_{t}$ from Remark \ref{rem: Bounds on mean curvature of level sets} and Claim \ref{claim: derivative of log area element}, we have,
\begin{equation*}
    (\lambda - H_{\xi + s}(x)) \theta_{\xi + s} (x) < - m (s + \xi) \theta_{\xi + s} \leq - m (s - 2 \varepsilon \Lambda) \theta_{2r} (x),
\end{equation*} 
Then for $r$ in $(2 \varepsilon \Lambda, r_{0}]$, we compute an energy decrease from the first term on the right-hand side of (\ref{eqn: Fixed energy gain integral term with mean curvature of level sets}),
\begin{equation*}
    \int_{2 \varepsilon \Lambda}^{r} \int_{\tilde{M} \setminus B_{L}} \int_{-2 \varepsilon \Lambda}^{2 \varepsilon \Lambda} Q_{\varepsilon} (\xi) (\lambda - H_{\xi + s} (x)) \theta_{\xi + s} (x) \, d \xi \, d \mathcal{H}^{n}(x) \, ds \leq - \frac{m \sigma}{2} \mathcal{H}^{n}(\tilde{M}) r^{2} + C_{2} \varepsilon \Lambda,
\end{equation*}
potentially rechoosing $C_{2} = C_{2} (N, M, g, \lambda, W) < + \infty$.

\bigskip

For $r$ in $[0, 2 \varepsilon \Lambda]$, by repeating arguments similar to those in Section \ref{subsec: Approximating function for CMC}, 
\begin{eqnarray*}
    \rom{2}_{\varepsilon}^{r, g_{2}} &\leq& C_{2} \bigg( \int_{\tilde{M} \setminus B_{\frac{L}{k}}} m_{\varepsilon}^{1} (x) \, d \mathcal{H}^{n} (x) + \varepsilon \Lambda \bigg),
\end{eqnarray*}
where we are potentially rechoosing $C_{2} = C_{2} (N, M, g, W, \lambda)$, and $m_{\varepsilon}^{1} (x) = \max_{t \in [- 2 \varepsilon \Lambda, 4 \varepsilon \Lambda]} \theta_{t} (x) - \min_{t \in [-2 \varepsilon \Lambda, 4 \varepsilon \Lambda]} \theta_{t} (x)$.

\bigskip 

For $r$ in $[0, r_{0}]$, consider the term, 
\begin{equation*}
    \rom{1}_{\varepsilon}^{r, g_{2}} = \int_{A_{L}^{k}} \int_{- 2 \varepsilon \Lambda}^{2 \varepsilon \Lambda} \frac{\varepsilon}{2} ((\overHe)' (\xi))^{2} | r \nabla \tilde{f} (x)|_{(x,r \tilde{f}(x) + \xi)}^{2} \theta_{r \tilde{f}(x) + \xi} (x) \, d\xi \, d \mathcal{H}^{n} (x).
\end{equation*}
By choice of $L$ and $r_{0}$ in Remark \ref{rem: Choice of L and r 0 to fit in ball B delta}, and constant $C_{3} = C_{3} (z_{0}, M, N, g, \delta, \lambda, W)$ from Remark \ref{rem: Choice of delta for normal neighbourhood, and C 1}, we have, for all $x$ in $A_{L}^{k}$, $r$ in $[0, r_{0}]$, and $\xi$ in $[-2 \varepsilon \Lambda, 2 \varepsilon \Lambda]$,
\begin{equation*}
    |\nabla \tilde{f} (x)|_{(x,r \tilde{f}(x) + \xi)} \leq C_{3}^{1/2}  |\nabla \tilde{f} (x)|_{(x, 0)}
\end{equation*}
Thus we have, 
\begin{equation*}
    \rom{1}_{\varepsilon}^{r, g_{2}} \leq C_{3} \|\nabla \tilde{f}\|_{L^{2}(\tilde{M})}^{2} r^{2}.
\end{equation*}
Again we are potentially rechoosing $C_{3} = C_{3} (z_{0}, M, N, g, \delta, W, \lambda)$.

\bigskip 

\begin{rem}\label{rem: Choice of gamma and k}
    Choose $k = k (z_{0}, M, N, g, \delta, W, \lambda, L)$ such that 
    \begin{equation*}
        \| \nabla \tilde{f} \|_{L^{2}(\tilde{M})}^{2} < C_{3}^{-1} \frac{m \sigma}{4} \mathcal{H}^{n} (\tilde{M}).
    \end{equation*}
\end{rem}

Therefore, after potentially rechoosing $C_{3} = C_{3} (z_{0}, M, N, g, \delta, W, \lambda)$,
\begin{equation*}
    \rom{1}_{\varepsilon}^{r,g_{2}} \leq \frac{m \sigma}{4} \mathcal{H}^{n} (\tilde{M}) r^{2}.
\end{equation*}

For $r$ in $(0, 2 \varepsilon \Lambda]$, 
\begin{equation*}
    \mathcal{F}_{\varepsilon, \lambda} (v_{\varepsilon}^{r, g_{2}}) - \mathcal{F}_{\varepsilon, \lambda} (v_{\varepsilon}^{0, g_{2}}) \leq \rom{3}_{\varepsilon}^{3, r},
\end{equation*}
where, 
\begin{eqnarray*}
\rom{3}_{\varepsilon}^{3, r} &=& C_{2} \left( \int_{\tilde{M} \setminus B_{\frac{L}{k}}} m_{\varepsilon} (x) \, d \mathcal{H}^{n} (x) + \varepsilon \Lambda \right),
\end{eqnarray*}
where we are potentially rechoosing $C_{2} = C_{2} (N, M, g, W, \lambda) < + \infty$.
For $r$ in $(2 \varepsilon \Lambda, r_{0}]$, 
\begin{equation*}
    \mathcal{F}_{\varepsilon, \lambda} (v_{\varepsilon}^{r, g_{2}}) - \mathcal{F}_{\varepsilon, \lambda} (v_{\varepsilon}^{0, g_{2}}) \leq  - \frac{m \sigma}{4} \mathcal{H}^{n} (\tilde{M}) r^{2} + \rom{3}_{\varepsilon}^{4,r}
\end{equation*}
where, 
\begin{eqnarray*}
    \rom{3}_{\varepsilon}^{4,r} &=& C_{2} \bigg( \mathcal{H}^{n} ( \{ x \colon x \in \tilde{M} \setminus B_{\frac{L}{k}}, \, \sigma^{-}(x) \geq 2 \varepsilon \Lambda (\tilde{f}(x) - 1) \, \} ) \\
    && +\int_{A_{L}^{k}} q_{\varepsilon}^{2} (x) \, d \mathcal{H}^{n} (x) + \int_{2 \varepsilon \Lambda}^{r} \int_{\tilde{M}} \Theta_{\varepsilon, g_{2}}^{1} (s, x) - \Theta_{\varepsilon, g_{2}}^{2} \, d \mathcal{H}^{n} (x) \, ds \\
    && + \int_{\tilde{M} \setminus B_{\frac{L}{k}}} m_{\varepsilon} (x) \, d \mathcal{H}^{n} (x) + \varepsilon \Lambda \bigg),
\end{eqnarray*}
again, we are potentially rechoosing $C_{2} = C_{2}(M, N, g, W, \lambda)$.

\bigskip 

As $g_{2} (0, x) = g_{1} (\rho, x)$, we have, for $r$ in $(0, 2 \varepsilon \Lambda]$, 
\begin{equation*}
    \mathcal{F}_{\varepsilon, \lambda} (v_{\varepsilon}^{r, g_{2}}) \leq \mathcal{F}_{\varepsilon, \lambda} (v_{\varepsilon}) - \frac{\sigma \mathcal{H}^{n} (A_{l})}{2^{n} - 1} + \rom{3}_{\varepsilon}^{\rho, 2} + \rom{3}_{\varepsilon}^{r, 3}, 
\end{equation*}
and for $r$ in $(2 \varepsilon \Lambda, r_{0}]$, 
\begin{equation*}
    \mathcal{F}_{\varepsilon, \lambda} (v_{\varepsilon}^{r, g_{2}}) \leq \mathcal{F}_{\varepsilon, \lambda} (v_{\varepsilon}) - \frac{\sigma \mathcal{H}^{n} (A_{l})}{2^{n} - 1} - \frac{m \sigma}{4} \mathcal{H}^{n} (\tilde{M}) r^{2} + \rom{3}_{\varepsilon}^{\rho, 2} + \rom{3}_{\varepsilon}^{r, 4}. 
\end{equation*}

We may define the appropriate function on $N$, for $r$ in $[0, r_{0}]$, 
\begin{equation*}
    \tilde{v}_{\varepsilon}^{r,g_{2}} (y) = \begin{cases}
        \overHe_{\varepsilon} (\tilde{d}(y) - r), & y \not\in B_{\delta}^{N} (z_{0}), \\
        v_{\varepsilon}^{r,g_{2}} (F_{1}^{-1} (y)), & y \in \overline{\Upsilon_{1}} \cap B_{\delta}^{N} (z_{0}), \\
        v_{\varepsilon}^{r, g_{2}} (F_{2}^{-1} (y)), & y \in \overline{\Upsilon_{2}} \cap B_{\delta}^{N} (z_{0}).
    \end{cases}
\end{equation*}
Following similar arguments to Sections \ref{subsec: Function on Manifold Disks to Crab}, and \ref{subsec: Continuity of path from disks to crab}, we may show that $\tilde{v}_{\varepsilon}^{r, g_{2}}$ lies in $W^{1, \infty} (N)$, $\mathcal{F}_{\varepsilon, \lambda} (\tilde{v}_{\varepsilon}^{r, g_{2}}) (N) = \mathcal{F}_{\varepsilon, \lambda} (v_{\varepsilon}^{r,g_{2}}) (\tilde{T})$ and that the path $\tilde{v}_{\varepsilon}^{0, g_{2}} \rightarrow \tilde{v}_{\varepsilon}^{r_{0}, g_{2}}$ is continuous in $W^{1,2} (N)$.

\subsection{Reversing Construction of Competitor}\label{sec: Undoing the Crab}

We construct the path from (3) to (4) in Figure \ref{fig: The Path}.

\bigskip 

For $r$ in $[0, \rho]$, we set, 
\begin{equation*}
    g_{3} (r, x) = r_{0} \tilde{f} (x) + (\rho - r) f(x).
\end{equation*}
For $x$ in $B_{2l}$, 
\begin{equation*}
    g_{3} (r, x) = (\rho - r) f(x) = g_{1} (\rho - r, x),
\end{equation*}
and for $x$ in $\tilde{M} \setminus B_{2l}$, 
\begin{equation*}
    g_{3} (r,x) = r_{0} \tilde{f} (x) = g_{3} (0, x).
\end{equation*}
Therefore, 
\begin{equation*}
    \mathcal{F}_{\varepsilon, \lambda} (v_{\varepsilon}^{r, g_{3}}) - \mathcal{F}_{\varepsilon, \lambda} (v_{\varepsilon}^{0, g_{3}}) = \mathcal{F}_{\varepsilon, \lambda} (v_{\varepsilon}^{\rho - r, g_{1}}) - \mathcal{F}_{\varepsilon, \lambda} (v_{\varepsilon}^{\rho, g_{1}}).
\end{equation*}
As $g_{1} (\rho,x) = g_{2} (0, x)$, and $g_{3}(0, x) = g_{2} (r_{0}, x)$, we have, 
\begin{equation*}
    \mathcal{F}_{\varepsilon, \lambda} (v_{\varepsilon}^{r, g_{3}}) = \mathcal{F}_{\varepsilon, \lambda} (v_{\varepsilon}^{\rho - r, g_{1}}) + \mathcal{F}_{\varepsilon, \lambda} (v_{\varepsilon}^{r_{0}, g_{2}}) - \mathcal{F}_{\varepsilon, \lambda} (v_{\varepsilon}^{0, g_{2}}).
\end{equation*} 

\begin{rem}\label{rem: choice of rho based on r 0}
    We choose $\rho > 0$, such that, 
\begin{equation*}
    \frac{\sigma \mathcal{H}^{n} (A_{l})}{2^{n} - 1} < \frac{m \sigma}{4} \mathcal{H}^{n} (\tilde{M}) r_{0}^{2}.
\end{equation*}
\end{rem}

Therefore, we have that 
\begin{equation*}
    \mathcal{F}_{\varepsilon, \lambda} (v_{\varepsilon}^{r, g_{3}}) < \mathcal{F}_{\varepsilon, \lambda} (v_{\varepsilon}^{\rho - r, g_{1}}) - \frac{\sigma \mathcal{H}^{n} (A_{l})}{2^{n} - 1} + \rom{3}_{\varepsilon}^{4,r_{0}}.
\end{equation*} 
Furthermore, for $r$ in $[0, \rho - 4 \varepsilon \Lambda]$, we have, 
\begin{eqnarray*}
    \mathcal{F}_{\varepsilon, \lambda} (v_{\varepsilon}^{r, g_{3}}) &<& \mathcal{F}_{\varepsilon, \lambda} (v_{\varepsilon}) + \frac{\sigma \mathcal{H}^{n} (A_{l})}{2(2^{n} - 1)} + \rom{3}_{\varepsilon}^{\rho - r, 2} - \frac{\sigma \mathcal{H}^{n} (A_{l})}{2^{n} - 1} + \rom{3}_{\varepsilon}^{4,r_{0}}, \\
    &=& \mathcal{F}_{\varepsilon, \lambda} (v_{\varepsilon}) - \frac{\sigma \mathcal{H}^{n} (A_{l})}{2(2^{n} - 1)} + \rom{3}_{\varepsilon}^{\rho - r, 2} + \rom{3}_{\varepsilon}^{4,r_{0}}.
\end{eqnarray*} 
For $r$ in $(\rho - 4 \varepsilon \Lambda, \rho]$, we similarly have, 
\begin{equation*}
    \mathcal{F}_{\varepsilon, \lambda} (v_{\varepsilon}^{r, g_{3}}) < \mathcal{F}_{\varepsilon, \lambda} (v_{\varepsilon}) - \frac{\sigma \mathcal{H}^{n} (A_{l})}{2(2^{n} - 1)} + \rom{3}_{\varepsilon}^{\rho - r, 1} + \rom{3}_{\varepsilon}^{4,r_{0}}.
\end{equation*}

\bigskip 

We define the appropriate function on $N$. 
For $r$ in $[0, \rho]$, 
\begin{equation*}
    \tilde{v}_{\varepsilon}^{r,g_{3}} (y) = \begin{cases}
        \overHe_{\varepsilon} (\tilde{d}(y) - r_{0}), & y \not\in B_{\delta}^{N} (z_{0}), \\
        v_{\varepsilon}^{r,g_{3}} (F_{1}^{-1} (y)), & y \in \overline{\Upsilon_{1}} \cap B_{\delta}^{N} (z_{0}), \\
        v_{\varepsilon}^{r, g_{3}} (F_{2}^{-1} (y)), & y \in \overline{\Upsilon_{2}} \cap B_{\delta}^{N} (z_{0}).
    \end{cases}
\end{equation*}
Following similar arguments to Sections \ref{subsec: Function on Manifold Disks to Crab}, and \ref{subsec: Continuity of path from disks to crab}, we may show that $\tilde{v}_{\varepsilon}^{r, g_{3}}$ lies in $W^{1, \infty} (N)$, $\mathcal{F}_{\varepsilon, \lambda} (\tilde{v}_{\varepsilon}^{r, g_{3}}) (N) = \mathcal{F}_{\varepsilon, \lambda} (v_{\varepsilon}^{r,g_{3}}) (\tilde{T})$ and that the path $\tilde{v}_{\varepsilon}^{0, g_{3}} \rightarrow \tilde{v}_{\varepsilon}^{\rho, g_{3}}$ is continuous in $W^{1,2} (N)$.

\subsection{Lining Up With Level Set $\Gamma_{r_{0}}$}\label{sec: Disks Push Up}

We construct path from (4) to (5) in Figure \ref{fig: The Path}

\bigskip 

For $r$ in $[0, r_{0}]$, consider,
\begin{equation*}
    g_{4} (r, x) = r_{0} \tilde{f} (x) + r (1 - \tilde{f} (x)) = (r_{0} - r) \tilde{f} (x) + r \geq r.
\end{equation*}

\bigskip 

By applying similar arguments to those in Section \ref{subsec: Path From Disks to Crab, calculation}, we have
\begin{eqnarray*}
    \rom{2}_{\varepsilon}^{r, g_{4}} &\leq& C_{3} \bigg( \mathcal{H}^{n} (\{ x \in B_{L} \colon \sigma^{-} (x) \geq - 2 \varepsilon \Lambda \}) + \varepsilon \Lambda \\
    && \hspace{1cm} + \int_{0}^{r} \int_{\tilde{M}} \Theta_{\varepsilon, g_{4}}^{1} (s, x) - \Theta_{\varepsilon, g_{4}}^{2} (s,x) \, d \mathcal{H}^{n} (x) \, ds \bigg),
\end{eqnarray*}
where we are potentially rechoosing $C_{3} = C_{3} (z_{0}, M, N, g, \delta, W, \lambda) < + \infty$.

\bigskip 

We turn our attention to the term, 
\begin{eqnarray*}
    \rom{1}_{\varepsilon}^{r,g_{4}} &=& \int_{A_{L}^{k}} \int_{\sigma^{-}(x)}^{\sigma^{+}(x)} \frac{\varepsilon}{2} ( (\overHe_{\varepsilon})'(t - (r_{0} - r) \tilde{f}(x) - r))^{2} (r_{0} - r)^{2} |\nabla_{x} \tilde{f} (x) |^{2} \, \theta_{t}(x) \, dt \, d \mathcal{H}^{n}(x) \\
    && \hspace{1cm} - \int_{A_{L}^{k}} \int_{\sigma^{-}(x)}^{\sigma^{+}(x)} \frac{\varepsilon}{2} ( (\overHe_{\varepsilon})'(t - r_{0} \tilde{f}(x)))^{2} r_{0}^{2} |\nabla_{x} \tilde{f} (x) |^{2} \, \theta_{t}(x) \, dt \, d \mathcal{H}^{n}(x).
\end{eqnarray*}
For $r$ in $[0, r_{0}]$,
\begin{eqnarray*}
    && \int_{\sigma^{-}(x)}^{\sigma^{+}(x)} \frac{\varepsilon}{2} ( (\overHe_{\varepsilon})'(t - (r_{0} - r) \tilde{f}(x) - r))^{2} |\nabla_{x} \tilde{f} (x) |_{(x,t)}^{2} \, \theta_{t}(x) \, dt \\
    && \hspace{1cm} \leq (\sigma + \beta \varepsilon^{2}) \max_{t \in [-2 \varepsilon \Lambda, 2 \varepsilon \Lambda]}  |\nabla \tilde{f} (x)|_{(x, t + g_{4} (r,x))}^{2} \theta_{t + g_{4} (r,x)} (x), \\
    && \hspace{1.5cm} = \sigma | \nabla \tilde{f} (x) |_{(x, g_{4} (r,x))}^{2} \theta_{g_{4} (r,x)} (x) + \beta \varepsilon^{2} \max_{t \in [-2 \varepsilon \Lambda, 2 \varepsilon \Lambda]}  |\nabla \tilde{f} (x)|_{(x, t + g_{4} (r,x))}^{2} \theta_{t + g_{4} (r,x)} (x) \\
    && \hspace{2cm} + \sigma \left( \max_{t \in [-2 \varepsilon \Lambda, 2 \varepsilon \Lambda]}  |\nabla \tilde{f} (x)|_{(x, t + g_{4} (r,x))}^{2} \theta_{t + g_{4} (r,x)} (x) - | \nabla \tilde{f} (x) |_{(x, g_{4} (r,x))}^{2} \theta_{g_{4} (r,x)} (x) \right) 
\end{eqnarray*}
Denote the functions,
\begin{equation*}
    q_{\varepsilon}^{3} (x, r) = \max_{t \in [-2 \varepsilon \Lambda, 2 \varepsilon \Lambda]}  |\nabla \tilde{f} (x)|_{(x, t + g_{4} (r,x))}^{2} \theta_{t + g_{4} (r,x)} (x) - | \nabla \tilde{f} (x) |_{(x, g_{4} (r,x))}^{2} \theta_{g_{4} (r,x)} (x),
\end{equation*}
and, 
\begin{equation*}
    p^{1}_{\varepsilon} (r) = \int_{A_{L}^{k}} q_{\varepsilon}^{3} (x, r) \, d \mathcal{H}^{n} (x).
\end{equation*}

\bigskip 

We have, 
\begin{eqnarray*}
    && \int_{A_{L}^{k}} \int_{\sigma^{-}(x)}^{\sigma^{+}(x)} \frac{\varepsilon}{2} ( (\overHe_{\varepsilon})'(t - (r_{0} - r) \tilde{f}(x) - r))^{2} (r_{0} - r)^{2} |\nabla_{x} \tilde{f} (x) |^{2} \, \theta_{t}(x) \, dt \, d \mathcal{H}^{n}(x) \\
    && \hspace{2cm} \leq \sigma \int_{A_{L}^{k}} (r_{0} - r)^{2} |\nabla \tilde{f} (x)|_{(x, g_{4} (r,x))}^{2} \theta_{g(r,x)} (x) \, d \mathcal{H}^{n}(x) + C_{5} (\varepsilon^{2} + p^{1}_{\varepsilon} (r)).
\end{eqnarray*}
where $C_{5} = C_{5} (z_{0}, N, M, g, \delta, L, r_{0}, k, W, \lambda) < + \infty$.
Similarly, we have, 
\begin{eqnarray*}
    && \int_{A_{L}^{k}} \int_{\sigma^{-}(x)}^{\sigma^{+}(x)} \frac{\varepsilon}{2} ( (\overHe_{\varepsilon})'(t - r_{0} \tilde{f}(x)))^{2} r_{0}^{2} |\nabla_{x} \tilde{f} (x) |^{2} \, \theta_{t}(x) \, dt \, d \mathcal{H}^{n}(x) \\
    && \hspace{2cm} \geq \sigma \int_{A_{L}^{k}} r_{0}^{2} |\nabla \tilde{f} (x)|^{2}_{(x, g_{4}(0,x))} \theta_{g_{4} (0,x)} (x) \, d \mathcal{H}^{n} (x) - C_{5} (\varepsilon^{2} + p_{\varepsilon}^{2} (0)),
\end{eqnarray*}
where, 
\begin{equation*}
    p_{\varepsilon}^{2} (0) = \int_{A_{L}^{k}} q_{\varepsilon}^{4} (x,0) \, d \mathcal{H}^{n} (x), 
\end{equation*}
and 
\begin{equation*}
    q_{\varepsilon}^{4} (x, r) = \min_{t \in [-2 \varepsilon \Lambda, 2 \varepsilon \Lambda]}  |\nabla \tilde{f} (x)|_{(x, t + g_{4} (r,x))}^{2} \theta_{t + g_{4} (r,x)} (x) - | \nabla \tilde{f} (x) |_{(x, g_{4} (r,x))}^{2} \theta_{g_{4} (r,x)} (x) \leq 0
\end{equation*}

Therefore, we have, for $r$ in $[0, r_{0}]$,
\begin{eqnarray*}
    \rom{1}_{\varepsilon}^{r, g_{3}} &\leq& \sigma \int_{A_{L}^{k}} (r_{0} - r)^{2} |\nabla \tilde{f} (x)|^{2}_{(x, g_{4}(r,x))} \theta_{g_{4}(r,x)} (x) - r_{0}^{2} |\nabla \tilde{f} (x)|_{(x, g_{4}(0, x))}^{2} \theta_{g_{4}(0,x)} (x) \, d \mathcal{H}^{n} (x) \\
    && + C_{5} ( \varepsilon^{2} + p_{\varepsilon}^{1} (r) - p_{\varepsilon}^{2} (0) ).
\end{eqnarray*}

\begin{claim}\label{lem: Second choice of r 0}
    There exists an $r_{0} > 0$, such that for all $r$ in $[0, r_{0}]$, 
    \begin{equation*}
        \int_{A_{L}^{k}} (r_{0} - r)^{2} |\nabla \tilde{f} (x)|^{2}_{(x, g_{4}(r,x))} \theta_{g_{4}(r,x)} (x) - r_{0}^{2} |\nabla \tilde{f} (x)|_{(x, g_{4}(0, x))}^{2} \theta_{g_{4}(0,x)} (x) \, d \mathcal{H}^{n} (x) \leq 0.
    \end{equation*}
\end{claim}

\begin{proof}
    For $(x,t) \in \tilde{M} \times \mathbb{R}$, denote, 
    \begin{equation*}
        \zeta (x,t) \coloneqq |\nabla \tilde{f} (x)|^{2}_{(x,t)} \theta_{t} (x).
    \end{equation*}
    We note that $\zeta \in C^{\infty} (\tilde{T})$, and $\zeta (x,t) = 0$, for $(x,t) \not\in \tilde{T}$.

    \bigskip 

    For $r$ in $[0, r_{0}]$, denote, 
    \begin{equation*}
        G(r) \coloneqq \int_{A_{L}^{k}} \zeta (x, g_{4} (r,x)) \, d \mathcal{H}^{n} (x) = \int_{A_{L}^{k} \cap \{ \tilde{f} \not= 0\}} \zeta (x, g_{4} (r,x)) \, d \mathcal{H}^{n} (x).
    \end{equation*}
    For $r \in [0,r_{0}]$, and $\tilde{f} (x) \not= 0$, we have that, 
    \begin{equation*}
        g_{4} (r,x) \in (\sigma^{-}(x), \sigma^{+} (x)), 
    \end{equation*}
    i.e. $(x, g_{4} (r, x)) \in \tilde{T}$. 
    Therefore, $G$ lies in $C^{\infty} ([0, r_{0}])$, and 
    \begin{equation*}
        G' (r) = \int_{ A_{L}^{k} \cap \{ \tilde{f} \not= 0\} } (\partial_{t} \zeta) (x, g_{4} (r,x)) (1 - \tilde{f} (x)) \, dx.
    \end{equation*}
    We may obtain a bound $|G' (r)| \leq C_{6} = C_{6} (z_{0}, M, N, g, \delta, W, \lambda, L, k) < + \infty$, for $r_{0} \leq R_{0} < + \infty$, $R_{0} = R_{0} (z_{0}, M, N, g, \delta, W, \lambda, L, k)$ fixed. 
    We also have that $G(0) = \|\nabla \tilde{f}\|_{L^{2} (\tilde{M})}^{2} > 0$, and we may choose $r_{0} > 0$, small enough such that, 
    \begin{equation*}
        \min_{r \in [0,r_{0}]} G(r) \geq \frac{1}{2} G(0) > 0.
    \end{equation*}

    \bigskip 

    Denoting, 
    \begin{equation*}
        F(r) \coloneqq (r_{0} - r)^{2} G(r).
    \end{equation*}
    Differentiating we obtain, 
    \begin{equation*}
        F'(r) = (r - r_{0}) (2 G(r) + (r - r_{0}) G' (r)).
    \end{equation*}
    Thus setting, $r_{0} < G(0) / C_{6}$, we have that $F' (r) \leq 0$, for all $r$ in $[0, r_{0}]$.
    This completes the proof. 
\end{proof}

\begin{rem}\label{rem: Choice of r 0 from disks push up claim}
    Claim \ref{lem: Second choice of r 0} is a further choice of $r_{0} = r_{0} (z_{0}, M, N, g, \delta, W, \lambda, L, k) > 0$.
\end{rem}

Therefore, there exists an $r_{0}$, such that for all $r$ in $[0, r_{0}]$, 
\begin{eqnarray*}
    \mathcal{F}_{\varepsilon, \lambda} (v_{\varepsilon}^{r, g_{4}}) - \mathcal{F}_{\varepsilon, \lambda} (v_{\varepsilon}^{0, g_{4}}) &=& \rom{1}_{\varepsilon}^{r, g_{4}} + \rom{2}_{\varepsilon}^{r, g_{4}}, \\
    &\leq& \rom{3}_{\varepsilon}^{5,r},
\end{eqnarray*}
where
\begin{eqnarray*}
    \rom{3}_{\varepsilon}^{5,r} &\leq& C_{5} \bigg( p_{\varepsilon}^{1} (r) -  p_{\varepsilon}^{2} (0) + \mathcal{H}^{n} (\{ x \in B_{2 L} \colon \, \sigma^{-} (x) \geq - 2 \varepsilon \Lambda \}) + \varepsilon \Lambda \\ 
    && \hspace{2cm} + \lambda \int_{0}^{r} \int_{\tilde{M}} \Theta_{\varepsilon, g_{4}}^{1} (s,x) - \Theta_{\varepsilon, g_{4}}^{2} (s,x) \, d \mathcal{H}^{n} (x) \, ds \bigg),
\end{eqnarray*}
where we are potentially rechoosing $C_{5} = C_{5}(z_{0}, M, N, g, k, L, \delta, r_{0}, W, \lambda)$.

\bigskip 

As $g_{4} (0, x) = g_{3} (\rho, x)$, we have, for $r$ in $[0, r_{0}]$, 
\begin{equation*}
    \mathcal{F}_{\varepsilon, \lambda} (v_{\varepsilon}^{r, g_{4}}) \leq \mathcal{F}_{\varepsilon, \lambda} (v_{\varepsilon}) - \frac{\sigma \mathcal{H}^{n} (x)}{2(2^{n} - 1)} + \rom{3}_{\varepsilon}^{1, 0} + \rom{3}_{\varepsilon}^{4, r_{0}} + \rom{3}_{\varepsilon}^{5, r}.
\end{equation*}

\bigskip 

Consider the following function on $N$, for $r$ in $[0, r_{0}]$, 
\begin{equation*}
    \tilde{v}_{\varepsilon}^{r,g_{4}} (y) = \begin{cases}
        \overHe_{\varepsilon} (\tilde{d}(y) - r_{0}), & y \not\in B_{\delta}^{N} (z_{0}), \\
        v_{\varepsilon}^{r,g_{4}} (F_{1}^{-1} (y)), & y \in \overline{\Upsilon_{1}} \cap B_{\delta}^{N} (z_{0}), \\
        v_{\varepsilon}^{r, g_{4}} (F_{2}^{-1} (y)), & y \in \overline{\Upsilon_{2}} \cap B_{\delta}^{N} (z_{0}).
    \end{cases}
\end{equation*}
We can show, as in Section \ref{subsec: Function on Manifold Disks to Crab} and \ref{subsec: Continuity of path from disks to crab}, that $\tilde{v}_{\varepsilon}^{r,g_{3}}$ lies in $W^{1,\infty} (N)$, $\mathcal{F}_{\varepsilon, \lambda} (\tilde{v}_{\varepsilon}^{r, g_{4}}) (N) = \mathcal{F}_{\varepsilon, \lambda} (v_{\varepsilon}^{r,g_{4}}) (\tilde{T})$, and that, $r \mapsto \tilde{v}_{\varepsilon}^{r,g_{4}}$ is a continuous path in $W^{1,2} (N)$.

\subsection{Completing Path to $a_{\varepsilon}$}\label{sec: Completing Path to -1}

We construct the path from (5) to '-1' in Figure \ref{fig: The Path}.

\bigskip 

Consider, for $r$ in $[r_{0}, 2 \, \text{diam} (N)]$, 
\begin{equation*}
    g_{5} (r, x) =  r.
\end{equation*}

\bigskip 

By repeating similar arguments to those in Section \ref{subsec: Path From Disks to Crab, calculation}, we have
\begin{equation*}
    \mathcal{F}_{\varepsilon, \lambda} (v_{\varepsilon}^{r, g_{5}}) - \mathcal{F}_{\varepsilon, \lambda} (v_{\varepsilon}^{r_{0}, g_{5}}) \leq \rom{3}_{\varepsilon}^{6,r},
\end{equation*}
where 
\begin{equation*}
    \rom{3}_{\varepsilon}^{6,r} = \lambda \int_{\rho}^{r} \Theta_{\varepsilon, g_{5}}^{1} (s,x) - \Theta_{\varepsilon, g_{5}}^{2} (s,x) \, d \mathcal{H}^{n} (x) \, ds.
\end{equation*}

Recalling that, $g_{5} (r_{0}, x) = g_{4} (r_{0}, x)$, we have, for all $r$ in $[r_{0}, 2 \, \text{diam} (N)]$,
\begin{eqnarray*}
    \mathcal{F}_{\varepsilon, \lambda} (v_{\varepsilon}^{r, g_{5}}) &\leq& \mathcal{F}_{\varepsilon, \lambda} (v_{\varepsilon}^{r_{0}, g_{4}}) + \rom{3}_{\varepsilon}^{6, r}, \\
    &<& \mathcal{F}_{\varepsilon, \lambda} (v_{\varepsilon}) - \frac{\sigma \mathcal{H}^{n}(A_{l})}{2(2^{n} - 1)} + \rom{3}_{\varepsilon}^{1, 0} + \rom{3}_{\varepsilon}^{4, r_{0}} + \rom{3}_{\varepsilon}^{5, r_{0}} + \rom{3}_{\varepsilon}^{6,r}.
\end{eqnarray*}

Define the function, $\tilde{v}_{\varepsilon}^{r, g_{5}}(y) = \overHe_{\varepsilon} (\tilde{d} (y) - r)$, in $N$. 
This function lies in $W^{1, \infty} (N)$, $\mathcal{F}_{\varepsilon, \lambda} (\tilde{v}_{\varepsilon}^{r, g_{5}}) (N) = \mathcal{F}_{\varepsilon, \lambda} (v_{\varepsilon}^{r, g_{5}}) (\tilde{T})$, and $r \mapsto \tilde{v}_{\varepsilon}^{r, g_{5}}$ is a continuous path in $W^{1,2} (N)$.

\bigskip

As $|\tilde{d}(y)| \leq \, \text{diam}(N)$, we have that, 
\begin{equation*}
    \tilde{d} (y) - 2 \, \text{diam} (N) \leq - \, \text{diam} (N) < - 2 \varepsilon \Lambda.
\end{equation*}
Therefore, 
\begin{equation*}
    \tilde{v}_{\varepsilon}^{2 \, \text{diam}(N), g_{5}} (y) = \overHe_{\varepsilon}(\tilde{d}(x) - 2 \, \text{diam} (N)) = - 1.
\end{equation*}

\bigskip 

Recall that our end point is $a_{\varepsilon} > -1$. 
We connect $-1$ to $a_{\varepsilon}$, by constant functions, 
\begin{equation*}
    u_{\varepsilon}^{r} (y) = r
\end{equation*}
for $r$ in $[-1, a_{\varepsilon}]$. 
Then, 
\begin{equation*}
    \mathcal{F}_{\varepsilon, \lambda} (u_{\varepsilon}^{r}) = \int_{N} \frac{W(r)}{\varepsilon} - \sigma \lambda r \, d \mu_{g} \leq \mathcal{F}_{\varepsilon, \lambda} (u_{\varepsilon}^{-1}). 
\end{equation*}
As $u_{\varepsilon}^{- 1} = \tilde{v}_{\varepsilon}^{2 \, \text{diam}(N), g_{5}}$, we have that, for all $r$ in $[- 1, a_{\varepsilon}]$, 
\begin{equation*}
    \mathcal{F}_{\varepsilon, \lambda} (u_{\varepsilon}^{r}) < \mathcal{F}_{\varepsilon, \lambda} (v_{\varepsilon}) - \frac{\sigma \mathcal{H}^{n}(A_{l})}{2(2^{n} - 1)} + \rom{3}_{\varepsilon}^{1,0} + \rom{3}_{\varepsilon}^{4, r_{0}} + \rom{3}_{\varepsilon}^{5, r_{0}} + \rom{3}_{\varepsilon}^{6,2 \, \text{diam}(N)}.
\end{equation*}

\section{Path to $b_{\varepsilon}$}\label{sec: path to plus 1}

\subsection{Lining Up With Level Set $\Gamma_{-\rho}$}\label{sec: Crab Legs Back, Main Step to 1}

We construct the path from (2) to (6) in Figure \ref{fig: The Path}

\bigskip 

We consider, for $r$ in $[0, \rho]$, and $x$ in $\tilde{M}$,
\begin{equation*}
    g_{6} (r, x) = \rho f(x) - r (1 + f(x)).
\end{equation*}

First consider $r$, in $(2 \varepsilon \Lambda, \rho]$, 
\begin{equation*}
    \mathcal{F}_{\varepsilon, \lambda} (v_{\varepsilon}^{r, g_{6}}) - \mathcal{F}_{\varepsilon, \lambda} (v_{\varepsilon}^{0, g_{6}}) = \rom{1}_{\varepsilon}^{r, g_{6}} + (\rom{2}_{\varepsilon}^{r, g_{6}} - \rom{2}_{\varepsilon}^{2 \varepsilon \Lambda, g_{6}}) + \rom{2}_{\varepsilon}^{2 \varepsilon \Lambda, g_{6}}
\end{equation*}

Similar to Section \ref{sec: Crab Legs Forward Away From Body} we have,
\begin{equation*}
    \rom{2}_{\varepsilon}^{r, g_{6}} - \rom{2}_{\varepsilon}^{2 \varepsilon \Lambda, g_{6}} \leq \lambda \int_{2 \varepsilon \Lambda}^{r} \int_{\tilde{M}} \Theta_{\varepsilon, g_{6}}^{1} (s,x) - \Theta_{\varepsilon, g_{6}}^{2} (s,x) \, d \mathcal{H}^{n} (x) \, ds. 
\end{equation*}

For $r$ in $[0, 2 \varepsilon \Lambda]$, again by similar arguments to those in Section \ref{sec: Crab Legs Forward Away From Body}
\begin{eqnarray*}
    \rom{2}_{\varepsilon}^{r, g_{6}} &\leq& C_{2} \bigg( \int_{0}^{r} \int_{\{\rho f < - 2 \varepsilon \Lambda\}} \Theta_{\varepsilon, g_{6}}^{1} (s,x) - \Theta_{\varepsilon, g_{6}}^{2} (s,x) \, d \mathcal{H}^{n} (x) \, ds \\
    && + \int_{\{\rho f \geq - 2 \varepsilon \Lambda\}} m_{\varepsilon}^{2} (x) \, d \mathcal{H}^{n} (x) + \varepsilon \Lambda \bigg),
\end{eqnarray*}
where, 
\begin{equation*}
    m_{\varepsilon}^{2} (x) \coloneqq \max_{t \in [- 6 \varepsilon \Lambda. 2 \varepsilon \Lambda]} \theta_{t} (x) - \min_{t \in [-6 \varepsilon \Lambda, 2 \varepsilon \Lambda]} \theta_{t} (x),
\end{equation*}
and we are potentially rechoosing $C_{2}(M, N, g, W, \lambda)$.

\bigskip 

For $r$ in $[0,\rho]$, we consider, 
\begin{eqnarray*}
    \rom{1}_{\varepsilon}^{r, g_{6}} &=& \int_{A_{l}} \int_{\sigma^{-}(x)}^{\sigma^{+}(x)} \frac{\varepsilon}{2} ((\overHe_{\varepsilon})' (t - g_{4} (r,x)))^{2} (\rho - r)^{2} |\nabla f (x)|_{(x,t)}^{2} \theta_{t}(x) \, dt \, d \mathcal{H}^{n}(x) \\ 
    && - \int_{A_{l}} \int_{\sigma^{-}(x)}^{\sigma^{+}(x)} \frac{\varepsilon}{2} ((\overHe_{\varepsilon})' (t - g_{4} (0,x)))^{2} \rho^{2} |\nabla f (x)|_{(x,t)}^{2} \theta_{t}(x) \, dt \, d \mathcal{H}^{n}(x)
\end{eqnarray*}
Following similar arguments to Section \ref{subsec: Path From Disks to Crab, calculation}, and after potentially rechoosing $C_{3} = C_{3} (z_{0}, M, N, g, \delta, W, \lambda)$, we have that,
\begin{equation*}
    \rom{1}_{\varepsilon}^{r, g_{6}} \leq C_{3} \mathcal{H}^{n} (A_{l}) \frac{\rho^{2}}{l^{2}}.
\end{equation*}
Therefore, recalling our choice of $\rho > 0$ in Remark \ref{rem: choice of rho for path from disks to crab}, we have
\begin{equation*}
    \rom{1}_{\varepsilon}^{r, g_{6}} \leq \frac{\sigma \mathcal{H}^{n} (A_{l})}{2(2^{n} - 1)}.
\end{equation*}

Thus, for $r$ in $[0, 2 \varepsilon \Lambda]$, 
\begin{eqnarray*}
    \mathcal{F}_{\varepsilon, \lambda}(v_{\varepsilon}^{r, g_{6}}) - \mathcal{F}_{\varepsilon}(v_{\varepsilon}^{0, g_{6}}) &<& \frac{\sigma \mathcal{H}^{n} (A_{l})}{2(2^{n} - 1)} + \rom{3}_{\varepsilon}^{7,r},
\end{eqnarray*}
where, 
\begin{eqnarray*}
    \rom{3}_{\varepsilon}^{7,r} &=& C_{2} \bigg( \int_{0}^{r} \int_{\{\rho f < - 2 \varepsilon \Lambda\}} \Theta_{\varepsilon, g_{6}}^{1} (s,x) - \Theta_{\varepsilon, g_{6}}^{2} (s,x) \, d \mathcal{H}^{n} (x) \, ds \\
    && + \int_{\{\rho f \geq - 2 \varepsilon \Lambda\}} m_{\varepsilon}^{2} (x) \, d \mathcal{H}^{n} (x) + \varepsilon \Lambda \bigg).
\end{eqnarray*}
For $r$ in $(2 \varepsilon \Lambda, \rho]$, 
\begin{eqnarray*}
    \mathcal{F}_{\varepsilon, \lambda}(v_{\varepsilon}^{r, g_{6}}) - \mathcal{F}_{\varepsilon}(v_{\varepsilon}^{0, g_{6}}) &<& \frac{\sigma \mathcal{H}^{n} (A_{l})}{2(2^{n} - 1)} + \rom{3}_{\varepsilon}^{8, r},
\end{eqnarray*}
where
\begin{eqnarray*}
    \rom{3}_{\varepsilon}^{8,r} &=& C_{2} \bigg( \int_{0}^{2 \varepsilon \Lambda} \int_{\{\rho f < - 2 \varepsilon \Lambda\}} \Theta_{\varepsilon, g_{6}}^{1} (s,x) - \Theta_{\varepsilon, g_{6}}^{2} (s,x) \, d \mathcal{H}^{n} (x) \, ds \\
    && + \int_{2 \varepsilon \Lambda}^{r} \int_{\tilde{M}} \Theta_{\varepsilon, g_{6}}^{1} (s, x) - \Theta_{\varepsilon, g_{6}}^{2} (s, x) \, d \mathcal{H}^{n} (x) \\
    && + \int_{\{\rho f \geq - 2 \varepsilon \Lambda\}} m_{\varepsilon}^{2} (x) \, d \mathcal{H}^{n} (x) + \varepsilon \Lambda \bigg).
\end{eqnarray*}

\bigskip

As $g_{6} (0, x) = g_{1} (\rho, x)$, we have, for $r$ in $[0, 2 \varepsilon \Lambda]$, 
\begin{equation*}
    \mathcal{F}_{\varepsilon, \lambda} (v_{\varepsilon}^{r, g_{6}}) \leq \mathcal{F}_{\varepsilon, \lambda} (v_{\varepsilon}) - \frac{\sigma \mathcal{H}^{n} (A_{l})}{2(2^{n} - 1)} + \rom{3}_{\varepsilon}^{2, \rho} + \rom{3}_{\varepsilon}^{7, r},
\end{equation*}
and for $r$ in $(2 \varepsilon \Lambda, \rho]$, we have, 
\begin{equation*}
    \mathcal{F}_{\varepsilon, \lambda} (v_{\varepsilon}^{r, g_{6}}) \leq \mathcal{F}_{\varepsilon, \lambda} (v_{\varepsilon}) - \frac{\sigma \mathcal{H}^{n} (A_{l})}{2(2^{n} - 1)} + \rom{3}_{\varepsilon}^{2, \rho} + \rom{3}_{\varepsilon}^{8, r}.
\end{equation*}

\bigskip 

For $r$ in $[0, \rho]$, we define the following function on $N$,
\begin{equation*}
    \tilde{v}_{\varepsilon}^{r,g_{6}} (y) = \begin{cases}
        \overHe_{\varepsilon} (\tilde{d}(y) + r), & y \not\in B_{\delta}^{N} (z_{0}), \\
        v_{\varepsilon}^{r,g_{6}} (F_{1}^{-1} (y)), & y \in \overline{\Upsilon_{1}} \cap B_{\delta}^{N} (z_{0}), \\
        v_{\varepsilon}^{r, g_{6}} (F_{2}^{-1} (y)), & y \in \overline{\Upsilon_{2}} \cap B_{\delta}^{N} (z_{0}).
    \end{cases}
\end{equation*}
We can show, as in Section \ref{subsec: Function on Manifold Disks to Crab} and \ref{subsec: Continuity of path from disks to crab} that, $\tilde{v}_{\varepsilon}^{r,g_{6}}$ lies in $W^{1,\infty} (N)$, $\mathcal{F}_{\varepsilon, \lambda} (\tilde{v}_{\varepsilon}^{r, g_{6}}) (N) = \mathcal{F}_{\varepsilon, \lambda} (v_{\varepsilon}^{r,g_{6}}) (\tilde{T})$, and the path $r \mapsto \tilde{v}_{\varepsilon}^{r,g_{6}}$ is continuous in $W^{1,2} (N)$.

\subsection{Completing Path to $b_{\varepsilon}$}\label{sec: Completing Path to +1}

We construct the path from (6) to '+1' in Figure \ref{fig: The Path}.
This is done in an identical way to Section \ref{sec: Completing Path to -1}.

\bigskip 

For $r$ in $[\rho, 2 \, \text{diam}(N)]$, we define the following function on $N$, 
\begin{equation*}
    \tilde{v}_{\varepsilon}^{r,g_{7}} (y) \coloneqq \overHe_{\varepsilon} (\tilde{d}(y) + r).
\end{equation*}
Similar to arguments in Section \ref{sec: Completing Path to -1} we have, 
\begin{equation*}
    \mathcal{F}_{\varepsilon, \lambda} (\tilde{v}_{\varepsilon}^{r, g_{7}}) \leq \mathcal{F}_{\varepsilon, \lambda} (v_{\varepsilon}) - \frac{\sigma \mathcal{H}^{n} (A_{l})}{2(2^{n} - 1)} + \rom{3}_{\varepsilon}^{2, \rho} + \rom{3}_{\varepsilon}^{8, \rho} + \rom{3}_{\varepsilon}^{9,r},
\end{equation*}
where, 
\begin{equation*}
    \rom{3}_{\varepsilon}^{9,r} = \lambda \int_{\rho}^{r} \Theta_{\varepsilon, g_{6}}^{1} (s,x) - \Theta_{\varepsilon, g_{6}}^{2} (s,x) \, d \mathcal{H}^{n} (x) \, ds.
\end{equation*}

\bigskip 

Again as in Section \ref{sec: Completing Path to -1}, we connect $\tilde{v}_{\varepsilon}^{2 \, \text{diam}(N), g_{7}} = 1$, to $b_{\varepsilon}$, by constant functions, $u_{\varepsilon}^{r} = r$, for $r$ in $[1, b_{\varepsilon}]$. 
We have that for all $r$ in $[1, b_{\varepsilon}]$, 
\begin{equation*}
    \mathcal{F}_{\varepsilon, \lambda} (u_{\varepsilon}^{r}) \leq \mathcal{F}_{\varepsilon, \lambda} (\tilde{v}_{\varepsilon}^{2 \, \text{diam}(N), g_{7}}) \leq \mathcal{F}_{\varepsilon, \lambda} (v_{\varepsilon}) - \frac{\sigma \mathcal{H}^{n} (A_{l})}{2(2^{n} - 1)} + \rom{3}_{\varepsilon}^{2, \rho} + \rom{3}_{\varepsilon}^{8, \rho} + \rom{3}_{\varepsilon}^{9,2 \, \text{diam} (N)}.
\end{equation*}

\bigskip 

Both $\tilde{v}_{\varepsilon}^{r, g_{7}}$, and $u_{\varepsilon}^{r}$ give continuous paths in $W^{1,2} (N)$ with respect to $r$.

\section{Conclusion of the Paths}

\subsection{Error Terms}\label{sec: Error Terms}

\subsubsection{Theta Error Terms}\label{subsec: theta error terms}

Consider a function $g \colon \mathbb{R} \times \tilde{M} \rightarrow \mathbb{R}$, and the term  
\begin{eqnarray*}
    \Theta_{\varepsilon, g}^{1} (s, x) - \Theta_{\varepsilon, g}^{2} (s,x) &=& \sigma \int_{\sigma^{-}(x)}^{\sigma^{+}(x)} \partial_{s} g (s,x) (\overHe_{\varepsilon})' (t - g(s,x)) \theta_{t}(x) \, dt \\ 
    && \hspace{1cm} - \int_{\sigma^{-}(x)}^{\sigma^{+}(x)} \partial_{s} g(s,x) Q_{\varepsilon} (t - g(s,x)) \theta_{t} (x) \, dt.
\end{eqnarray*}
Assuming that $g$ is monotone in the first variable, we have,
\begin{eqnarray*}
    |\Theta_{\varepsilon, g}^{1} (s, x) - \Theta_{\varepsilon, g}^{2} (s,x)| \leq 2 \sigma |\partial_{s} g(s,x)| m_{\varepsilon}(g(s,x), x) + C_{7} \varepsilon^{2}.
\end{eqnarray*}
where, 
\begin{equation*}
    m_{\varepsilon} (T, x) = \max_{t \in [T - 2 \varepsilon \Lambda, T + 2 \varepsilon \Lambda]} \theta_{t} (x) - \min_{t \in [T - 2 \varepsilon \Lambda, T + 2 \varepsilon \Lambda]} \theta_{t} (x).
\end{equation*}
and $C_{7} = C_{7} (N, m, \lambda, W, |g|_{C^{1}}) < + \infty$.

\bigskip 

Now we assume that $\partial_{s} g \geq 0$, and $|\partial_{s} g|_{C^{0} (\mathbb{R} \times \tilde{M})} < + \infty$. 
Apply Fubini's Theorem to swap integrals,
\begin{eqnarray*}
    \int_{0}^{r} \int_{\tilde{M}} \left| \Theta_{\varepsilon, g}^{1} (s, x) - \Theta_{\varepsilon, g}^{2} (s,x) \right| \, d \mathcal{H}^{n} (x) \, ds \leq 2 \sigma \int_{\tilde{M}} \int_{g(0,x)}^{g(r,x)} m_{\varepsilon} (T, x) \, dT \, \mathcal{H}^{n} (x) + C_{7} r \varepsilon^{2}.
\end{eqnarray*}
Fixing $x$ in $\tilde{M}$, we see that for all $T$ in $\mathbb{R} \setminus \{ \sigma^{-}(x), \sigma^{+}(x) \}$, 
\begin{equation*}
    m_{\varepsilon} (T, x) \rightarrow 0, \, as \, \varepsilon \rightarrow 0,
\end{equation*}
and furthermore, we have the following bounds, $0 \leq m_{\varepsilon} (T, x) \leq e^{\frac{\lambda^{2}}{2m}}$.
Therefore, we can apply Dominated Convergence Theorem for fixed $x$ in $\tilde{M}$ and $r$ in $[0, \infty)$, 
\begin{equation*}
    \int_{g(0,x)}^{g(r,x)} m_{\varepsilon} (T, x) \, dT \rightarrow 0, \, as \, \varepsilon \rightarrow 0.
\end{equation*}
Furthermore, as $0 \leq g(r,x) - g(0, x) \leq |\partial_{s} g|_{C^{0} (\mathbb{R} \times \tilde{M})} \, r$, we have the bounds, 
\begin{equation*}
    0 \leq \int_{g(0,x)}^{g(r,x)} m_{\varepsilon} (T, x) \, dT \leq |\partial_{s} g|_{C^{0} (\mathbb{R} \times \tilde{M})} r e^{\frac{\lambda^{2}}{2m}}. 
\end{equation*}
Therefore, again by Dominated Convergence Theorem, we have, for fixed $r$ in $[0, \infty)$
\begin{equation*}
    \int_{\tilde{M}} \int_{g(0,x)}^{g(r,x)} m_{\varepsilon} (T, x) \, dT \, \mathcal{H}^{n} (x) \rightarrow 0, \, as \, \varepsilon \rightarrow 0.
\end{equation*}

Define the following continuous function on $[0, +\infty)$,
\begin{equation*}
    M_{\varepsilon}^{g} (r) = \int_{\tilde{M}} \int_{g(0,x)}^{g(r,x)} m_{\varepsilon} (T, x) \, dT \, \mathcal{H}^{n} (x).
\end{equation*}
We have that $M_{\varepsilon}^{g} (r) \rightarrow 0$, pointwise, as $\varepsilon \rightarrow 0$, and furthermore, as 
\begin{equation*}
    0 \leq m_{\varepsilon_{1}} (T,x) \leq m_{\varepsilon_{2}} (T,x),
\end{equation*} 
for all $T$ in $\mathbb{R}$, $x$ in $\tilde{M}$, and $0 < \varepsilon_{1} < \varepsilon_{2}$, this implies that, 
\begin{equation*}
    0 \leq M_{\varepsilon_{1}}^{g} (r) \leq M_{\varepsilon_{2}}^{g}(r),
\end{equation*}
for all $r$ in $[0, + \infty)$. 
Therefore, by Dini's Theorem, we have that, 
\begin{equation*}
    M_{\varepsilon}^{g} \rightarrow 0, \, as \, \varepsilon \rightarrow 0,
\end{equation*}
uniformly on compact sets of $[0, + \infty)$.
Thus,
\begin{equation}\label{eqn: Theta error converging to zero}
    \int_{0}^{r} \int_{\tilde{M}} \left| \Theta_{\varepsilon, g}^{1} (s, x) - \Theta_{\varepsilon, g}^{2} (s,x) \right| \, d \mathcal{H}^{n} (x) \, ds \rightarrow 0
\end{equation}
as $\varepsilon \rightarrow 0$, uniformly in $r$, on compact sets of $[0, + \infty)$.
The same holds assuming that $g$ satisfies $\partial_{s} g \leq 0$, on $\mathbb{R} \times \tilde{M}$, and $|\partial_{s} g|_{C^{0} (\mathbb{R} \times \tilde{M})} < + \infty$.

\bigskip

For $i = 1, \ldots, 7$ our $g_{i}$'s are monotone in the first variable and $|\partial_{s} g_{i}|_{C^{0} (\mathbb{R} \times \tilde{M})} < + \infty$.
Therefore, (\ref{eqn: Theta error converging to zero}) holds for each $i$.

\subsubsection{The Other Error Terms}

We first consider, 
\begin{equation*}
    \int_{B_{2l}} q_{\varepsilon}^{1} (x) \, d \mathcal{H}^{n} (x),
\end{equation*}
with, 
\begin{equation*}
    q_{\varepsilon}^{1} (x) = \max_{t \in [-4 \varepsilon \Lambda, 2 \varepsilon \Lambda]} (H_{t} (x) - \lambda) \theta_{t} (x).
\end{equation*}
By choice of $\varepsilon > 0$, in Remark \ref{rem: choice in epsilon 1}, $2 \varepsilon \Lambda < < \rho$. 
Therefore, by choice of $\rho > 0$, in Remark \ref{rem: choice of rho so that B l times -rho rho fits into B delta}, and $\delta > 0$, from Remark \ref{rem: choice of delta for upper and lower mean curvature bounds in ball about z 0}, we have 
\begin{equation*}
    0 \leq \max_{x \in B_{2 l}} q_{\varepsilon}^{1} (x) \leq \frac{\lambda}{2} e^{\frac{\lambda^{2}}{2 m}}.
\end{equation*}
Fixing $x'$ in $B_{2 l} \setminus \{x \colon \sigma^{-} (x) = 0 \}$, we see that there exists an $\varepsilon' = \varepsilon' (x') > 0$, such that for all $0 < \varepsilon \leq \varepsilon'$, 
\begin{equation*}
    [-4 \varepsilon \Lambda, 2 \varepsilon \Lambda] \subset (\sigma^{-} (x'), \sigma^{+}(x')).
\end{equation*}
Therefore, $(H_{t} (x') - \lambda) \theta_{t} (x')$, is a smooth function in $t$ on $[- 4 \varepsilon \Lambda, 2 \varepsilon \Lambda]$, and clearly, 
\begin{equation*}
    \max_{t \in [-4 \varepsilon \Lambda, 2 \varepsilon \Lambda]} (H_{t} (x') - \lambda) \theta_{t} (x') \rightarrow 0, \, as \, \varepsilon \rightarrow 0.
\end{equation*}
Thus, $q_{\varepsilon}^{1} \rightarrow 0$, $\mathcal{H}^{n}$--a.e in $B_{2l}$, and we can apply Dominated Convergence Theorem to say that 
\begin{equation*}
    \int_{B_{2 l}} q_{\varepsilon}^{1} (x) \, d \mathcal{H}^{n} (x) \rightarrow 0, \, as \, \varepsilon \rightarrow 0.
\end{equation*}
Identically we also have, 
\begin{equation*}
    \int_{A_{L}^{k}} q_{\varepsilon}^{2} (x) \, d \mathcal{H}^{n} (x) \rightarrow 0, \, as \, \varepsilon \rightarrow 0,
\end{equation*}
recalling $q_{\varepsilon}^{2} (x) = \max_{t \in [-2 \varepsilon \Lambda, 4 \varepsilon \Lambda]} (\lambda - H_{t} (x)) \theta_{t} (x)$.

\bigskip 

Now considering 
\begin{equation*}
    q_{\varepsilon}^{3} (r,x) = \max_{t \in [-2 \varepsilon \Lambda, 2 \varepsilon \Lambda]} \zeta(x, t + g_{4} (r,x)) - \zeta (x, g_{4} (r,x))
\end{equation*}
where we are recalling the function
\begin{equation*}
    \zeta (x,t) = |\nabla \tilde{f}(x)|_{(x,t)}^{2} \theta_{t} (x)
\end{equation*}
from Claim \ref{lem: Second choice of r 0}. 
For $x$ in $B_{2 L}$ such that $\tilde{f} (x) = 0$, we have that $\zeta (x,t) = 0$, for all $t$. 
Considering $x$ in $A_{L}^{k} \cap \{ \tilde{f} \not= 0\}$, such that $r_{0} \tilde{f} (x) > 2 \varepsilon \Lambda$, then
\begin{equation*}
    [-2 \varepsilon \Lambda + g_{4} (r,x), 2 \varepsilon \Lambda + g_{4} (r,x)] \subset (0, 2 r_{0}) \subset (\sigma^{-} (x), \sigma^{+} (x)).
\end{equation*}
Thus, for $t$ in $[-2 \varepsilon \Lambda + g_{4} (r,x), 2 \varepsilon \Lambda + g_{4} (r,x)]$, 
\begin{equation*}
    (x,t) \in \tilde{T} \cap ( B_{L} \times (-2 r_{0}, 2 r_{0}) ) \subset \subset \tilde{V}_{1} \cup \tilde{V}_{2}, 
\end{equation*}
where we are recalling sets $\tilde{V}_{1}$ and $\tilde{V}_{2}$ from Remark \ref{rem: diffeos F i}. 
Therefore, $\zeta$ is differentiable at $(x,t)$ and 
\begin{equation*}
    |\partial_{t} \zeta (x,t)| \leq C(|F_{i}|_{C^{2}(B_{L} \times (-2 r_{0}, 2 r_{0}) )}, |\tilde{f}|_{C^{1}}, \lambda) \leq C_{6}.
\end{equation*}
Where we are potentially rechoosing $C_{6} = C_{6} (z_{0}, M, N, g, k, L, \delta, W, \lambda)$.
Therefore, for $x$ in $A_{L}^{k}$ such that $r_{0} \tilde{f} (x) > 2 \varepsilon \Lambda$, we have that $0 \leq q_{\varepsilon}^{3} (r,x) \leq C_{6} \varepsilon \Lambda$. 
Furthermore, for all $x$ in $A_{L}^{k}$, 
\begin{equation*}
    |q_{\varepsilon}^{3} (r,x)| \leq \max_{(x,t) \in B_{L} \times (-2 r_{0}, 2 r_{0})} \zeta (x, t) \leq C_{6}.
\end{equation*}
Again we are potentially rechoosing $C_{6} = C_{6} (z_{0}, M, N, g, k, L, \delta, W, \lambda)$.

\bigskip 

Therefore, 
\begin{eqnarray*}
    p_{\varepsilon}^{1} (r) &=& \int_{A_{L}^{k} \cap \{ r_{0} \tilde{f} > 2 \varepsilon \Lambda\}}  q_{\varepsilon}^{3} (r,x) \, d \mathcal{H}^{n} (x) + \int_{A_{L}^{k} \cap \{0 < r_{0} \tilde{f} \leq 2 \varepsilon \Lambda\}}  q_{\varepsilon}^{3} (r,x) \, d \mathcal{H}^{n} (x), \\
    &\leq& C_{6} \left( \varepsilon \Lambda + \mathcal{H}^{n} (\{ x \in A_{L}^{k} \colon 0 < r_{0} \tilde{f} (x) \leq 2 \varepsilon \Lambda\}) \right).
\end{eqnarray*}
Thus,
\begin{equation*}
    \max_{r \in [0, r_{0}]} p_{\varepsilon}^{1} (r) \rightarrow 0
\end{equation*}
as $\varepsilon \rightarrow 0$.
Similarly, $p_{\varepsilon}^{2} (0) \rightarrow 0$, as $\varepsilon \rightarrow 0$.

\bigskip 

For the remaining error terms, as $\mathcal{H}^{n} (\{ x \in \tilde{M} \colon \, \sigma^{-} (x) = 0\}) = 0$, by Dominated Convergence Theorem, we have that, 
\begin{equation*}
    \mathcal{H}^{n} (\{ x \in \tilde{M} \colon \, \sigma^{-}(x) \geq - 2 \varepsilon \Lambda \}) \rightarrow 0,
\end{equation*}
and 
\begin{equation*}
    \int_{\tilde{M}} m_{\varepsilon}^{i} (x) \, d \mathcal{H}^{n} (x) \rightarrow 0,
\end{equation*}
where, 
\begin{eqnarray*}
    m_{\varepsilon}^{1} (x) &=& \max_{t \in [-2 \varepsilon \Lambda, 4 \varepsilon \Lambda]} \theta_{t} (x) - \min_{t \in [-2 \varepsilon \Lambda, 4 \varepsilon \Lambda]} \theta_{t} (x), \\
    m_{\varepsilon}^{2} (x) &=& \max_{t \in [-6 \varepsilon \Lambda, 2 \varepsilon \Lambda]} \theta_{t} (x) - \min_{t \in [-6 \varepsilon \Lambda, 2 \varepsilon \Lambda]} \theta_{t} (x).
\end{eqnarray*}

\subsection{Path for Theorem \ref{thm: Allen-Cahn minmax limit}}\label{subsec: Recovery Path}

Consider the following continuous path in $W^{1,2} (N)$, for $\varepsilon > 0$,
\begin{equation*}
    \gamma_{\varepsilon} (t) = \begin{cases}
        -1 - 2 \text{diam} \, (N) - t, & t \in [- 2 \, \text{diam} \, (N) - a_{\varepsilon} - 1, 2 \, \text{diam} \, (N)], \\
        \overHe_{\varepsilon} (\tilde{d} - t), & t \in [-2 \, \text{diam} \, (N), 2 \, \text{diam} \, (N)], \\
        1 - 2 \, \text{diam} \, (N) + t, & t \in [2 \, \text{diam} \, (N), 2 \, \text{diam} \, (N) + b_{\varepsilon} - 1].
    \end{cases}
\end{equation*}
which satisfies $\gamma_{\varepsilon} (-1 - 2 \text{diam} \, (N) - a_{\varepsilon}) = a_{\varepsilon}$, and $\gamma_{\varepsilon} (1 - 2 \, \text{diam} \, (N) + b_{\varepsilon}) = b_{\varepsilon}$.

\bigskip 

Replacing $r_{0} = 2 \varepsilon \Lambda$, in Section \ref{sec: Completing Path to -1}, and $\rho = 2 \varepsilon \Lambda$, in Section \ref{sec: Completing Path to +1}, we see that, for all $\varepsilon$ in $(0, \tilde{\varepsilon})$, for some $\tilde{\varepsilon} = \tilde{\varepsilon} (N, M, g, \lambda, W) > 0$, fixed,
\begin{equation*}
\begin{cases}
    \mathcal{F}_{\varepsilon, \lambda} (\gamma_{\varepsilon} (t)) < \mathcal{F}_{\varepsilon, \lambda} (v_{\varepsilon}) + \rom{3}_{\varepsilon}^{6,2 \, \text{diam} \, (N)}, & t \in [- 2 \, \text{diam} \, (N) - a_{\varepsilon} - 1, 2 \, \text{diam} \, (N)], \\
    \mathcal{F}_{\varepsilon, \lambda} (\gamma_{\varepsilon} (t)) < \mathcal{F}_{\varepsilon, \lambda} (v_{\varepsilon}) + \rom{3}_{\varepsilon}^{6, - t}, &  t \in [- 2 \text{diam} \, (N), - 2 \varepsilon \Lambda], \\
    \mathcal{F}_{\varepsilon, \lambda} (\gamma_{\varepsilon} (t)) < \mathcal{F}_{\varepsilon, \lambda} (v_{\varepsilon}) + \rom{3}_{\varepsilon}^{9,t}, & t \in [2 \varepsilon \Lambda, 2 \text{diam} \, (N)], \\
    \mathcal{F}_{\varepsilon, \lambda} (\gamma_{\varepsilon} (t)) < \mathcal{F}_{\varepsilon, \lambda} (v_{\varepsilon}) + \rom{3}_{\varepsilon}^{9,2 \, \text{diam} \, (N)}, & t \in [2 \, \text{diam} \, (N), 2 \, \text{diam} \, (N) + b_{\varepsilon} - 1].
\end{cases}
\end{equation*}
Recalling from Section \ref{subsec: Approximating function for CMC}
\begin{equation*}
    \mathcal{F}_{\varepsilon, \lambda} (v_{\varepsilon}) \rightarrow 2 \sigma \mathcal{H}^{n} (M) - \sigma \lambda \mu_{g} (E) + \sigma \lambda \mu_{g} (N \setminus E),
\end{equation*}
as $\varepsilon \rightarrow 0$, and Section \ref{subsec: theta error terms}, 
\begin{equation*}
    \max_{t \in [2 \varepsilon \Lambda, 2 \, \text{diam} \, (N)]} \left( \rom{3}_{\varepsilon}^{6, t} + \rom{3}^{9, t} \right) \rightarrow 0,
\end{equation*}
as $\varepsilon \rightarrow 0$.
Therefore, for $\tau > 0$, there exists a $ 0 < \varepsilon_{\tau} = \varepsilon_{\tau} (N, M, g, \lambda, W) \leq \tilde{\varepsilon}$, such that for all $\varepsilon$ in $(0, \varepsilon_{\tau})$ and $t$ in $[- 2 \, \text{diam} \, (N) - a_{\varepsilon} - 1, 2 \, \text{diam} \, (N) + b_{\varepsilon} - 1] \setminus (-2 \varepsilon \Lambda, 2 \varepsilon \Lambda)$, 
\begin{equation*}
    \mathcal{F}_{\varepsilon, \lambda} (\gamma_{\varepsilon} (t) ) < 2 \sigma \mathcal{H}^{n} (M) - \sigma \lambda \mu_{g} (E) + \sigma \lambda \mu_{g} (N \setminus E) + \tau.
\end{equation*}
Furthermore, by similar arguments to those in Section \ref{subsec: Approximating function for CMC}, and after potentially rechoosing $\varepsilon_{\tau} > 0$, we have that for all $\varepsilon$ in $(0, \varepsilon_{\tau})$
\begin{equation*}
    \max_{t \in [-2 \varepsilon \Lambda, 2 \varepsilon \Lambda]} \mathcal{F}_{\varepsilon, \lambda} (\gamma_{\varepsilon} (t)) < 2 \sigma \mathcal{H}^{n} (M) - \sigma \lambda \mu_{g} (E) + \sigma \lambda \mu_{g} (N \setminus E) + \tau.
\end{equation*} 
Therefore, this is an admissible path in $W^{1,2} (N)$, that proves that for the limiting varifold $V = V_{\lambda} + V_{0}$, we must have that $V_{0} = 0$.
This completes the proof of Theorem \ref{thm: Allen-Cahn minmax limit}.

\begin{rem}\label{rem: suitable properties for path}
Note that we can build the path $\gamma_{\varepsilon}$, for any suitable Caccioppoli set $E$. 
The suitable properties are the following:
\begin{enumerate}
    \item $\partial^{*} E \not= \emptyset$, has a quasi embedded $\lambda$-CMC structure, with respect to unit normal pointing into $E$. 
    \item $\partial^{*} E$ satisfies the Geodesic Touching Lemma (Lemma \ref{lem: end points of length minimising geodesics to M are smooth points}). 
\end{enumerate}
\end{rem}

\bigskip 

From Remark \ref{rem: suitable properties for path} we can deduce that $E$ must be a single, connected component and minimises the value
\begin{equation*}
    F_{\lambda} (E) = \mathcal{H}^{n} (\partial^{*} E) - \lambda \mu_{g} (E) > 0,
\end{equation*}
among all suitable competitors.

\subsection{Contradiction Path for Theorem \ref{thm: Main Theorem}}\label{sec: Pieceing Path Together}

Recall all the error terms from Sections \ref{sec: Path From Disks to Crab}, \ref{sec: path to minus 1} and \ref{sec: path to plus 1}.
By Section \ref{subsec: Approximating function for CMC} and \ref{sec: Error Terms}, for $\tau > 0$, there exists an $\varepsilon_{\tau} = \varepsilon (z_{0}, M, N, g, \delta, W, \lambda, L, k, r_{0}, \rho, \tau) \in (0, \varepsilon_{3})$, such that for all $\varepsilon$ in $(0, \varepsilon_{\tau})$, we have that
\begin{eqnarray*}
    && \mathcal{F}_{\varepsilon, \lambda} (v_{\varepsilon}) + \max_{r \in [0, 4 \varepsilon \Lambda)} \rom{3}_{\varepsilon}^{1,r} 
    + \max_{r \in [4 \varepsilon \Lambda, \rho]} \rom{3}_{\varepsilon}^{2,r} 
    + \max_{r \in [0, 2 \varepsilon \Lambda]} \rom{3}_{\varepsilon}^{3, r} \\
    && \hspace{1cm} + \max_{r \in (2 \varepsilon \Lambda, r_{0}]} \rom{3}_{\varepsilon}^{4,r} 
    + \max_{r \in [0, r_{0}]} \rom{3}_{\varepsilon}^{5,r} 
    + \max_{r \in [r_{0}, 2 \, \text{diam} \, (N)]} \rom{3}_{\varepsilon}^{6,r} \\
    && \hspace{2cm} + \max_{r \in [0, 2 \varepsilon \Lambda]} \rom{3}_{\varepsilon}^{7,r} + \max_{r \in (2 \varepsilon \Lambda, \rho]} \rom{3}_{\varepsilon}^{8,r} + \max_{r \in [\rho, 2 \, \text{diam} \, (N)]} \rom{3}_{\varepsilon}^{9,r} \\
    && \hspace{3cm} < 2 \sigma \mathcal{H}^{n} (M) - \sigma \lambda \mu_{g} (E) + \sigma \lambda \mu_{g} (N \setminus E) + \tau.
\end{eqnarray*}

\bigskip 

Therefore, for any $\tau > 0$, there exists an $\varepsilon_{\tau} > 0$, such that for any $\varepsilon$ in $(0, \varepsilon_{\tau})$, we can define the continuous path, 
\begin{equation*}
    \gamma_{\varepsilon} \colon [-1 - a_{\varepsilon}, 4 \, \text{diam}\, (N) + r_{0} + \rho + b_{\varepsilon} - 1] \rightarrow W^{1,2} (N),
\end{equation*}
by
\begin{equation*}
    \gamma_{\varepsilon} (t) = \begin{cases}
        - 1 - t, & t \in [-1 - a_{\varepsilon}, 0], \\
        \overHe_{\varepsilon} (\tilde{d} + t - 2 \, \text{diam} \, (N)), & [0, 2 \, \text{diam} \, (N) - r_{0}], \\
        \tilde{v}_{\varepsilon}^{2 \, \text{diam} \, (N) - t, g_{4}}, & [2 \, \text{diam} \, (N) - r_{0}, 2 \, \text{diam} \, (N)], \\
        \tilde{v}_{\varepsilon}^{2 \, \text{diam} \, (N) + \rho - t , g_{3}}, & [2 \, \text{diam} \, (N), 2 \, \text{diam} \, (N) + \rho], \\
        \tilde{v}_{\varepsilon}^{2 \, \text{diam} \, (N) + \rho +r_{0} - t, g_{2}}, & [2 \, \text{diam} \, (N) + \rho, 2 \, \text{diam} \, (N) + \rho + r_{0}], \\
        \tilde{v}_{\varepsilon}^{t - (2 \, \text{diam} \, (N) + \rho + r_{0}),g_{6}}, & [2 \, \text{diam} \, (N) + \rho + r_{0}, 2 \, \text{diam} \, (N) + 2 \rho + r_{0}], \\
        \overHe_{\varepsilon} (\tilde{d} + t - (2 \, \text{diam} \, (N) + \rho + r_{0})), & [2 \, \text{diam} \, (N) + 2 \rho + r_{0}, 4 \, \text{diam} \, (N) + \rho + r_{0}], \\
        1 + t - (4 \, \text{diam}\, (N) + \rho + r_{0}), & [4 \, \text{diam}\, (N) + \rho + r_{0}, 4 \, \text{diam}\, (N) + \rho + r_{0} + b_{\varepsilon} - 1]
    \end{cases}
\end{equation*}

\bigskip 

This path satisfies the following; $\gamma_{\varepsilon} (-1 - a_{\varepsilon}) = a_{\varepsilon}$, $\gamma_{\varepsilon} (4 \, \text{diam}\, (N) + r_{0} + \rho + b_{\varepsilon} - 1) = b_{\varepsilon}$, and 
\begin{equation*}
    \gamma_{\varepsilon} (t) < 2 \sigma \mathcal{H}^{n} (M) - \sigma \lambda \mu_{g} (E) + \sigma \lambda \mu_{g} (N \setminus E) - \frac{\sigma \mathcal{H}^{n} (A_{l})}{2 (2^{n} - 1)} + \tau,
\end{equation*}
for all $t$ in $[-1 - a_{\varepsilon}, 4 \, \text{diam}\, (N) + r_{0} + \rho + b_{\varepsilon} - 1]$.
This contradicts the min-max construction of $M$, implying that $M$ must be embedded, and therefore completing Theorem \ref{thm: Main Theorem}.

\section{Morse Index}\label{sec: index}

Recall the functional defined on Caccioppoli sets $\Omega \subset N$, 
\begin{equation*}
    F_{\lambda} (\Omega) = \mathcal{H}^{n} (\partial^{*} \Omega) - \lambda \mu_{g} (\Omega).
\end{equation*}
For a $C^{2}$ vector field $X$, we may take variations in direction $X$ by considering its flow $\{ \Phi_{t} \}$.
We define the first variation of $F_{\lambda}$ by, 
\begin{equation}\label{eqn: first variation of F lambda}
    \delta F_{\lambda} (\Omega) (X) = \frac{d}{dt} F_{\lambda} (\Phi_{t} (\Omega))_{| t = 0}.
\end{equation}
and the second variation by, 
\begin{equation}\label{eqn: second variation of F lambda}
    \delta^{2} F_{\lambda} (\Omega) (X) = \frac{d^{2}}{dt^{2}} F_{\lambda} (\Phi_{t} (\Omega))_{| t = 0}.
\end{equation}

\bigskip 

We have that $\delta F_{\lambda} (E) (X) = 0$, for all $C^{1}$ vector fields $X$.
Note that we require $M$ to be embedded and orientable for the following to be well-defined.
Consider the class of vector fields $X \in C^{2}_{c} (N \setminus (\overline{M} \setminus M))$, such that $X_{|M} = \varphi \nu$, where $\varphi \in C_{c}^{2} (M)$.
By \cite[Proposition 2.5]{doCarmo-Barbosa-CMC},
\begin{equation}
    \delta^{2} F_{\lambda} (E) (X) = \int_{M} |\nabla^{M} \varphi|^{2} - (|A_{M}|^{2} + \text{Ric} (\nu, \nu)) \varphi^{2} \, d \mathcal{H}^{n}.
\end{equation}
We extend the expression on the right-hand side to all functions in $W^{1,2}_{0} (M)$, and define the following quadratic form, 
\begin{equation*}
    B_{M} (\varphi, \varphi) \coloneqq \int_{M} |\nabla^{M} \varphi|^{2} - (|A_{M}|^{2} + \text{Ric} (\nu, \nu)) \varphi^{2} \, d \mathcal{H}^{n}, \hspace{1cm} \varphi \in W_{0}^{1,2} (M).
\end{equation*}
After integrating by parts we obtain the second order elliptic operator on $M$,
\begin{equation*}
    L_{M} \coloneqq \Delta_{M} + |A_{M}|^{2} + \text{Ric} (\nu, \nu).
\end{equation*}

\bigskip 

We restrict ourselves to a set $W \subset \subset N \setminus (\overline{M} \setminus M)$, to avoid our curvature term $|A_{M}|$, from potentially blowing up. 
A value $\kappa = \kappa (W) \in \mathbb{R}$ is said to be an eigenvalue of $L_{M}$ in $W$, if there exists a $\varphi \in W_{0}^{1,2} (W \cap M)$ such that 
\begin{equation*}
    L_{M} \varphi + \kappa \varphi = 0.
\end{equation*}
By standard elliptic theory, see \cite{Gilbarg-Trudinger}, the spectrum of $L_{M}$ in $W \cap M$,
\begin{equation*}
    \kappa_{1} (W) \leq \kappa_{2} (W) \leq \cdots \rightarrow + \infty,
\end{equation*}
is discrete and bounded from below. 
We then define the index of $M$ in $W$ by, 
\begin{equation*}
    \text{ind}_{W} (M) = |\{ p \colon \kappa_{p} (W) < 0 \}|,
\end{equation*}
or equivalently, it is the maximum dimension of a linear subspace of $W_{0}^{1,2} (W \cap M)$ on which $B_{M}$ is negative definite. 
If $\text{ind}_{W} (\text{ind} \, M) = 0$, then we say that $M$ is stable in $W$ and $\kappa_{p} (W) \geq 0$, for all $p$ in $\mathbb{N}$, and 
\begin{equation*}
    B_{M} (\varphi, \varphi) \geq 0, \, for \, all \, \varphi \in W_{0}^{1,2} (W \cap M).
\end{equation*}
We define, 
\begin{equation*}
    \text{ind} (M) = \sup_{W \subset \subset N \setminus (\overline{M} \setminus M)} (\text{ind}_{W} (M)).
\end{equation*}

\bigskip 

As $M$ is embedded, and our sequence of critical points $\{u_{i}\}$ from Section \ref{subsec: Allen-Cahn and CMC preliminaries} has $\text{ind} \, u_{i} \leq 1$, by \cite[Theorem 1a.]{Mantoulidis_2022}, we have that $\text{ind} M \leq 1$. 

\bigskip 

\begin{rem}
    As $M$ is two-sided and embedded, and the inhomogeneous term is a constant, we may also apply the ideas and arguments of \cite{Fritz-index-paper} verbatim to conclude that $\text{ind} M \leq 1$.
\end{rem}

\begin{claim}\label{claim: lower bound on index of M}
    $\text{ind} \, M  = 1$.
\end{claim}

\begin{proof}
    We only need to show a lower bound, which follows from the Ricci positivity on $N$. 
    We construct an appropriate function on $M$, using a similar argument to \cite[Lemma 5.1]{bellettini2020multiplicity1}.

    \bigskip 

    We wish to prove that we can find a set $W \subset \subset N \setminus (\overline{M} \setminus M)$, and a function $\varphi$ in $W_{0}^{1,2} (M \cap W)$ such that, 
    \begin{equation*}
        B_{M} (\varphi, \varphi) < 0.
    \end{equation*}

    \bigskip 

    By the Ricci positivity of $N$, for any $W \subset \subset N \setminus (\overline{M} \setminus M)$, and $\varphi$ in $W_{0}^{1,2} (M \cap W)$, we have 
    \begin{equation*}
        B_{M} (\varphi, \varphi) \leq \int_{M} |\nabla^{M} \varphi|^{2} - |A_{M}|^{2} \varphi^{2} \, d \mathcal{H}^{n}. 
    \end{equation*}

    \bigskip 

    If $\overline{M} \setminus M = \emptyset$, we set $W = N$, and $\varphi = 1$,
    \begin{equation*}
        B_{M} (\varphi, \varphi) \leq - \int_{M} |A_{M}|^{2} \, d \mathcal{H}^{n} < 0. 
    \end{equation*} 

    \bigskip 

    For $\overline{M} \setminus M \not= \emptyset$, we first we note that we must have $n \geq 7$, and $\mathcal{H}^{n-1} (\overline{M} \setminus M) = 0$. 
    Therefore, the 2-capacity of $\overline{M} \setminus M$ is 0, \cite[Section 4.7.2, Theorem 3]{EG1991measure}, implying that for all $\delta > 0$, there exists a function $f_{\delta}$ such that, 
    \begin{equation*}
        \begin{cases}
            f_{\delta} \in C_{c}^{\infty} (N \setminus (\overline{M} \setminus M)), \\
            f_{\delta} (y) \in [0,1], \, y \in N, \\
            \int_{N} | \nabla f_{\delta}|^{2} d \mu_{g} < \delta, \\
            \mu_{g} (\{ f_{\delta} = 1\}) > \mu_{g} (N) - \delta.
        \end{cases}
    \end{equation*}
    Furthermore, as $|\overline{M}|$ is a multiplicity 1 integral varifold with uniformly bounded generalised mean curvature, we have a monotonicity formula \cite[Corollary 17.8]{LSimonGMTNotes}. 
    The existence of such a monotonicity formula implies Euclidean volume growth about each point in $\overline{M}$.
    Therefore, there exists a constant $C_{8} = C_{8} (N, M, g)$, such that, by the construction of $f_{\delta}$ as in \cite[Section 4.7.2, Theorem 3]{EG1991measure},
    \begin{equation*}
        \int_{M} |\nabla^{M} f_{\delta}|^2 \, d \mathcal{H}^{n} \leq C_{8} \delta.
    \end{equation*} 
    Taking $W_{\delta} = \text{supp} \, f_{\delta} \subset \subset N \setminus (\overline{M} \setminus M)$, we have that $(f_{\delta})_{|M} \in W^{1,2}_{0} (M \cap W_{\delta})$, and, 
    \begin{equation*}
        B_{M} (f_{\delta}, f_{\delta}) \leq C \delta - n^{-2} \lambda^{2} \mathcal{H}^{n} (\{ f_{\delta} = 1 \} \cap M).
    \end{equation*}
    We have that as we send $\delta \rightarrow 0$, $\mathcal{H}^{n} (\{ f_{\delta} = 1 \} \cap M) \rightarrow \mathcal{H}^{n} (M)$.
    Therefore, for small enough $\delta > 0$, we have that $B_{M} (f_{\delta}, f_{\delta}) < 0$.
    This implies that $\text{ind} M \geq 1$.
\end{proof}

The fact that $M$ is connected immediately follows from this, as on each connected component we could construct a function as in Claim \ref{claim: lower bound on index of M}. 
Therefore, each connected component adds at least $1$ to the index.

\end{document}